\documentclass[a4paper,10pt]{amsart}
\usepackage[utf8]{inputenc}
\usepackage[T1]{fontenc}
\usepackage[english]{babel}
\usepackage{amsmath,amssymb,amsthm}
\usepackage{bbm}

\usepackage[pdftex]{color}

\usepackage[pdftex]{color}
\usepackage[bookmarks=true,hyperindex,pdftex,colorlinks, citecolor=MidnightBlue,linkcolor=MidnightBlue, urlcolor=MidnightBlue
]{hyperref}
\usepackage[dvipsnames]{xcolor}

\usepackage{tikz}
\usetikzlibrary{intersections,calc}

\usepackage{pgfplots}
\usepgfplotslibrary{fillbetween}
\pgfplotsset{compat=1.12}

\let\oldtocsection=\tocsection
\let\oldtocsubsection=\tocsubsection
\renewcommand{\tocsection}[2]{\hspace{0em}\oldtocsection{#1}{#2}}
\renewcommand{\tocsubsection}[2]{\hspace{1em}\oldtocsubsection{#1}{#2}}

\theoremstyle{plain}
\newtheorem{theorem}{Theorem}[section]
\newtheorem{proposition}[theorem]{Proposition}
\newtheorem{corollary}[theorem]{Corollary}
\newtheorem{lemma}[theorem]{Lemma}

\theoremstyle{definition}
\newtheorem{definition}[theorem]{Definition}

\newtheorem{example}[theorem]{Example}
\newtheorem{remark}[theorem]{Remark}

\newcommand{\conv}{\operatorname{conv}}
\newcommand{\aconv}{\operatorname{aconv}}
\newcommand{\Lip}{\operatorname{Lip}}
\newcommand{\Id}{\mathrm{Id}}

\newcommand{\eps}{\varepsilon}

\newcommand{\re}{\operatorname{Re}}

\newcommand{\vol}{\operatorname{vol}}

\newcommand{\spn}{\operatorname{span}}

\newcommand{\Face}{\operatorname{F}}

\newcommand{\eins}{\mathbbm{1}}

\newcommand{\FF}{\mathcal{F}}

\newcommand{\N}{\mathbb{N}}
\newcommand{\R}{\mathbb{R}}
\newcommand{\C}{\mathbb{C}}
\newcommand{\T}{\mathbb{T}}
\newcommand{\K}{\mathbb{K}}

\renewcommand{\geq}{\geqslant}
\renewcommand{\leq}{\leqslant}

\DeclareMathOperator{\e}{e}

\numberwithin{equation}{section}

\begin{document}

\title{On the numerical index with respect to an operator}

\author[Kadets, Mart\'{\i}n, Mer\'{\i}, P\'{e}rez, Quero]{Vladimir Kadets, Miguel Mart\'{\i}n, Javier Mer\'{\i}, Antonio P\'{e}rez \and Alicia Quero}

\address[Kadets]{School of Mathematics and Computer Sciences \\
V. N. Karazin Kharkiv National University \\ pl.~Svobody~4 \\
61022~Kharkiv \\ Ukraine}

\email{v.kateds@karazin.ua}

\address[Mart\'{\i}n, Mer\'{\i}, Quero]{Universidad de Granada \\ Facultad de Ciencias \\
Departamento de An\'{a}lisis Matem\'{a}tico \\ E-18071 Granada \\
Spain}
\email{mmartins@ugr.es, jmeri@ugr.es, aliciaquero@ugr.es}

\address[P\'{e}rez]{Instituto de Ciencias Matem\'{a}ticas (CSIC-UAM-UC3M-UCM) \\ Campus Cantoblanco \\ C/ Nicol\'{a}s Cabrera, 13--15 \\ E-28049 Madrid \\ Spain }
 \email{antonio.perez@icmat.es}

\thispagestyle{plain}


\date{May 27th, 2019}

\thanks{The research of the first author is done in frames of Ukrainian Ministry of Science and Education Research Program 0118U002036, and it was partially done during his stay in the University of Granada which was supported by the project MTM2015-65020-P (MINECO/FEDER, UE). Research of second, third, and fifth authors is supported by projects MTM2015-65020-P (MINECO/FEDER, UE),  PGC2018-093794-B-I00 (MCIU/AEI/FEDER, UE), and FQM-185 (Junta de Andaluc\'{\i}a/FEDER, UE). The fourth author acknowledges financial support from the Spanish Ministry of Economy and Competitiveness, through the ``Severo Ochoa Programme for Centres of Excellence in R\&D'' (SEV-2015-0554)}

\begin{abstract}
The aim of this paper is to study the numerical index with respect to an operator between Banach spaces. Given Banach spaces $X$ and $Y$, and a norm-one operator $G\in \mathcal{L}(X,Y)$ (the space of all bounded linear operator from $X$ to $Y$), the numerical index with respect to $G$, $n_G(X,Y)$, is the greatest constant $k\geq 0$ such that
$$
k\|T\|\leq \inf_{\delta>0} \sup\bigl\{|y^\ast(Tx)|\colon y^\ast\in Y^\ast,\,x\in X,\,\|y^\ast\|=\|x\|=1,\,\re y^\ast(Gx)>1-\delta\bigr\}
$$
for every $T\in \mathcal{L}(X,Y)$. Equivalently, $n_G(X,Y)$ is the greatest constant $k\geq 0$ such that
$$
\max_{|w|=1}\|G+wT\|\geq 1 + k \|T\|
$$
for all $T\in \mathcal{L}(X,Y)$. Here, we first provide some tools to study the numerical index with respect to $G$. Next, we present some results on the set $\mathcal{N}(\mathcal{L}(X,Y))$ of the values of the numerical indices with respect to all norm-one operators on $\mathcal{L}(X,Y)$. For instance, we show that $\mathcal{N}(\mathcal{L}(X,Y))=\{0\}$ when $X$ or $Y$ is a real Hilbert space of dimension greater than one and also when $X$ or $Y$ is the space of bounded or compact operators on an infinite-dimensional real Hilbert space. In the real case, we show that for $1<p<\infty$,
$$
\mathcal{N}(\mathcal{L}(X,\ell_p))\subseteq [0,M_p] \quad \text{ and } \quad \mathcal{N}(\mathcal{L}(\ell_p,Y))\subseteq [0,M_p]
$$
for all real Banach spaces $X$ and $Y$, where $M_p=\sup_{t\in[0,1]}\frac{|t^{p-1}-t|}{1+t^p}$. For complex Hilbert spaces $H_1$, $H_2$ of dimension greater than one, we show that $\mathcal{N}(\mathcal{L}(H_1,H_2))\subseteq \{0,1/2\}$ and the value $1/2$ is taken if and only if $H_1$ and $H_2$ are isometrically isomorphic. Besides, $\mathcal{N}(\mathcal{L}(X,H))\subseteq [0,1/2]$ and $\mathcal{N}(\mathcal{L}(H,Y))\subseteq [0,1/2]$ when $H$ is a complex infinite-dimensional Hilbert space and $X$ and $Y$ are arbitrary complex Banach spaces. We also show that $\mathcal{N}(\mathcal{L}(L_1(\mu_1),L_1(\mu_2)))\subseteq \{0,1\}$ and $\mathcal{N}(\mathcal{L}(L_\infty(\mu_1),L_\infty(\mu_2)))\subseteq \{0,1\}$ for arbitrary $\sigma$-finite measures $\mu_1$ and $\mu_2$, in both the real and the complex cases.
Also, we show that the Lipschitz numerical range of Lipschitz maps from a Banach space to itself can be viewed as the numerical range of convenient bounded linear operators with respect to a bounded linear operator. Further, we provide some results which show the behaviour of the value of the numerical index when we apply some Banach space operations, as constructing diagonal operators between $c_0$-, $\ell_1$-, or $\ell_\infty$-sums of Banach spaces, composition operators on some vector-valued function spaces, taking the adjoint to an operator, and composition of operators.
\end{abstract}

\maketitle

\thispagestyle{plain}

\begin{center}
\begin{minipage}[c]{0,9\textwidth}\small
\tableofcontents
\end{minipage}
\end{center}

\section{Introduction}
The study of isometric properties of the space $\mathcal{L}(X,Y)$ of all bounded linear operators between two Banach spaces $X$ and $Y$ and their impact on the domain and range spaces is a traditional subject of Banach space theory, and it remains to be an active area of research. For instance, in the second part of twentieth century there were a number of results  \cite{Blu-Lin-Phe,ExtremeLH, Mor-Phe, Sharir72, Sharir73, Sharir76, Sharir77} on the structure of extreme points of the unit ball of $\mathcal{L}(X,Y)$ (sometimes known as extreme operators or extreme contractions), but the subject attracts researchers till now, see for instance \cite{CabreraSerrano-Mena-Complu2017,CabMena2017, MalSainDeb, Navarro-Navarro,SainRayPaul} and references therein. When $X=Y$, the space $\mathcal{L}(X):=\mathcal{L}(X,X)$ is a Banach algebra with unit $\Id$ (or $\Id_X$ if it is necessary to mention), and there are many deep results in this case (see, for instance, the classical references \cite{Palmer,RussoDye}). The starting point of all these results is a celebrated result of 1955 by Bohnenblust and Karlin \cite{Bohn-Karlin} which related the geometric and the algebraic properties of the unit. To state their result, they introduce and study a numerical range of elements of a unital algebra which generalized the classical Toeplitz numerical range of operators on Hilbert spaces from 1918. Let us state here an extension of this numerical range, which implicitly appeared in Bohnenblust-Karlin paper, and which was introduced in the 1985 paper \cite{MarMenaPayaRod1985}. We refer the reader to the classical books \cite{B-D1,B-D2} by Bonsall and Duncan, and to sections 2.1 and 2.9 of the recent book \cite{Cabrera-Rodriguez} for more information and background. Given a Banach space $Z$, we write $B_Z$ and $S_Z$ to denote the closed unit ball and the unit sphere of $Z$, respectively. If $u\in Z$ is a norm-one element, the \emph{(abstract) numerical range} of $z\in Z$ with respect to $(Z,u)$ is given by
$$
V(Z,u,z):=\{\phi(z)\colon \phi\in \Face(B_{Z^\ast},u)\},
$$
where $Z^\ast$ denotes the topological dual of $Z$ and
$$
\Face(B_{Z^\ast},u):=\{\phi\in S_{Z^\ast}\colon \phi(u)=1\}
$$
is the \emph{face} of $B_{Z^\ast}$ generated by $u\in S_Z$ (also known as the set of \emph{states} of $Z$ relative to $u$). Let us mention that when $Z=A$ is a unital Banach algebra and $u$ is the unit of $A$, then $V(A,u,a)$ is the \emph{algebra numerical range} of the element $a\in A$. The well-known formula
\begin{equation*}
\sup \re V(Z,u,z)=\lim_{\alpha\to 0^+} \frac{\|u+\alpha z\|-1}{\alpha}
\end{equation*}
(see Lemma \ref{lemma:derivadaani}) connects the geometry of the space $Z$ around $u$ with the numerical range with respect to $(Z,u)$.
The \emph{numerical radius} of $z\in Z$ with respect to $(Z,u)$ is
$$
v(Z,u,z):=\sup\{|\lambda|\colon \lambda\in V(Z,u,z)\},
$$
which is obviously a seminorm on $Z$ satisfying that $v(Z,u,z)\leq \|z\|$ for every $z\in Z$. To determine if the numerical radius is actually an equivalent norm on $Z$, it is used the (\emph{abstract}) \emph{numerical index} of $(Z,u)$ (or the \emph{numerical index of $Z$ with respect to $u$}), namely the constant
$$
n(Z,u):=\inf\{v(Z,u,z)\colon z\in S_Z\}=\max\{k\geq 0\colon k\|z\|\leq v(Z,u,z)\ \forall z\in Z\}.
$$
It is clear that $0\leq n(Z,u)\leq 1$ and that $n(Z,u)>0$ if and only if $v(Z,u,\cdot)$ is an equivalent norm on $Z$ (and this is equivalent to the fact that $u$ is a geometrically unitary element of $B_Z$, see the beginning of section \ref{sect:abstractnumericalindex}). When $n(Z,u)=0$, it is possible that $v(Z,u,\cdot)$ is not a norm, or that $v(Z,u,\cdot)$ is a non-equivalent norm on $Z$ (and in this case, $u$ is a vertex of $B_Z$ which is not a geometrically unitary element, see also the beginning of section \ref{sect:abstractnumericalindex}). The value $n(Z,u)=1$ means that the numerical radius with respect to $(Z,u)$ coincides with the given norm of $Z$ (and in this case, we say that $u$ is a \emph{spear} element of $Z$, see Proposition \ref{Prop:CharacterizationsNumericalIndexVector} and the paragraph after it for some equivalent formulations). With this language in mind, the announced result of Bohnenblust and Karlin states that unitary elements of a unital complex algebra $A$ (a purely algebraic concept) are geometrically unitary elements of $A$ (a purely geometric concept), actually one has that $n(A,u)\geq 1/\e$ if $u$ is a unitary element of the complex Banach algebra $A$, see \cite[Corollary 2.1.21]{Cabrera-Rodriguez}. This is no longer true in the real case as, for instance, the identity is not even a vertex of $\mathcal{L}(H)$ when $H$ is any Hilbert space of dimension greater than one. Nevertheless, numerical range arguments show that the unit of a unital real Banach algebra is a strongly extreme point, see \cite[Corollary 2.1.42]{Cabrera-Rodriguez}  and \cite{KaKaMaVi} for a quantitative version. For (complex) $C^*$-algebras, the concepts of unitary element and geometrically unitary element coincide, see \cite[Theorem 2.1.27]{Cabrera-Rodriguez} for the details. Let us also comment that the study  of the algebra numerical range was crucial to state very important results in the theory of Banach algebras as Vidav's characterization of $C^*$-algebras, see \cite{B-D1} or \cite{Cabrera-Rodriguez}. More recently, geometric characterizations of algebraic properties of elements of $C^*$-algebras have been given by Akeman and Weaver \cite{AkemanWeaver}, some of which can be expressed in terms of the numerical ranges, see \cite{Rodriguez-JMAA-unitaries}. Let us observe that geometrically unitary elements (and even vertices) of the unit ball of a Banach space are extreme points of the unit ball (see Lemma \ref{Lemm:unotsmoothnotextreme}, for instance) so, when non-zero, the abstract numerical index measures ``how extreme'' is a point of the unit ball of a Banach space. Finally, let us recall that the concept of numerical range (and so the ones of numerical radius and numerical index) depends on the base field, as for a complex Banach space $Z$ and a norm-one element $u\in Z$, $V(Z_\R,u,z)=\re V(Z,u,z)$, where $Z_\R$ is the real space underlying the space $Z$ and $\re$ represents the real part function.

Let us now return to our aim of studying the geometry of $\mathcal{L}(X,Y)$ around a norm-one operator $G$. For this to be done, we introduce the numerical range with respect to $G$. If $X$ and $Y$ are Banach spaces and $G\in \mathcal{L}(X,Y)$ is a norm-one operator, we consider the \emph{numerical range} of $T\in \mathcal{L}(X,Y)$ \emph{with respect to $G$} which is the set
$$
V(\mathcal{L}(X,Y),G,T)=\{\phi(T)\colon \phi\in \mathcal{L}(X,Y)^\ast,\,\phi(G)=1\}.
$$
Analogously, we may consider the corresponding \emph{numerical radius with respect to $G$}:
$$
v(\mathcal{L}(X,Y),G,T)=\sup\{|\lambda|\colon \lambda \in V(\mathcal{L}(X,Y),G,T)\} \qquad (T\in \mathcal{L}(X,Y)),
$$
and the \emph{numerical index} of $(\mathcal{L}(X,Y),G)$ (or the numerical index of $\mathcal{L}(X,Y)$ with respect to $G$):
$$
n_G(X,Y):=n(\mathcal{L}(X,Y),G)=\inf\{v(\mathcal{L}(X,Y),G,T)\colon T\in \mathcal{L}(X,Y),\,\|T\|=1\}.
$$
This will be the central concept of study in this paper. Let us comment that $n_G(X,Y)$ is the greatest constant $k\geq 0$ such that
$$
\max_{|w|=1}\|G + wT\|\geq 1 + k \|T\|
$$
for every $T\in \mathcal{L}(X,Y)$, see Proposition \ref{Prop:CharacterizationsNumericalIndexOperators}. The case of $k=1$ in the inequality above gives the concept of spear operator, introduced in \cite{Ardalani} and deeply studied in \cite{SpearsBook}.

Usually, when one deals with the geometry of spaces of operators, it is convenient to have tools which allow to describe this geometry in terms of the geometry of the domain and range spaces, allowing us to work on these spaces and not on the whole space of operators and, even more, on its wild dual space. In the case of the numerical range of operators on a Banach space (with respect to the identity operator), this tool is the ``spatial'' version of the numerical range. If $X$ is a Banach space and $T\in \mathcal{L}(X)$, the \emph{spatial numerical range} of $T$ was introduced by Bauer (and in a somehow equivalent reformulation by Lumer) in the 1960's (see \cite{B-D1} for instance) as the set
\begin{equation}\label{eq:spatialnumericalrange}
W(T):=\{x^\ast(Tx)\colon x^\ast\in S_{X^\ast},\,x\in S_X,\,x^\ast(x)=1\},
\end{equation}
which is the direct extension of Toeplitz's numerical range of operators on Hilbert spaces. It is immediate that $W(T)\subseteq V(\mathcal{L}(X),\Id,T)$ and, actually, one has
$$
\overline{\conv}\bigl(W(T)\bigr)=V(\mathcal{L}(X),\Id,T)
$$
for every $T\in \mathcal{L}(X)$ (see \cite[Proposition 2.1.31]{Cabrera-Rodriguez}, for instance). From here, it is clear that the \emph{spatial numerical radius} $v(T)$ of an operator $T\in \mathcal{L}(X)$ coincides with the numerical radius with respect to $\Id$, that is,
$$
v(T):=\sup\{|\lambda|\colon \lambda\in W(T)\}=v(\mathcal{L}(X),\Id,T).
$$
Therefore, the same happens with the corresponding numerical index:
$$
n(X):=\inf\{v(T)\colon T\in \mathcal{L}(X),\, \|T\|=1\}=n(\mathcal{L}(X),\Id).
$$
With this tool it has been possible to construct an example of a Banach space $X$ such that the identity operator is a vertex but not a geometrically unitary element, see \cite[Proposition 2.1.39]{Cabrera-Rodriguez} for instance. For a detailed study of the Banach space numerical index, we refer the reader to the expository paper \cite{KaMaPa} and to subsection 1.1 of the very recent paper \cite{secondni}.

When dealing with a general operator $G\in \mathcal{L}(X,Y)$, the possibility of getting a spatial numerical range with respect to $G$ analogous to the formula \eqref{eq:spatialnumericalrange} does not work. Indeed, for the set
$$
\{(x,y^\ast)\colon x\in S_X,\, y^\ast\in S_{Y^\ast},\,y^\ast(Gx)=1\}
$$
to be non-empty, we need the operator $G$ to attain its norm; but even in the case of $G$ being an inclusion operator, the above set is not always representative (see \cite{MarMerPay}). Nevertheless, there is an ``approximate spatial'' numerical range with respect to an operator recently introduced by Ardalani \cite{Ardalani} which does the job. Given two Banach spaces $X$ and $Y$ and a norm-one operator $G\in \mathcal{L}(X,Y)$, the \emph{approximate spatial numerical range} of $T\in \mathcal{L}(X,Y)$ \emph{with respect to} $G$ is the set
$$
V_G(T):=\bigcap_{\delta>0}\overline{\{y^\ast(Tx)\colon y^\ast\in S_{Y^\ast},\,x\in S_X,\, \re y^\ast(Gx)>1-\delta \}}.
$$
It was shown in \cite{Ardalani}, using the Bishop-Phelps-Bollob\'{a}s theorem, that $V_{\Id}(T)=\overline{W(T)}$ for every $T\in \mathcal{L}(X)$ and every Banach space $X$, so both numerical ranges produce the same associated numerical radii. Moreover, it is shown in \cite[Theorem~2.1]{Mar-numrange-JMAA2016} that
\begin{equation}\label{eq:Ardalani-MartinJMAA2016}
\conv\bigl(V_G(T)\bigr)=V(\mathcal{L}(X,Y),G,T)
\end{equation}
for all Banach spaces $X$, $Y$ and all operators $G,T\in \mathcal{L}(X,Y)$. It then follows that
$$
v_G(T):=\inf_{\delta>0}\sup\{|y^\ast(Tx)|\colon y^\ast\in S_{Y^\ast},\,x\in S_X,\, \re y^\ast(Gx)>1-\delta\}=v(\mathcal{L}(X,Y),G,T),
$$
and that
$$
n_G(X,Y)=\inf\{v_G(T)\colon T\in \mathcal{L}(X,Y),\,\|T\|=1\}=n(\mathcal{L}(X,Y),G).
$$
This provides a ``spatial'' way to deal with the numerical radius and the numerical index with respect to an arbitrary operator which is specially interesting when we work in concrete Banach spaces and when we study the behaviour of these concepts with respect to Banach space operations on the domain and range spaces.

The aim of this paper is to present a number of results on the numerical indices with respect to operators. Let us detail the content of the paper. Fist, we finish this introduction with a short subsection containing the needed terminology and notation. Next, we provide in section \ref{sect:abstractnumericalindex} some basic results on abstract numerical index. Some of the results were previously known, but some others are new. Among the new ones, we may stress the fact that the set $\{u\in S_Z \colon n(Z,u)>0\}$ is countable (i.e.\ finite or infinite and countable) when $Z$ is a finite-dimensional real space, and we provide with some estimations on the sum of the values $n(Z,u)$ moving $u\in S_Z$. On the other hand, for every subset $A\subseteq [0,1]$ containing $0$, we show that there is a (real or complex) Banach space $Z$ such that $\{n(Z,u)\colon u\in S_Z\}=A$. Besides, an extension of the formula \eqref{eq:Ardalani-MartinJMAA2016} is given which provides some useful ways to calculate numerical radii with respect to operators. Next, we particularize these results to numerical indices with respect to operators and also give some more important tools in section \ref{sect:tools_num_index_operators}. Namely, we show that the numerical index with respect to an operator always dominates the numerical index with respect to its adjoint, we calculate the value of the numerical index with respect to a rank-one operator and we show some estimations of the numerical index with respect to an operator in terms of the numerical radii of operators on the domain space or on the range space. In section \ref{sect:setofvalues} we provide results on the set of values of the numerical indices with respect to all norm-one operators between two fixed Banach spaces, that is, on the set
$$
\mathcal{N}(\mathcal{L}(X,Y)):=\bigl\{n_G(X,Y)\colon G\in \mathcal{L}(X,Y),\, \|G\|=1 \bigr\}
$$
for given Banach spaces $X$ and $Y$ (this notation is coherent with the one that we will introduce at the beginning of subsection \ref{subsection:fin-dim-spaces} for the abstract numerical index). We show that $0\in \mathcal{N}(\mathcal{L}(X,Y))$ unless both $X$ and $Y$ are one-dimensional, and that the set $\mathcal{N}(\mathcal{L}(X,Y))$ is countable if $X$ and $Y$ are finite-dimensional real spaces. Besides, for a real Hilbert space $H$ with $\dim(H)\geq 2$ one has
$$
\mathcal{N}(\mathcal{L}(X,H))=\mathcal{N}(\mathcal{L}(H,Y))=\{0\}
$$
for all Banach spaces $X$ and $Y$. We also show that the role of the space $H$ can also be played by some non-Hilbertian real Banach spaces as $\mathcal{L}(H)$  when $H$ is an infinite-dimensional real Hilbert space.  Estimations of the numerical indices with respect to operators which leave from or arrive to the real spaces $\ell_p$ are also given: for $1<p<\infty$,
$$
\mathcal{N}(\mathcal{L}(X,\ell_p))\subseteq [0,M_p] \quad \text{ and } \quad \mathcal{N}(\mathcal{L}(\ell_p,Y))\subseteq [0,M_p]
$$
for all real Banach spaces $X$ and $Y$, where $M_p=\sup_{t\in[0,1]}\frac{|t^{p-1}-t|}{1+t^p}$. For complex Hilbert spaces $H_1$, $H_2$ of dimension greater than one, we show that $\mathcal{N}(\mathcal{L}(H_1,H_2))\subseteq \{0,1/2\}$ and the value $1/2$ is taken if and only if $H_1$ and $H_2$ are isometrically isomorphic. Besides, $\mathcal{N}(\mathcal{L}(X,H))\subseteq [0,1/2]$ and $\mathcal{N}(\mathcal{L}(H,Y))\subseteq [0,1/2]$ when $H$ is a complex infinite-dimensional Hilbert space and $X$ and $Y$ are arbitrary complex Banach spaces. It is also shown that
$$
\mathcal{N}(\mathcal{L}(C(K_1),C(K_2)))=\{0,1\}
$$
for many families of Hausdorff topological compact spaces $K_1$ and $K_2$, both in the real and the complex cases. As a consequence, we show that
$$
\mathcal{N}(\mathcal{L}(L_\infty(\mu_1),L_\infty(\mu_2)))\subseteq \{0,1\} \quad \text{and} \quad \mathcal{N}(\mathcal{L}(L_1(\mu_1),L_1(\mu_2)))\subseteq \{0,1\}
$$
for all $\sigma$-finite positive measures $\mu_1$ and $\mu_2$.

In section \ref{sect:Lipschitz} we use the tools presented in section \ref{sect:tools_num_index_operators} to prove that the concept of Lipschitz numerical range introduced in \cite{Wang2012, Wang-Huang-Tan} for Lipschitz self-maps of a Banach space can be viewed as a particular case of numerical range with respect to a linear operator between two different Banach spaces.

Finally, we collect in section \ref{sect:stability} some results which show the behaviour of the value of the numerical index when we apply some Banach space operations. For instance, we show that the numerical index of a $c_0$-, $\ell_1$- or $\ell_\infty$-sum of Banach spaces with respect to a direct sum of norm-one operators in the corresponding spaces coincides with the infimum of the numerical indices of corresponding summands. As a consequence, we show that there are real and complex Banach spaces $X$ for which $\mathcal{N}\bigl(\mathcal{L}(X)\bigr) =[0,1]$. We also show that a composition operator between spaces of vector-valued continuous, integrable, or essentially bounded functions produces the same numerical index as the original operator. Next, we provide two conditions assuring that the numerical index with respect to an operator and the numerical index with respect to its adjoint coincides: that the range space is $L$-embedded and that the operator is rank-one. Finally, we discuss some results on the value of the numerical index with respect to a composition of two operators and then we provide with ways to extend the domain or the codomain of an operator retaining the numerical index. In particular, the results of the section allow us to solve Problem 9.14 of \cite{SpearsBook}.

\subsection{Notation and terminology}

By $\mathbb{K}$ we denote the scalar field ($\mathbb{R}$ or $\mathbb{C}$), and we use the standard notation $\T:=\{ \lambda \in \mathbb{K} \colon |\lambda|=1 \}$ for its unit sphere. We use the letters $X, Y, Z$ for Banach spaces over $\mathbb{K}$ and by subspace we always mean closed subspace. In some cases, we have to distinguish between the real and the complex case, but for most results this difference is insignificant. We write $J_X:X\longrightarrow X^{\ast\ast}$ to denote the natural isometric inclusion of $X$ into its bidual $X^{\ast\ast}$.

Let $\Gamma$ be a non-empty index set, and $\{X_\gamma \colon \gamma \in \Gamma\}$ be a collection of Banach spaces. We write
$$
\bigl[\bigoplus\nolimits_{\lambda\in \Lambda} X_\lambda\bigr]_{c_0},\qquad
\bigl[\bigoplus\nolimits_{\lambda\in \Lambda} X_\lambda\bigr]_{\ell_1},\qquad \bigl[\bigoplus\nolimits_{\lambda\in \Lambda} X_\lambda\bigr]_{\ell_\infty},
$$
to denote, respectively, the $c_0$-, $\ell_1$-, and $\ell_\infty$-sum of the family. If $E$ is $\R^n$ endowed with an absolute norm $|\cdot|_E$ and $X_1,\ldots,X_n$ are Banach spaces, we write $X=[X_1\oplus\dots\oplus X_n]_E$ to denote the product space $X_1\times \cdots \times X_n$ endowed with the norm
$$
\|(x_1,\ldots,x_n)\|=\bigl|(\|x_1\|,\ldots,\|x_n\|) \bigr|_E
$$
for all $x_i\in X_i$, $i=1,\ldots,n$.

Given $1\leq p \leq \infty$ and a non-empty set $\Gamma$, we write $\ell_p(\Gamma)$ to denote the $L_p$-space associated to the counting measure on $\Gamma$. For $n\in \N$, we just write $\ell_{p}^{n}$ to denote $\ell_p(\{1,\ldots,n\})$. Given a Banach space $X$, a compact Hausdorff topological space $K$, and a $\sigma$-finite measure space $(\Omega,\Sigma,\mu)$, we write $C(K,X)$, $L_1(\mu,X)$, and $L_\infty(\mu,X)$ to denote, respectively, the spaces of continuous functions from $K$ to $X$, Bochner-integrable (classes of) functions from $\Omega$ to $X$, and strongly measurable and essentially bounded (classes of) functions from $\Omega$ to $X$.

\section{Some old and new results on abstract numerical index}\label{sect:abstractnumericalindex}
Our aim here is to recall some basic facts about the abstract numerical range and to provide with some ones which, as far as we know, are new. We start by recalling some related definitions which were already mentioned in the introduction.

\begin{definition}
Let $Z$ be a Banach space and let $u\in S_Z$.
\begin{enumerate}
\item[(a)] We say that $u$ is a \emph{vertex} of $B_Z$ if $\Face(B_{Z^\ast},u)$ separates the points of $Z$ (i.e.\ for every $z\in Z\setminus \{0\}$, there is $\phi \in \Face(B_{Z^\ast},u)$ such that $\phi(z)\neq 0$). This is clearly equivalent to the fact that $v(Z,u,z)=0$ for $z\in Z$ implies that $z=0$.
\item[(b)] We say that $u$ is a \emph{geometrically unitary element} of $B_Z$ if the linear span of $\Face(B_{Z^\ast},u)$ is equal to the whole $Z^\ast$. It is known, see \cite[Theorem 2.1.17]{Cabrera-Rodriguez}, that $u$ is a geometrically unitary element if and only if $n(Z,u)>0$.
\end{enumerate}
\end{definition}

We refer the reader to the already cited book \cite{Cabrera-Rodriguez}, and to the papers \cite{BandJaRao,Godefroy-Indumathi,Godefroy_Rao,Rodriguez-JMAA-unitaries} for more information and background on these concepts.

\subsection{A few known elementary results}

First, we present some known results on abstract numerical index which we will use throughout the paper. They are elementary and come from many sources, but we have decided to refer all of them to the recent monograph \cite{Cabrera-Rodriguez} for convenience of the reader.

The first result allows to relate the numerical range to a directional derivative.

\begin{lemma}\label{lemma:derivadaani}
Let $Z$ be a Banach space and let $u\in S_Z$. Then
$$
\max \re V(Z,u,z)=\lim_{t\to 0^+} \frac{\|u+tz\|-1}{t}
$$
for every $z\in Z$. Therefore,
$$
v(Z,u,z)=\max_{\theta\in \T} \lim_{t\to 0^+} \frac{\|u+t\theta z\|-1}{t}= \lim_{t\to 0^+} \max_{\theta\in \T}\frac{\|u+t\,\theta z\|-1}{t}.
$$
\end{lemma}

The first part of the above lemma is folklore and can be found in \cite[Proposition 2.1.5]{Cabrera-Rodriguez}. The first equality for the numerical radius is an immediate consequence, and the second equality follows routinely from the compactness of $\T$. Indeed, let $(t_n)$ a sequence of positive scalars converging to $0$ and for each $n\in \N$, take $\theta_n\in \T$ such that
$$
\max_{\theta\in \T}\frac{\|u+t_n\,\theta z\|-1}{t_n}= \frac{\|u+t_n\,\theta_n\,z\|-1}{t_n}.
$$
Extract a subsequence $(\theta_{\sigma(n)})$ which is convergent to, say, $\theta_0\in \T$. Then
$$
\frac{\|u+t_{\sigma(n)}\theta_0 z\|-1}{t_{\sigma(n)}} \geq  \frac{\|u+t_{\sigma(n)}\theta_{\sigma(n)} z\|-1}{t_{\sigma(n)}} - |\theta_{\sigma(n)}-\theta_{0}|\|z\|.
$$
Finally,
\[
v(Z,u,z)\geq \lim_{n\to \infty} \frac{\|u+t_{\sigma(n)}\theta_0 z\|-1}{t_{\sigma(n)}} \geq \lim_{n\to \infty} \max_{\theta\in \T}\frac{\|u+t_{\sigma(n)}\,\theta z\|-1}{t_{\sigma(n)}}
.
\]

The next easy to prove  result relates the numerical index with respect to a point to the geometry at the point. Recall that a norm-one element $z$ of a Banach space $Z$ is said to be a \emph{strongly extreme point} of $B_Z$ if whenever $(x_n)$ and $(y_n)$ are sequences in $B_Z$ such that $\lim (x_n + y_n) = 2u$, then $\lim (x_n - y_n) = 0$. It is clear that strongly extreme points are extreme points, but the reciprocal result is not true, see \cite{KunenRosenthal} for instance.

\begin{lemma}\label{Lemm:unotsmoothnotextreme}
Let $Z$ be a Banach space and $u\in S_Z$.
\begin{enumerate}
  \item[(a)] If $u$ is a vertex of $B_Z$, then $u$ is an extreme point and, if moreover $\dim(Z)\geq 2$, then the norm of $Z$ is not smooth at $u$.
  \item[(b)] If $u$ is a geometrically unitary element of $B_Z$ \emph{(}i.e.\ $n(Z,u)>0$\emph{)}, then $u$ is a strongly extreme point of $B_Z$.
\end{enumerate}
\end{lemma}

The extreme point condition appears in \cite[Lemma 2.1.25]{Cabrera-Rodriguez}; if the norm of $Z$ is smooth at $u$, then $\Face(B_{Z^\ast},u)$ is a singleton, so either $\dim(Z)=1$ or $u$ cannot be a vertex. The result in (b) appears in \cite[Proposition 2.1.41]{Cabrera-Rodriguez}. Let us comment that there are vertices which are not strongly extreme points \cite[Example 2.1.43]{Cabrera-Rodriguez}.

The next result, which can be found in \cite[Corollary 2.1.2]{Cabrera-Rodriguez}, is elementary and very useful.

\begin{lemma}\label{Lemm:numericalIndexThroughMaps}
Let $\psi: Z_1 \longrightarrow Z_2$ be a linear operator between Banach spaces $Z_1$ and $Z_2$, let $u \in S_{Z_1}$ be such that  $\| \psi(u)\| = 1$. Then,
\begin{enumerate}
  \item[(a)] If $\|\psi\|=1$, then $v(Z_2, \psi(u), \psi(z)) \leq v(Z_1, u, z)$ for every $z \in Z_1$.
  \item[(b)] If $\psi$ is an isometric embedding, then $v(Z_2, \psi(u), \psi(z)) = v(Z_1, u, z)$ for every $z \in Z_1$; therefore, we have in this case that
      $n(Z_2,\psi(u))\leq n(Z_1,u)$.
\end{enumerate}
\end{lemma}

We next would like to present a pair of characterizations of the abstract numerical index.
\begin{proposition}\label{Prop:CharacterizationsNumericalIndexVector}
	Let $Z$ be a Banach space, $u\in S_Z$, and $0 < \lambda \leq 1$. Then, the following statements are equivalent:
	\begin{enumerate}
		\item[(i)] $n(Z,u) \geq \lambda$;
		\item[(ii$_\R$)]  In the real case, $\lambda B_{Z^\ast } \subseteq \conv{\bigl( \Face(B_{Z^\ast},u) \cup -\Face(B_{Z^\ast},u)\bigr)}$;
		\item[(ii$_\C$)]  In the complex case, given $\varepsilon > 0$,  $\theta_{1}, \ldots, \theta_{k} \in B_{\mathbb{C}}$ satisfying \mbox{$B_{\mathbb{C}} \subseteq (1 + \varepsilon) \conv{\{ \theta_{1}, \ldots, \theta_{k} \}}$}, it holds that
		\[
		\lambda  B_{Z^\ast } \subseteq (1+ \varepsilon) \, \conv{\left(\bigcup\nolimits_{j=1}^{k} \theta_{j} \Face(B_{Z^\ast},u)\right)} \, .
		\]
    \item[(iii)] $\displaystyle \max_{\theta\in \T}\|u + \theta z\|\geq 1 + \lambda \|z\|$ for every $z\in Z$.
	\end{enumerate}
\end{proposition}

The equivalence between (i) and (ii) is well known and can be found in  \cite[Theorem~2.1.17]{Cabrera-Rodriguez}, for instance. The implication (i)$\Rightarrow$(iii) is immediate from the Hahn-Banach theorem. The reciprocal result follows straightforwardly from the last equality in Lemma \ref{lemma:derivadaani}.

Let us comment that the strongest possibility in Proposition \ref{Prop:CharacterizationsNumericalIndexVector}, that is, $\lambda=1$, gives rise to the concept of spear vector introduced in \cite{SpearsBook}. A norm-one element $u$ of a Banach space $Z$ is a \emph{spear vector} if the equality $$\max_{\theta\in \T}\|u+\theta z\|=1+\|z\|$$ holds for every $z\in Z$. With the help of the previous proposition, this is equivalent to the fact that $n(Z,u)=1$. We refer the reader to chapter 2 of the book \cite{SpearsBook} for more information and background.

Finally, we present a result relating the numerical index of a Banach space with respect to a point to the numerical index of its bidual with respect to the same point which can be found in \cite[Theorem 2.1.17.v]{Cabrera-Rodriguez}.

\begin{lemma}\label{Lemm:bidual-abstract}
Let $Z$ be a Banach space and let $u\in S_Z$. Then, $n(Z^{\ast\ast},J_Z(u))=n(Z,u)$.
\end{lemma}

\subsection{On the set of values of the abstract numerical indices with respect to all unit vectors of a given space} \label{subsection:fin-dim-spaces}

For a given Banach space $Z$, denote
$$
\mathcal N(Z) := \{n(Z,u) \colon u \in S_Z \}.
$$
In this subsection we concentrate on the properties of $\mathcal N(Z)$ for various classes of Banach spaces $Z$.

Let us start with a general important observation.

\begin{proposition}\label{Prop0in:N(X)}
Let $Z$ be a Banach space with $\dim(Z) \geq 2$. Then, $0 \in \mathcal N(Z)$.	
\end{proposition}	

\begin{proof}
Let $Y$ be a two-dimensional subspace of $Z$. Then, there is a smooth point $u \in S_Y$ and we have $n(Y,u)=0$ by Lemma \ref{Lemm:unotsmoothnotextreme}.a. Now, Lemma \ref{Lemm:numericalIndexThroughMaps}.b gives that $n(Z,u) = 0$.
\end{proof}	

For many Banach spaces $Z$, zero is the only element of $\mathcal N(Z)$. Say, this happens for smooth spaces of dimension greater than one, a fact which follows immediately from the above proof. In section \ref{sect:setofvalues} the reader will find many examples of operator spaces $Z = \mathcal{L}(X_1, X_2)$ with the property that $\mathcal N(Z) = \{0\}$. On the other hand, for ``big bad'' spaces $Z$, the corresponding set $\mathcal N(Z)$ can be big. Moreover, it is possible to show that this set can be any subset of $[0, 1]$ that contains $0$.

\begin{proposition}\label{propN(X)=A}
For every subset $A$ of $[0, 1]$ with $0 \in A$, one can find a (real or complex) Banach space $Z$ with $\mathcal N(Z) = A$.	
\end{proposition}	

In order to demonstrate this result, we need a little bit of preparatory work.

\begin{example} \label{ex-dim2-two-points}
{\slshape For every $a \in [0, 1]$ there is a two-dimensional (real or complex) space $Z_a$ with $\mathcal N(Z_a) = \{0, a\}$.}

Indeed, for $r \in [0, 1]$ denote by $Z_{r}^\ast $ the two-dimensional space $\K^2$ equipped  with the norm
$$
\|(x_1, x_2)\| = \max\bigl\{ |x_1|, \sqrt{r |x_1|^2 + |x_2|^2}\bigr\}.
$$
Then, the intersections of $B_{Z_{r}^\ast}$ with the lines $\{x_1 = \theta\}$ for $\theta\in \T$ are the only non-trivial faces of $B_{Z_{r}^\ast}$. Therefore, in the predual space $Z_{r}$ the only elements $u$ of $S_{Z_{r}}$ with $n(Z_{r},u) \neq 0$ are $u = (\theta, 0)$ with $\theta\in \T$. As $Z_r$ has the same abstract numerical index with respect to all these elements, $\mathcal N(Z_{r})$ consists of two points: $0$ and some $h(r) \geq 0$. It is straightforward to show that $h(r)$ moves continuously from $1$ to $0$ as $r$ moves from $0$ to $1$ (as $Z_0 = \ell_{\infty}^{2}$ and $Z_1 = \ell_{2}^{2}$).
\end{example}

\begin{figure}[h]
\begin{tikzpicture}[thick, scale=1.8]

\definecolor{light-gray}{gray}{0.83}


\def\r{0.49};
\def\doml{1/0.8};
\def\domr{1/0.8};

\draw[draw=black, fill=light-gray, opacity=0.8]
plot [domain=-1:1] (\x,{sqrt(1-\r*\x*\x)})
-- plot [smooth,domain=1:-1] (\x,{-sqrt(1-\r*\x*\x)})
-- cycle;

\draw plot [domain=-\doml:\domr,samples=100,thin] (\x,{sqrt(1-\r*\x*\x)});

\draw plot [domain=-\doml:\domr,samples=100,thin] (\x,{-sqrt(1-\r*\x*\x)}) ;

\draw [->] (-1.5,0) -- (1.9,0);

\draw [->] (0,-1.5) -- (0,1.5);

\draw (0.25,1.12) node {$(0,1)$};
\draw (0,1) node {$\bullet$};

\draw (1.23,-0.15) node {$(1,0)$};
\draw (1,0) node {$\bullet$};

\draw (-1.2,1.1)  node {$y=\sqrt{1-rx^2}$};

\draw (-1.2,-1.1)  node  {$y=-\sqrt{1-rx^2}$};

\end{tikzpicture}
\label{figure:ballzstar}
\caption{The unit ball of $Z_r^\ast$}
\end{figure}
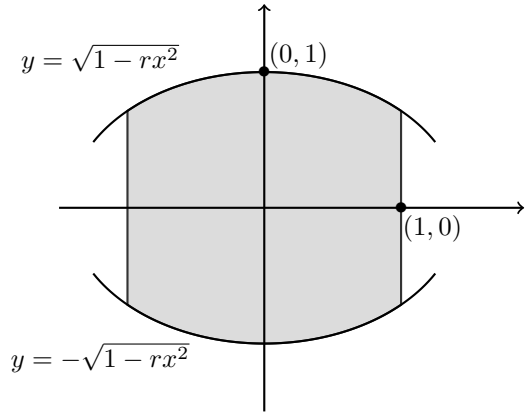

The next result may be known, but we include the easy proof as we have not found any reference to it in the literature.

\begin{lemma}\label{ell_1(Gamma{Xgamma}}
Let $\{Z_\gamma\colon \gamma\in \Gamma\}$ be a family of Banach spaces. Then,
$$
\mathcal{N}\left(\bigl[\bigoplus\nolimits_{\gamma\in \Gamma} Z_\gamma\bigr]_{\ell_1} \right)= \bigcup\nolimits_{\gamma \in \Gamma} \mathcal N(Z_\gamma).
$$
\end{lemma}

\begin{proof}
If a norm-one element $u = (u_\gamma)_{\gamma \in \Gamma} \in \bigl[\bigoplus\nolimits_{\gamma\in \Gamma} Z_\gamma\bigr]_{\ell_1}$ has more than one non-zero coordinate, then $n\left(\bigl[\bigoplus\nolimits_{\gamma\in \Gamma} Z_\gamma\bigr]_{\ell_1}, u\right) = 0$ as $u$ is then not an extreme point. In the case of $u$ having just one non-zero coordinate $u_\tau$, then we have that $n\left(\bigl[\bigoplus\nolimits_{\gamma\in \Gamma} Z_\gamma\bigr]_{\ell_1}, u\right) = n(X_\tau, u_\tau)$ routinely.
\end{proof}	

We are now ready to provide the pending proof.

\begin{proof}[Proof of Proposition \ref{propN(X)=A}]
For every $a \in A$ select a two-dimensional $Z_a$ such that $\mathcal N(Z_a) = \{0, a\}$ provided by Example \ref{ex-dim2-two-points} and then the desired example is $Z=
\bigl[\bigoplus\nolimits_{a\in A} Z_a\bigr]_{\ell_1}$ by Lemma \ref{ell_1(Gamma{Xgamma}}.
\end{proof}	

Our next goal is to find out which restrictions on $\mathcal N(Z)$ appear in the finite-dimensional case. We start by showing that, in this case, the corresponding $\mathcal N(Z)$ is at most countable.

\begin{proposition}\label{Prop:numberVertexes}
Let $Z$ be a finite-dimensional real Banach space. Then, the set of points $u \in S_{Z}$ satisfying $n(Z,u) >0 $ is countable. As a consequence, $\mathcal{N}(Z)$ is countable.
\end{proposition}

\begin{proof}
Let $u\in S_Z$ be such that $n(Z,u)>0$. By Proposition \ref{Prop:CharacterizationsNumericalIndexVector}, we get that $\conv\bigl(\Face(B_{Z^\ast},u) \cup -\Face(B_{Z^\ast},u)\bigr)$ has non-empty interior so, being $Z^\ast $ finite-dimensional, $\Face(B_{Z^\ast},u)$ has non-empty interior relative to $S_{Z^\ast }$ (indeed, otherwise $\Face(B_{Z^\ast},u)$ has affine dimension at most $\dim(Z^\ast )-2$, so its linear span has dimension at most $\dim(Z^\ast )-1$, and so $\conv\bigl(\Face(B_{Z^\ast},u) \cup -\Face(B_{Z^\ast},u)\bigr)$ has empty interior, a contradiction). Besides, taken $u_{1},u_2\in S_Z$, as
\begin{equation}\label{eq:prop-numberVertexes}
\Face(B_{Z^\ast},u_1) \cap \Face(B_{Z^\ast},u_2) \subseteq \ker(u_{1} - u_{2})\,,
\end{equation}
the  relative interiors of $\Face(B_{X^\ast},u_1)$ and $\Face(B_{X^\ast},u_2)$ are disjoint if $u_{1} \neq u_{2}$. Hence, by separability, the set of those $u \in S_{Z}$ satisfying $n(Z,u) >0 $ has to be countable and, a fortiori, so is  $\mathcal{N}(Z)$.
\end{proof}

We do not know if the above corollary remains valid for ``small'' infinite-dimensional spaces, such as Banach spaces with separable dual. We also do not know whether $\mathcal{N}(Z)$ is countable for every finite-dimensional complex Banach space $Z$.

Our next aim is to give a strengthening of Proposition~\ref{Prop:numberVertexes} for real finite-dimensional spaces, where some techniques from combinatorial geometry are applicable. Let us comment that neither Theorem \ref{Theo:sum-n(X,u)} nor Proposition \ref{Prop:lower-bound-sum-n(X,u)} below are needed in the rest of the paper. We introduce some notation. For a convex body $K \subseteq \mathbb{R}^{n}$ let us denote its \emph{inradius} by
$$
r(K) := \sup\{ r > 0 \colon \exists {x \in K} \text{ such that } x + r B_{\ell_{2}^{n}} \subseteq K \}\,.
$$
Remark that in the case of $K = -K$, the above formula simplifies to $r(K) = \sup\{ r > 0 \colon r B_{\ell_{2}^{n}} \subseteq K \}$.
We denote by $\vol_n[K]$ and $S(K)$ the volume and the surface area of $K$, respectively.

\begin{theorem}\label{Theo:sum-n(X,u)}
Let $Z$ be a real space with $\dim(Z)=m \geq 2$. Then,
$$
\sum_{u \in S_Z}{n(Z,u)^{m-1}} < \infty.
$$	
\end{theorem}	

\begin{proof}
Let us identify, as usual, $Z$ with $ (\R^m, \|\cdot\|)$, $Z^\ast $ with $(\R^m,  \|\cdot\|^\ast )$ and $B_{Z^\ast }$ with the polar body of $B_Z$. Fixed a finite set of points $F$ in $S_{Z}$, we evidently have by (\ref{eq:prop-numberVertexes}) that
\begin{equation}\label{eq:thm-sum-n(X,u)}
\sum_{u \in F}{\vol_{m-1}\left[\Face(B_{Z^\ast},u)\right]}  \leq S(B_{Z^{\ast}}).
\end{equation}
Using Proposition \ref{Prop:CharacterizationsNumericalIndexVector}, for every $u\in F$, we have that
$$
n(Z,u)r(B_{Z^\ast })B_{\ell_2^m}\subseteq n(Z,u)B_{Z^\ast }\subseteq \conv\bigl(\Face(B_{Z^\ast},u)\cup -\Face(B_{Z^\ast},u)\bigr)
$$
and so,
$$
n(Z,u)r(B_{Z^\ast })B_{\ell_2^m}\cap\ker(u)\subseteq \left[\conv\bigl(\Face(B_{Z^\ast},u)\cup -\Face(B_{Z^\ast},u)\bigr)\right]  \cap \ker(u).
$$
For an arbitrary $z^\ast \in \Face(B_{Z^\ast},u)$, the latter set can be rewritten as
$$
\frac{1}{2}\bigl[\Face(B_{Z^\ast},u) - \Face(B_{Z^\ast},u)\bigr] = \frac{1}{2}\bigl[(\Face(B_{Z^\ast},u) -z^\ast) - (\Face(B_{Z^\ast},u) -z^\ast)\bigr].
$$
According to the Rogers--Shephard theorem \cite[Theorem 1]{RogersSheph}, for every convex body $K$ in an $n$-dimensional space one has that
$$
\vol_{n}[K - K] \leq {{2n}\choose{n}} \vol_{n}(K).
$$
Applying this to the subset $(\Face(B_{Z^\ast},u) -z^\ast)$ of the $(m-1)$-dimensional space $\ker(u)$, we obtain the inequality
$$
\vol_{m-1} \left[n(Z,u)r(B_{Z^\ast })B_{\ell_2^m}\cap\ker(u)\right] \leq  \frac{1}{2^{m-1}} {{2(m-1)}\choose{m-1}} \vol_{m-1}\left[\Face(B_{Z^\ast},u)\right].
$$
Therefore, we can write
\begin{align*}
n(Z,u)^{m-1}r(B_{Z^\ast })^{m-1}\vol_{m-1}\left[B_{\ell_2^{m-1}}\right]&=\vol_{m-1} \left[n(Z,u)r(B_{Z^\ast })B_{\ell_2^m}\cap\ker(u)\right]\\
&\leq \frac{1}{2^{m-1}} {{2(m-1)}\choose{m-1}} \vol_{m-1}\left[\Face(B_{Z^\ast},u)\right]
\end{align*}
which, combined with \eqref{eq:thm-sum-n(X,u)}, gives
\begin{equation} \label{eq-sum-above}
\sum_{u \in F}{n(Z,u)^{m-1}} \leq  \frac{1}{2^{m-1}} {{2(m-1)}\choose{m-1}}\frac{S(B_{Z^{\ast}})}{\vol_{m-1}\left[B_{\ell_{2}^{m-1}}\right] \cdot r(B_{Z^{\ast}})^{m-1}}\,.
\end{equation}
As $F$ was arbitrary, we get the desired result.
\end{proof}

For a finite-dimensional polyhedral space (i.e.\ finite-dimensional real space whose unit ball has finitely many faces), we can give a lower bound for the sum of numerical indices of the elements of the unit sphere.

\begin{proposition}\label{Prop:lower-bound-sum-n(X,u)}
Let $Z$ be $\R^m$ endowed with a polyhedral norm such that $B_{Z^\ast } \subseteq B_{\ell_2^m}$. Then,
\begin{equation} \label{eq-sum-below}
\sum_{u \in S_Z}{n(Z,u)}  \geq r(B_{Z^{\ast}})\,.
\end{equation}
\end{proposition}
\begin{proof}
Since $Z^{\ast}$ is also polyhedral, $S_{Z^{\ast}}$ is a finite union of sets of the form $\Face(B_{Z^\ast},u) \cup -\Face(B_{Z^\ast},u)$ for some $u \in S_{Z}$. Let us denote by $F$ the set of corresponding $u \in S_{Z}$. Then, we obviously have the following inclusion
$$
B_{Z^{\ast}} \subseteq \bigcup\nolimits_{u \in F}{\conv{\big( \Face(B_{Z^\ast},u) \cup -\Face(B_{Z^\ast},u)\big)}}\,.
$$
Since
\begin{align*}
\conv{\big( \Face(B_{Z^\ast},u) \cup -\Face(B_{Z^\ast},u)\big)} &\supset r\left(\conv{\big( \Face(B_{Z^\ast},u) \cup -\Face(B_{Z^\ast},u)\big)}\right) B_{\ell_2^m} \\ &\supset r\left(\conv{\big( \Face(B_{Z^\ast},u) \cup - \Face(B_{Z^\ast},u)\big)}\right) B_{Z^\ast },
\end{align*}
Proposition \ref{Prop:CharacterizationsNumericalIndexVector} implies that $n(Z,u) \geq r\left(\conv{\big( \Face(B_{Z^\ast},u) \cup -\Face(B_{Z^\ast},u)\big)}\right)$.
As the convex body $B_{Z^\ast}$ is covered by a finite number of convex bodies,  we can use
\cite[Theorem~2.1]{Kadets05} to get that
$$
\sum_{u \in F}{n(Z,u)} \geq \sum_{u \in F} r\big( \conv \big(\Face(B_{Z^\ast},u) \cup -\Face(B_{Z^\ast},u)\big)\big) \geq r(B_{Z^{\ast}}),
$$
which finishes the proof.
\end{proof}

Let us remark that the estimates in \eqref{eq-sum-above} and \eqref{eq-sum-below} depend on the particular chosen representation of $Z$ as $\R^n$, and they do not pretend to be optimal. It would be interesting to find the sharp estimates in both inequalities.

\subsection{A new result on abstract numerical ranges}

Our goal here is to present a very general result about numerical range spaces which extends and generalizes the results of \cite{Mar-numrange-JMAA2016} and which will be useful to study the behaviour by some Banach space operations on the domain and range spaces of the numerical ranges with respect to operators (see section \ref{sect:stability}) and also to study Lipschitz numerical ranges (see section \ref{sect:Lipschitz}).

\begin{proposition}\label{AProp:num-ranges}
	Let $Z$ be a Banach space, let $u\in S_Z$, and let $C\subseteq B_{Z^\ast}$ be such that $B_{Z^\ast}=\overline{\conv}^{w^\ast}(C)$. Then
	$$
	V(Z,u,z)=\conv\,\bigcap\nolimits_{\delta>0} \overline{\bigl\{z^\ast(z)\colon z^\ast\in C,\, \re z^\ast(u)>1-\delta\bigr\}}
	$$
	for every $z\in Z$. Consequently,
	$$
	v(Z,u, z) = \inf_{\delta>0}\sup \bigl\{|z^\ast(z)|\colon z^\ast\in C,\, \re z^\ast(u)>1-\delta\bigr\}
	$$
	for every $z\in Z$.
\end{proposition}

Let us first observe that the inclusion ``$\supseteq$'' is a straightforward application of the Banach-Alaoglu theorem. Indeed, given $\lambda_0\in \bigcap\nolimits_{\delta>0} \overline{\bigl\{z^\ast(z)\colon z^\ast\in C,\, \re z^\ast(u)>1-\delta\bigr\}}$, for every $n\in \N$ there is $z_n^\ast\in C$ such that
$$
\re z_n^\ast(u)>1-1/n \qquad \text{and} \qquad \bigl|\lambda_0 - z_n^\ast(z)\bigr|<1/n.
$$
If $z_0^\ast\in B_{Z^\ast}$ is a limiting point of the sequence $\{z_n^\ast\}_{n\in \N}$, we have that $z_0^\ast(u)=1$ and $z_0^\ast(z)=\lambda_0$, so $\lambda_0\in V(Z,u,z)$. As this latter set is convex, the inclusion follows.

To prove the more intriguing reverse inequality, we need a couple of preliminary results. The first one is a general version of  \cite[Lemma~2.5]{Mar-numrange-JMAA2016}.

\begin{lemma}\label{Alemma:mainlemma-numericalranges}
	Let $Z$ be a Banach space, let $C\subseteq B_{Z^\ast}$ be such that $B_{Z^\ast}=\overline{\conv}^{w^\ast}(C)$, and let $u\in S_Z$ and $z\in Z$. Then, for every $z_0^\ast\in S_{Z^\ast}$ with $z_0^\ast(u)=1$ and every $\delta>0$, there is $z^\ast\in C$ such that
	$$
	\re z^\ast(u) > 1-\delta \qquad \text{and} \qquad \re z^\ast(z) > \re z_0^\ast(z) - \delta.
	$$
\end{lemma}

\begin{proof}
	As $B_{Z^\ast}=\overline{\conv}^{w^\ast}(C)$, for $\delta'>0$ satisfying $2\|z\|\delta'<\delta$, we may find $n\in \N$, $z_1^\ast,\ldots,z_n^\ast\in C$, $\alpha_1,\ldots,\alpha_n\in [0,1]$ with $\sum\nolimits_{k=1}^n \alpha_k=1$ such that
	$$
	\sum_{k=1}^n \alpha_k \re z_k^\ast(u) > 1 - (\delta')^2 \quad \text{and} \quad \sum_{k=1}^n \alpha_k \re z_k^\ast(z) > \re z_0^\ast(z) - \delta/2.
	$$
	Now, consider
	$$
	J=\bigl\{k\in \{1,\ldots,n\}\colon \re z^\ast_k(u)>1-\delta'\bigr\}
	$$
	and let $L=\{1,\ldots,n\}\setminus J$. We have that
	$$
	1-(\delta')^2 < \sum_{k=1}^n \alpha_k \re z_k^\ast(u) \leq \sum_{k\in J} \alpha_k + \sum_{k\in L} \alpha_k (1-\delta') = 1 - \delta'\sum_{k\in L}\alpha_k,
	$$
	from which we deduce that
	$$
	\sum_{k\in L}\alpha_k < \delta'.
	$$
	Now, we have that
	\begin{align*}
	\re z_0^\ast(z) - \delta/2 & < \sum_{k=1}^n \alpha_k \re z_k^\ast(z) \\ & \leq \sum_{k\in J} \alpha_k\re z_k^\ast(z) + \|z\|\sum_{k\in L} \alpha_k < \sum_{k\in J} \alpha_k\re z_k^\ast(z) + \delta/2.
	\end{align*}
	Therefore,
	$$
	\sum_{k\in J} \alpha_k\re z_k^\ast(z) > \re z_0^\ast(z) - \delta,
	$$
	and an obvious convexity argument provides the existence of $k\in J$ such that
	$$
	\re z_k^\ast(z)> \re z_0^\ast(z) - \delta.
	$$
	On the other hand, $\re z_k^\ast(u)>1-\delta$ as $k\in J$, so the proof is finished.
\end{proof}

The next preliminary result follows straightforwardly from \cite[Lemma 2.4]{Mar-numrange-JMAA2016}.

\begin{lemma}\label{Alemma:nested-family}
	Let $\{W_\delta\}_{\delta>0}$ be a monotone family of compact subsets of $\K$ (i.e.\ $W_{\delta_1} \subseteq W_{\delta_2}$ when $\delta_1 < \delta_2$). Then
	$$
	\sup \re \bigcap\nolimits_{\delta>0} W_\delta \, = \, \inf_{\delta>0} \sup \re W_\delta.
	$$
\end{lemma}

\begin{proof}[Proof of the main part of Proposition \ref{AProp:num-ranges}]
	For $z\in Z$, write $$W_\delta(z):= \overline{\bigl\{z^\ast(z)\colon z^\ast\in C,\, \re z^\ast(u)>1-\delta\bigr\}} \quad \text{ and } \quad W(z):=\bigcap\nolimits_{\delta>0} W_\delta(z).$$ To get the desired inclusion $V(Z,u,z)\subseteq \conv\,W(z)$ for every $z\in Z$, it is enough to prove that for every $\delta>0$ and every $z\in Z$,
	\begin{equation}\label{A-eq:what-to-prove}
	\sup \re V(Z,u,z) \leq \sup \re W_\delta(z) + \delta.
	\end{equation}
	Indeed, it then follows from Lemma \ref{Alemma:nested-family} that
	$\sup \re V(Z,u,z) \leq \sup \re W(z)$ for every $z\in Z$.
	Now, as for every $\theta\in \T$, we have
	$$
	V(Z,u,\theta z)=\theta V(z,u,z) \quad \text{ and } \quad W(\theta z)=\theta W(z),
	$$
	the desired inclusion follows easily.
	
	So let us prove that inequality \eqref{A-eq:what-to-prove} holds. Fix $z\in Z$ and $\delta>0$. Given $z_0^\ast\in \Face(B_{Z^\ast},u)$, we may use Lemma \ref{Alemma:mainlemma-numericalranges} to get $z^\ast\in C$ such that
	$$
	\re z^\ast(u)>1-\delta \qquad \text{and} \qquad \re z_0^\ast(z) < \re z^\ast(z) + \delta.
	$$
	So, $\re z_0^\ast(z)\leq \sup \re W_\delta(z) + \delta$. Moving $z_0^\ast\in \Face(B_{Z^\ast},u)$, we get that $$\sup \re V(Z,u,z) \leq \sup \re W_\delta(z) + \delta,$$ as desired.
\end{proof}

\section{Tools to study the numerical index with respect to an operator}\label{sect:tools_num_index_operators}
Our aim in this section is to provide with some tools to calculate or, at least estimate,  the numerical indices with respect to operators. Some of the results are just direct translation to the operator spaces setting of the abstract results of the previous section, but other ones rely on specifics of the operator case.

We need some notation. Let $X$, $Y$ be Banach spaces. For a norm-one operator $G\in \mathcal{L}(X,Y)$ and $\delta>0$, we write
$$
v_{G,\delta} (T):=\sup\big\{|y^\ast (Tx)| \colon y^\ast \in S_{Y^\ast }, \ x\in S_X, \ \re y^\ast (Gx)>1-\delta \big\}
$$
for every $T\in \mathcal{L}(X,Y)$. It then follows from \cite{Mar-numrange-JMAA2016} (or from Proposition \ref{AProp:num-ranges}) that
$$
v(\mathcal{L}(X,Y),G,T)=v_G(T)=\inf_{\delta>0} v_{G,\delta}(T)
$$
for every $T\in \mathcal{L}(X,Y)$, a result which we will use without any further mention (see Lemma \ref{Lemm-radio-G-A-B} for details).

We include first some results which directly follow from those of section \ref{sect:abstractnumericalindex}. The first one is the translation of Lemma \ref{Lemm:unotsmoothnotextreme} to the setting of the spaces of operators. For a simpler notation, let us say that a norm-one operator $G\in \mathcal{L}(X,Y)$ is an \emph{extreme operator} (or \emph{extreme contraction}) if $G$ is an extreme point of the unit ball of $\mathcal{L}(X,Y)$.

\begin{lemma}\label{lemma:operator-smooth-extreme}
Let $X$, $Y$ be Banach spaces and let $G\in \mathcal{L}(X,Y)$ be a norm-one operator with $n_G(X,Y)>0$. Then, $G$ is a strongly extreme point of $B_{\mathcal{L}(X,Y)}$, in particular, $G$ is an extreme operator. Moreover, if $\dim(X)\geq 2$ or $\dim(Y)\geq 2$, then the norm of $\mathcal{L}(X,Y)$ is not smooth at $G$.
\end{lemma}

Next, we particularize Lemma \ref{lemma:derivadaani} to our setting.

\begin{lemma}\label{Lemma:num-radius-with-right-derivative}
	Let $X$, $Y$ be Banach spaces and let $G\in \mathcal{L}(X,Y)$ be a norm-one operator. Then,
	$$
	v_G(T)= \max_{\theta\in \T} \lim_{\alpha\to 0^+} \frac{\|G+\alpha \theta T\|-1}{\alpha}=\lim_{\alpha\to 0^+} \max_{\theta\in \T}\frac{\|G+\alpha \theta T\|-1}{\alpha}
	$$
	for every $T\in \mathcal{L}(X,Y)$.
\end{lemma}	

We now include a part of Proposition \ref{Prop:CharacterizationsNumericalIndexVector}, particularized to spaces of operators, which allows to characterize the numerical index in terms of the norm of the space of operators.

\begin{proposition}\label{Prop:CharacterizationsNumericalIndexOperators}
	Let $X$, $Y$ be Banach spaces, let $G\in \mathcal{L}(X,Y)$ be a norm-one operator, and $0 < \lambda \leq 1$. Then, the following statements are equivalent:
	\begin{enumerate}
		\item[(i)] $n_G(X,Y) \geq \lambda$;
		\item[(ii)] $\displaystyle \max_{\theta\in \T}\|G + \theta\,T\|\geq 1 + \lambda \|T\|$ for every $T\in \mathcal{L}(X,Y)$.
	\end{enumerate}
\end{proposition}

The case $\lambda=1$ in the previous result gives us the concept of spear operator. A norm-one operator $G\in \mathcal{L}(X,Y)$ is said to be a \emph{spear operator} if
$$
\max_{\theta\in \T}\|G + \theta\,T\|=1+\|T\|
$$
for every $T\in \mathcal{L}(X,Y)$. This concept was introduced in \cite{Ardalani} and deeply studied in the book \cite{SpearsBook}, where we refer for more information and background. Observe that Proposition \ref{Prop:CharacterizationsNumericalIndexOperators} says, in particular, that $G$ is a spear operator if and only if $n_G(X,Y)=1$.

The next result is a direct consequence of Proposition \ref{AProp:num-ranges} and will be very useful later on.

\begin{lemma}\label{Lemm-radio-G-A-B}
	Let $X, Y$ be Banach spaces. Suppose that $A\subseteq B_X$ and $B\subseteq B_{Y^\ast }$ satisfy $\overline{\conv}(A)=B_X$ and $\overline{\conv}^{w^\ast }(B) =B_{Y^\ast }$. Then, given $G\in \mathcal{L}(X,Y)$ with $\|G\|=1$, we have
	$$
	V(\mathcal{L}(X,Y),G,T)=\conv \bigcap_{\delta>0} \overline{\left\{y^\ast (Tx) \colon y^\ast \in B, \ x\in A, \ \re y^\ast (Gx)>1-\delta \right\}}
	$$
	for every $T\in \mathcal{L}(X,Y)$. Accordingly,
	$$
	v_G(T)=\inf_{\delta>0} \sup \big\{|y^\ast (Tx)| \colon y^\ast \in B, \ x\in A, \ \re y^\ast (Gx)>1-\delta \big\}.
	$$
\end{lemma}

\begin{proof}
The result follows from Proposition \ref{AProp:num-ranges} as the hypotheses give that $B_{\mathcal{L}(X,Y)^\ast}=\overline{\conv}^{w^\ast}(A\otimes B)$. Indeed, for every $G\in \mathcal{L}(X,Y)$, we have that
\begin{align*}
\sup_{x\in A,\,y^\ast\in B} \re y^\ast(Gx) & = \sup_{y^\ast\in B}\sup_{x\in A} \re y^\ast(Gx) = \sup_{y^\ast\in B}\sup_{x\in B_X} \re y^\ast(Gx) \\
& = \sup_{x\in B_X}\sup_{y^\ast\in B} \re y^\ast(Gx) = \sup_{x\in B_X}\sup_{y^\ast\in B_{Y^\ast}} \re y^\ast(Gx)=\|G\|.\qedhere
\end{align*}
\end{proof}

We also may relate the numerical index with respect to an operator to the numerical index with respect to its adjoint.

\begin{lemma}\label{lemma:adjoint-inequality}
Let $X, Y$ be Banach spaces. Then,
$$
n_{G^\ast}(Y^\ast,X^\ast)\leq n_G(X,Y)
$$
for every norm-one $G\in L(X,Y)$.
\end{lemma}

\begin{proof}
The result follows immediately from Lemma \ref{Lemma:num-radius-with-right-derivative} and the fact that the norm of an operator and the norm of its adjoint coincide. Alternatively, it also follows from Lemma \ref{Lemm:numericalIndexThroughMaps} as the operator $\Psi:\mathcal{L}(X,Y) \longrightarrow \mathcal{L}(Y^\ast,X^\ast)$ given by $T\longmapsto T^\ast$ is an isometric embedding.
\end{proof}

In the case of rank-one operators, we may provide a formula for the numerical index with respect to it.

\begin{proposition}\label{prop:num-index-rankone-operators}
Let $X$, $Y$ be Banach spaces, $x_0^\ast \in S_{X^\ast }$, and $y_0\in S_Y$. Then, the rank-one operator $G=x_0^\ast \otimes y_0$ satisfies that
$$
n_G(X,Y)=n(X^\ast ,x_0^\ast )\,n(Y,y_0).
$$
\end{proposition}	

We need to introduce some notation, just for this proof. Given a Banach space $Z$, $u\in S_Z$, and $\delta\in (0,1)$, we write
$$
v_\delta(Z,u,z):=\sup \{|z^\ast (z)| \colon  z^\ast\in S_{Z^\ast},\, \re z^\ast (u)>1-\delta\}
$$
and observe (use Proposition \ref{AProp:num-ranges}, for instance) that $v(Z,u,z)=\inf\limits_{\delta>0}v_\delta(Z,u,z)$.

\begin{proof}
Fixed $x^\ast \in S_{X^\ast }$ and $y\in S_Y$, we consider the norm-one operator $T=x^\ast \otimes y$ and show that
$$
v_{G,\delta}(T)\leq v_\delta(X^\ast ,x_0^\ast ,x^\ast ) v_\delta(Y,y_0,y)
$$
for every $\delta>0$. To do so, we first observe that
$$v_\delta(X^\ast ,x_0^\ast ,x^\ast )=\sup \{|x^\ast (x)| \colon  x\in S_X, \re x_0^\ast (x)>1-\delta\}$$
as $J_X(B_X)$ is weak$^\ast$ dense in $B_{X^{\ast \ast }}$. Therefore, we can write
\begin{align*}
v_{G,\delta}(T)&=\sup \{|y^\ast (Tx)| \colon y^\ast \in S_{Y^\ast },\, x\in S_X,\, \re (y^\ast (y_0)x_0^\ast (x))>1-\delta\}\\
&\leq \sup \{|y^\ast (y)||x^\ast (x)| \colon y^\ast \in S_{Y^\ast },\, x\in S_X,\, \re y^\ast (y_0)>1-\delta,\, \re x_0^\ast (x)>1-\delta\}\\
&\leq \sup \{|x^\ast (x)| \colon  x\in S_X,\, \re x_0^\ast (x)>1-\delta\}\ \sup \{|y^\ast (y)| \colon y^\ast \in S_{Y^\ast },\, \re y^\ast (y_0)>1-\delta\}\\
&=v_\delta(X^\ast ,x_0^\ast ,x^\ast )\, v_\delta(Y,y_0,y).
\end{align*}
This clearly gives that $n_G(X,Y)\leq n(X^\ast ,x_0^\ast )n(Y,y_0)$. To prove the reverse inequality, fixed $T\in {\mathcal{L}(X,Y)}$ with $\|T\|=1$ and $\delta>0$, observe that
\begin{align*}
\sup \{\|Tx\| \colon x\in S_X,\,\re x_0^\ast (x)>1-\delta\}&= \sup \{|z^\ast (Tx)| \colon z^\ast \in S_{Y^\ast },\, x\in S_X,\, \re x_0^\ast (x)>1-\delta\}\\
&=\sup \{|[T^\ast z^\ast ](x)| \colon z^\ast \in S_{Y^\ast },\, x\in S_X,\, \re x_0^\ast (x)>1-\delta\}\\
&=\sup \{v_\delta(X^\ast ,x_0^\ast ,T^\ast z^\ast ) \colon z^\ast \in S_{Y^\ast }\}\\
&\geq \sup \{n(X^\ast ,x_0^\ast )\|T^\ast z^\ast \| \colon z^\ast \in S_{Y^\ast }\}\\
&=n(X^\ast ,x_0^\ast )\|T^\ast \|=n(X^\ast ,x_0^\ast ).
\end{align*} 	
Therefore, we can write
\begin{align*}
v_{G,2\delta}(T)&=\sup \{|y^\ast (Tx)| \colon y^\ast \in S_{Y^\ast },\, x\in S_X,\, \re (y^\ast (y_0)x_0^\ast (x))>1-2\delta\}\\
&\geq\sup \{|y^\ast (Tx)| \colon y^\ast \in S_{Y^\ast },\, x\in S_X,\, \re y^\ast (y_0)>1-\delta,\, \re x_0^\ast (x)>1-\delta\}\\
&\geq\sup \{n(Y,y_0)\|Tx\| \colon x\in S_X,\, \re x_0^\ast (x)>1-\delta\}\geq n(Y,y_0)\,n(X^\ast ,x_0^\ast ),
\end{align*}
which gives the desired inequality $n_G(X,Y)\geq n(X^\ast ,x_0^\ast )n(Y,y_0)$. 	
\end{proof}

To finish the section, we would like to present some results which allow to control the numerical index with respect to operators in terms of the numerical radius of the operators on the domain space or on the range space, which we will profusely use in section \ref{sect:setofvalues}. They all follow from this easy key lemma.

\begin{lemma}\label{Lemma:boundNumIndexCoro1}
Let $X, Y$ be Banach spaces and let $G \in \mathcal{L}(X,Y)$ with $\| G\| = 1$. Then,
\begin{enumerate}
\item[(a)] $v_{G}(G \circ T) \leq v(T)$ for every $T \in \mathcal{L}(X)$\,,
\item[(b)] $v_{G}(T \circ G) \leq v(T)$ for every $T \in \mathcal{L}(Y)$\,.
\end{enumerate}
\end{lemma}

\begin{proof}
Both statements are consequences of Lemma \ref{Lemm:numericalIndexThroughMaps} considering the operator $\mathcal{L}(X) \longrightarrow \mathcal{L}(X,Y)$ given by $T \longmapsto G \circ T$ for (a), and the operator $\mathcal{L}(Y) \longrightarrow \mathcal{L}(X,Y)$ given by $T \longmapsto T \circ G$ for (b).
\end{proof}

As a consequence of this result, we have the following chain of inequalities:
\begin{equation*}
\begin{aligned}
n_{G}(X,Y) & \, \leq \, \inf{\left\{ \frac{v(T)}{\| G \circ T \|} \colon T \in \mathcal{L}(X),\, G \circ T \neq 0 \right\}} \\[2mm]
& \, \leq \sup_{\varepsilon > 0}\, \inf{\left\{ v(T)\colon T \in \mathcal{L}(X),\, \|G \circ T\| > 1- \varepsilon \right\}}
\end{aligned}
\end{equation*}
and, analogously,
\begin{equation*}
\begin{aligned}
n_{G}(X,Y) & \, \leq \, \inf{\left\{ \frac{v(T)}{\| T \circ G \|} \colon T \in \mathcal{L}(Y),\, T \circ G \neq 0 \right\}} \\[2mm]
& \, \leq \sup_{\varepsilon > 0}\, \inf{\left\{ v(T)\colon T \in \mathcal{L}(Y),\, \|T \circ G\| > 1- \varepsilon \right\}}.
\end{aligned}
\end{equation*}

These inequalities immediately imply the following result.

\begin{lemma}\label{Lemm:superlemma1y2}
Let $X$, $Y$ be Banach spaces, $G\in \mathcal{L}(X,Y)$ with $\|G\|=1$, and $0\leq \alpha\leq 1$. Then, $n_G(X,Y)\leq \alpha$ provided one of the following statements is satisfied:
\begin{enumerate}
\item[(a)] For every $\eps>0$ there exists $T_\eps\in \mathcal{L}(X)$ such that $v(T_\eps)\leq \alpha$ and $\|G\circ T_\eps\|>1-\eps$;
\item[(b)] For every $\eps>0$ there exists $S_\eps\in \mathcal{L}(Y)$ such that $v(S_\eps)\leq \alpha$ and $\|S_\eps \circ G\|>1-\eps$.
\end{enumerate}
\end{lemma}

The previous result gives some important consequences.

\begin{proposition}\label{Prop:boundNumIndexCoro2}
Let $X$, $Y$ be Banach spaces and let $0\leq \alpha\leq 1$.
\begin{enumerate}
  \item[(a)] Let $\mathcal{A}(\alpha)=\bigl\{T\in \mathcal{L}(X)\colon \|T\|=1,\, v(T)\leq \alpha\bigr\}$. If
      $$
      B_X=\overline{\aconv}\bigcup_{T\in \mathcal{A}(\alpha)} T(B_X)\,,
      $$
      then $n_G(X,Y)\leq \alpha$ for every norm-one operator $G\in \mathcal{L}(X,Y)$.
  \item[(b)] Let $\mathcal{B}(\alpha)=\bigl\{T\in \mathcal{L}(Y)\colon \|T\|=1,\, v(T)\leq \alpha \bigr\}$. If for every $\eps>0$, the set
      $$
      \bigcup_{T\in \mathcal{B}(\alpha)} \bigl\{y\in S_Y\colon \|Ty\|>1-\eps\bigr\}
      $$
      is dense in $S_Y$, then $n_G(X,Y)\leq \alpha$ for every norm-one operator $G\in \mathcal{L}(X,Y)$.
  \item[(c)] In particular, if there exists a surjective isometry $T\in \mathcal{L}(X)$ with $v(T)\leq \alpha$ or there exists a surjective isometry $S\in \mathcal{L}(Y)$ with $v(S)\leq \alpha$, then $n_G(X,Y)\leq \alpha$ for every norm-one operator $G\in \mathcal{L}(X,Y)$.
\end{enumerate}
\end{proposition}

\begin{proof} Fix $G\in \mathcal{L}(X,Y)$ with $\|G\|=1$.

(a).  For every $\eps>0$, we may use the hypothesis to find $T_\eps\in \mathcal{L}(X)$ with $\|T_\eps\|=1$ and $v(T_\eps)\leq \alpha$ such that $\|G(T_\eps(x))\|>1-\eps$ for some $x\in B_X$. Therefore, $\|G\circ T_\eps\|>1-\eps$ and Lemma \ref{Lemm:superlemma1y2} gives the result.

(b). For every $\eps>0$, we take $x\in S_X$ such that $\|Gx\|>1-\eps/3$. Now, we may use the hypothesis to find $S_\eps\in \mathcal{L}(Y)$ with $\|S_\eps\|=1$ and $v(S_\eps)\leq \alpha$, and $y\in S_Y$ such that $\|S_\eps y\|>1-\eps/3$ and $\bigl\|y- Gx/\|Gx\|\bigr\|<\eps/3$. Now, $\|y-Gx\|<2\eps/3$, and so
$$
\|S_\eps(Gx)\|\geq \|S_\eps y\|-\|S_\eps(y-Gx)\|>1-\eps/3 - 2\eps/3=1-\eps.
$$
It follows that $\|S_\eps\circ G\|>1-\eps$ and Lemma \ref{Lemm:superlemma1y2} gives the result.

Finally, item (c) clearly follows from (a) and (b).
\end{proof}

For the special case of $\alpha=0$, the above result can be improved as we do not have to pay attention to the norm of the operators.

\begin{proposition}\label{Prop:boundNumIndexZERO} Let $X$, $Y$ be Banach spaces.
\begin{enumerate}
  \item[(a)] Let $G\in \mathcal{L}(X,Y)$ with $\|G\|=1$.
  \begin{enumerate}
    \item[(a.1)] If there exists $T\in \mathcal{L}(X)$ with $v(T)=0$ and $G\circ T\neq 0$,
      then $n_G(X,Y)=0$.
    \item[(a.2)] If there exists $T\in \mathcal{L}(Y)$ with $v(T)=0$ and $T\circ G\neq 0$, then $n_G(X,Y)=0$.
  \end{enumerate}
  \item[(b)] If $$\bigcap_{T\in \mathcal{L}(Y),\, v(T)=0} \ker T = \{0\},$$
  then $n_G(X,Y)=0$ for every norm-one operator $G\in \mathcal{L}(X,Y)$.
  \item[(c)] If $$\bigcup_{T\in \mathcal{L}(X),\, v(T)=0} T(X)$$ is dense in $X$, then
  $n_G(X,Y)=0$ for every norm-one operator $G\in \mathcal{L}(X,Y)$.
  \end{enumerate}
\end{proposition}

We emphasize the next immediate consequence of the previous result which will be useful in the sequel.

\begin{corollary}\label{Coro:spaces-with-onto-isometry-of-radius-zero}
	Let $W$ be a Banach space such that there is an onto isometry $J\in \mathcal{L}(W)$ with $v(J)=0$. Then
	\begin{enumerate}
		\item[(a)] $n_G(X,W)=0$ for every Banach space $X$ and every norm-one operator $G\in \mathcal{L}(X,W)$.
		\item[(b)] $n_G(W,Y)=0$ for every Banach space $Y$ and every norm-one operator $G\in \mathcal{L}(W,Y)$.
	\end{enumerate}
\end{corollary}

\section{Set of values of the numerical indices with respect to all operators between two given Banach spaces}\label{sect:setofvalues}

We start by showing some general results which can be deduced from the tools implemented in the previous sections. The first result shows that $0$ is always a possible value of the numerical index with respect to operators (unless we are in the trivial case of both spaces being one-dimensional). It is a direct consequence of Proposition \ref{Prop0in:N(X)}.

\begin{proposition}\label{prop:ifzeronotinsidedimensiononeOperators}
Let $X$, $Y$ be Banach spaces. If $\dim(X)\geq 2$ or $\dim(Y)\geq 2$, then $0\in \mathcal{N}(\mathcal{L}(X,Y))$.
\end{proposition}

The result above is actually an equivalence, as the following result is immediate.

\begin{example}
{\slshape $\mathcal{N}(\mathcal{L}(\K,\K))=\{1\}$.}
\end{example}

Next, we particularize Proposition \ref{Prop:numberVertexes} to the case of spaces of operators.

\begin{proposition}
	Let $X$, $Y$ be finite-dimensional real Banach spaces. Then, the set of those norm-one $G\in \mathcal{L}(X,Y)$ with $n_G(X,Y)>0$ is countable. In particular,
	$\mathcal{N}(\mathcal{L}(X,Y))$ is countable.
\end{proposition}	

Our next result shows that all values of the numerical index are valid for operators between Banach spaces. In the real case, this is clear as the numerical indices of all two-dimensional norms do the job (and they are the numerical index with respect to the corresponding identities). But in the complex case, the values of the numerical indices with respect to the identity are not enough (as they are always greater than or equal to $1/\e$, see \cite[Corollary 2.1.19]{Cabrera-Rodriguez}, for instance).

A first simple way of getting arbitrary values of the numerical indices with respect to operators is given in the following result which follows immediately from Proposition \ref{propN(X)=A}.

\begin{example}
{\slshape For every subset $A\subseteq [0,1]$ containing $0$, there is a Banach space $X$ such that $\mathcal{N}(\mathcal{L}(X,\K))=A$.}\ Indeed, just consider as $X$ the predual of the space $Z$ provided in Proposition \ref{propN(X)=A} (which is a dual Banach space as it is the $\ell_1$-sum of finite-dimensional spaces).
\end{example}

Let us also observe that if $X$ is a Banach space of dimension at least two whose dual space is smooth, it follows from Lemma \ref{Lemm:unotsmoothnotextreme} that $\mathcal{N}(\mathcal{L}(X,\K))=\{0\}$. This result contrasts with the already cited fact that $n(X)\geq 1/\e$ for every complex Banach space $X$, so $\mathcal{N}(\mathcal{L}(X,X))$ cannot reduce to $0$ when $X$ is a complex Banach space. Therefore, it looks more interesting to perform the study of the set of values of the numerical indices with respect to all operators from a Banach space to itself, that is, the set
$$
\{n_G(X,X)\colon X \text{ (real or complex) Banach space},\, G\in \mathcal{L}(X),\,\|G\|=1\}.
$$
In the real case it is immediate that this set covers $[0,1]$, just using identity operators \cite[Theorem 3.6]{D-Mc-P-W}. In the complex case, using identity operators one can only cover the interval $[1/\e,1]$. The result will be stated in Example \ref{example:foreverygamma-nG=gamma}. Even more, we will show that there are Banach spaces $X$ such that $\mathcal{N}(\mathcal{L}(X))=[0,1]$, both in the real and in the complex case, see Theorem \ref{theorem-allvaluesinX}.

For real Banach spaces, the Banach space numerical index may be zero, so there is no obstacle for the set $\mathcal{N}(\mathcal{L}(X))$ to be equal to $\{0\}$. We are going to prove that this happens when $X$ is a real Hilbert space of dimension greater than one. Actually, we show that zero is the only possible value of the numerical index with respect to operators, when either the domain space or the range space is a real Hilbert space of dimension at least two.

\begin{theorem}\label{theorem:Real-Hilbert-one-value}
Let $H$ be a real Hilbert space of dimension at least two. Then
$$
\mathcal{N}(\mathcal{L}(X,H))=\mathcal{N}(\mathcal{L}(H,Y))=\{0\}
$$
for all real Banach spaces $X$ and $Y$. In particular, $\mathcal{N}(\mathcal{L}(H))=\{0\}$.		
\end{theorem}	

\begin{proof} Observe that for every pair of points $x,y\in S_H$ with $\langle x,  y\rangle=0$, the operator $T\in S_H$ given by $T(z)=\langle z,x\rangle y-\langle z,y \rangle x$ for $z\in H$ satisfies that $v(T)=0$. So, it is clear that
$$\bigcup_{T\in \mathcal{L}(H),\, v(T)=0} T(H)\ \text{ is dense in $H$} \qquad \text{and} \qquad \bigcap_{T\in \mathcal{L}(H),\, v(T)=0} \ker(T)=\{0\}.$$
Now, both assertions are immediate consequences of Proposition \ref{Prop:boundNumIndexZERO}.
\end{proof}	

For every complex Banach space $W$, its underlying real Banach space $W_\R$ also has trivial set of values of the numerical indices with respect to operators. This is an immediate consequence of Corollary~\ref{Coro:spaces-with-onto-isometry-of-radius-zero} as the multiplication by $i$ is an onto isometry which has numerical radius zero when viewed in $\mathcal{L}(W_\R)$.

\begin{proposition}
	Let $W_\R$ the real Banach space underlying a complex Banach space $W$. Then
	$$
	\mathcal{N}(\mathcal{L}(X,W_\R))=\mathcal{N}(\mathcal{L}(W_\R,Y))=\{0\}
	$$
	for all real Banach spaces $X$ and $Y$. In particular, $\mathcal{N}(\mathcal{L}(W_\R))=\{0\}$.
\end{proposition}	

Other kind of spaces having trivial set of values of the numerical indices with respect to operators are $\mathcal{L}(H)$ and also $\mathcal{K}(H)$, the space of compact linear operators from $H$ to $H$.

\begin{theorem}\label{theorem:L(H)K(H)}
Let $H$ be a real Hilbert space of dimension at least two. Then
$$
\mathcal{N}\bigl(\mathcal{L}(X,\mathcal{L}(H))\bigr)=\mathcal{N}\bigl(\mathcal{L}(X,\mathcal{K}(H))\bigr)= \{0\}
$$
for every Banach space $X$. In particular,
$$
\mathcal{N}\bigl(\mathcal{L}(\mathcal{L}(H))\bigr)=\mathcal{N}\bigl(\mathcal{L}(\mathcal{K}(H))\bigr)=\{0\}.
$$
Moreover, if $H$ is infinite-dimensional or has even dimension, then
$$
\mathcal{N}\bigl(\mathcal{L}(\mathcal{L}(H),Y)\bigr)=\mathcal{N}\bigl(\mathcal{L}(\mathcal{K}(H),Y)\bigr)=\{0\}
$$
for every Banach space $Y$.	
\end{theorem}	

\begin{proof}
Let us start with the case of $\mathcal{L}(H)$. For $J\in S_{\mathcal{L}(H)}$ we define the operator $\Phi_J: \mathcal{L}(H)\longrightarrow \mathcal{L}(H)$ by $\Phi_J(T)=J\circ T$ for every $T\in \mathcal{L}(H)$. It is immediate that  $\|\Phi_J\|=\|J\|=1$ and that $\Phi_{\Id_H}=\Id_{\mathcal{L}(H)}$.
Therefore, by Lemma \ref{Lemm:numericalIndexThroughMaps}.b, it follows that
$$
v(\Phi_J)=v(\mathcal{L}(\mathcal{L}(H)),\Id_{\mathcal{L}(H)},\Phi_J)= v(\mathcal{L}(H),\Id_H,J)=v(J).
$$
Let us write
$$
\mathcal{B}=\{\Phi_J \colon J\in \mathcal{L}(H),\, \|J\|=1,\,v(J)=0\}
$$
and observe that the result will follow from Proposition~\ref{Prop:boundNumIndexZERO}.b if we show that
$$
\bigcap_{\Phi\in\mathcal{B}}\ker \Phi =\{0\}.
$$
To do so, fixed $T_0\in S_{\mathcal{L}(H)}$, let $x\in S_H$ be such that $\|T_0x\|>\frac12$. Now, define $e_1=\frac{T_0x}{\|T_0x\|}$ and take $e_2\in S_H$ satisfying $\langle e_1, e_2\rangle=0$. We define the operator $J\in \mathcal{L}(H)$ given by
$Jh=\langle h, e_2 \rangle e_1 - \langle h, e_1\rangle e_2$ for $h\in H$, which satisfies $\|J\|=1$ and $v(J)=0$, so $\Phi_J\in \mathcal{B}$. Besides, we can write
$$
\|\Phi_J(T_0)\|=\|J\circ T_0\|\geq \|J(T_0x)\|=\big\|-\|T_0x\|e_2\big\|=\|T_0x\|>\frac12.
$$
Therefore, $T_0\notin \ker \Phi_J$ and thus $\bigcap_{\Phi\in\mathcal{B}}\ker \Phi =\{0\}$, which finishes the proof for $\mathcal{L}(H)$. For $\mathcal{K}(H)$, it suffices to observe that the same argument is valid since $\Phi_J(\mathcal{K}(H))\subseteq \mathcal{K}(H)$ and we may repeat the argument considering $\Phi_J:\mathcal{K}(H)\longrightarrow \mathcal{K}(H)$ and getting that
$$
v(\Phi_J)=v(\mathcal{L}(\mathcal{K}(H)),\Id_{\mathcal{K}(H)},\Phi_J)= v(\mathcal{L}(H),\Id_H,J)=v(J).
$$
The rest of the proof is identical.
	
To prove the moreover part, observe that when $H$ is infinite-dimensional or has even dimension, then there is an onto isometry $J\in \mathcal{L}(H)$ with $v(J)=0$. Indeed, in this case we may write $H=\bigl[\oplus_{\lambda\in \Lambda} \ell_2^2\bigr]_{\ell_2}$ for suitable index set $\Lambda$ and, defining $A\in \mathcal{L}(\ell_2^2)$ by $A(x,y)=(y,-x)$, the surjective isometry with numerical index zero is given by
$$ J\bigl[(x_\lambda)_{\lambda\in \Lambda}\bigr]=(Ax_\lambda)_{\lambda\in \Lambda} \qquad \bigl((x_\lambda)_{\lambda\in \Lambda}\in H \bigr). $$
Now, the operator $\Phi_J$ is an onto isometry on $\mathcal{L}(H)$ or $\mathcal{K}(H)$ ($\Phi_{J^{-1}}$ is clearly the inverse of $\Phi_J$) satisfying that $v(\Phi_J)=0$. Then, Corollary \ref{Coro:spaces-with-onto-isometry-of-radius-zero} gives the result.
\end{proof}

When $H$ has odd dimension, we do not know if the equality $n_G(\mathcal{L}(H),Y)=0$ holds for every Banach space $Y$ and every operator $G\in \mathcal{L}\big(\mathcal{L}(H), Y\big)$.

Another result of the same kind tells us that there are many other spaces of operators having trivial set of values of the numerical indices with respect to an operator.

\begin{proposition}\label{prop:spaces-with-big-Z(X)}
Let $W_1,\ldots, W_n$ be real Banach spaces, let $E$ be $\R^n$ endowed with an absolute norm, and let $W=[W_1\oplus\dots\oplus W_n]_E$.  Then, the following statements hold:
\begin{enumerate}
\item[(a)] If $S_E$ is smooth at points whose first coordinate is zero and $\bigcap\bigl\{\ker(S_1)\colon S_1\in \mathcal{L}(W_1),\,v(S_1)=0\bigr\}=\{0\}$, then $\mathcal{N}(\mathcal{L}(X,W))=\{0\}$ for every Banach space $X$.
\item[(b)] If $S_E$ is rotund in the direction of the first coordinate, that is, $S_E$ does not contain line segments parallel to $(1, 0, \ldots, 0)$, and $\bigcup\bigl\{S_1(W_1)\colon S_1\in \mathcal{L}(W_1),\,v(S_1)=0\bigr\}$ is dense in $W_1$, then $\mathcal{N}(\mathcal{L}(W,Y))=\{0\}$ for every Banach space $Y$.
\end{enumerate}
Consequently, if (a) or (b) holds, $\mathcal{N}(\mathcal{L}(W))=\{0\}$.		
\end{proposition}

\begin{proof}
(a). Given a Banach space $X$, a norm-one operator $G\in \mathcal{L}(X,W)$ can be seen as $G=(G_1,\ldots,G_n)$ where $G_k\in \mathcal{L}(X,W_k)$ for $k=1,\ldots,n$. We claim that $n_G(X,W)=0$ if $G_1\neq 0$. Indeed, let $P_1\in \mathcal{L}(W,W_1)$ denote the natural projection on $W_1$ and let $I_1\in \mathcal{L}(W_1,W)$ be the natural inclusion, so $G_1=P_1\circ G$. Observe now that for every $S_1\in \mathcal{L}(W_1)$ with $v(S_1)=0$, the operator $S\in \mathcal{L}(W)$ given by $S=I_1\circ S_1\circ P_1$ clearly satisfies $\|S\|=\|S_1\|$ and $v(S)=0$. Since
$$
P_1\circ G\neq 0\qquad  \text{and}\qquad  \bigcap_{S_1\in \mathcal{L}(W_1),\,v(S_1)=0}\ker(S_1)=\{0\},
$$
we can find $S_1\in \mathcal{L}(W_1)$ with $v(S_1)=0$ such that $S_1\circ P_1\circ G\neq 0$ and so $I_1\circ S_1\circ P_1\circ G\neq 0$. As $v(I_1\circ S_1\circ P_1)=0$, we get $n_G(X,W)=0$ from (a.2) of Proposition \ref{Prop:boundNumIndexZERO}. Therefore, we may and do assume from now on that $G_1=0$. Next we fix $w_0\in S_{W_1}$ and $x^\ast \in S_{X^\ast }$, we consider the norm-one operator $T=x^\ast \otimes (w_0,0,\ldots,0) \in \mathcal{L}(X,W)$, and we shall prove that $v_G(T)=0$. To this end, as
$$
v_G(T)=\inf_{\delta>0} \sup \big\{|w^\ast (Tx)| \colon w^\ast \in S_{W^\ast }, \ x\in S_{X}, \ \re w^\ast (Gx)>1-\delta \big\}
$$
for every $k\in \N$ we can take $w_k^\ast = (w_{k, 1}^\ast, \ldots, w_{k, n}^\ast) \in S_{W^\ast }$ and $x_k\in S_X$ satisfying
\begin{equation*}
\lim_{k}\re w_k^\ast (Gx_k)=1\qquad \text{and} \qquad \lim_{k}\big|w_k^\ast (Tx_k)\big|=v_G(T).
\end{equation*}
For each $k\in \N$ define
$$
e_k^\ast =(\|w_{k,1}^\ast\|,\ldots,\|w_{k,n}^\ast\|)\in S_{E^\ast } \quad \text{and} \quad e_k=(\|G_1x_k\|,\ldots,\|G_nx_k\|)\in B_E
$$
which satisfy $1=\lim_{k}\re w_k^\ast (Gx_k)\leq \lim_{k}\langle e_k^\ast , e_k \rangle\leq 1$, and thus $\lim_{k}\langle e_k^\ast , e_k \rangle=1$. Now, by passing to a subsequence, we may find $y^\ast = (y_1^\ast, \ldots, y_n^\ast) \in S_{E^\ast }$ and $y = (y_1, \ldots, y_n) \in S_E$ such that $\lim_{k \to \infty}e_k^\ast = y^\ast $ and $\lim_{k \to \infty} e_k = y$. Then it follows that
$$
\langle y^\ast , y \rangle=\lim_{k}\langle e_k^\ast , e_k \rangle=1
$$
and $y^\ast $ is a supporting functional of $y$. Moreover, we have that $y_1=0$ as the first coordinate of $e_k$ is equal to $\|G_1x_k\| = 0$ for every $k$, so
\begin{equation*}\label{eq:prop-trivial-set-of-values-absolute-sum}
1=\langle y^\ast , y \rangle =\sum_{j=1}^n y_{j}^\ast(y_j)=\sum_{j=2}^n y_{j}^\ast(y_j)
\end{equation*}
and the element $\widetilde{y}^\ast =(0,y_2^\ast,\ldots,y_n^\ast)\in B_{E^\ast}$ is also a supporting functional of $y$. Therefore, we get that $\widetilde{y}^\ast =y^\ast$ by the smoothness of $S_E$ at point $y$ and so $y_1^\ast = 0$. Finally, we can write
\begin{align*}
v_G(T)&=\lim_{k}\big|w_k^\ast (Tx_k)\big|=\lim_{k}\big|w_k^\ast (w_0,0,\ldots,0)\big||x^\ast (x_k)|\\
&\leq \lim_{k}\|w_{k,1}^\ast\|\|w_0\|\leq \lim_{k}\|w_{k,1}^\ast\|=y_1^\ast=0
\end{align*}
which gives $v_G(T)=0$ and finishes the proof of $(1)$.

To prove (b) we start observing that we can suppose that $G\circ I_1\circ S_1\circ P_1=0$ for every $S_1\in \mathcal{L}(W_1)$ with $v(S_1)=0$. Indeed, if there is $S_1\in \mathcal{L}(W_1)$ with $v(S_1)=0$ such that $G\circ I_1\circ S_1\circ P_1\neq 0$, then $S=I_1\circ S_1\circ P_1$ satisfies $v(S)=0$ and $G\circ S\neq 0$. So Proposition \ref{Prop:boundNumIndexZERO} gives $n_G(W,Y)=0$. Then, we have $G\circ I_1\circ S_1\circ P_1=0$ for every $S_1\in \mathcal{L}(W_1)$ with $v(S_1)=0$. This, together with the fact that the set
$$
\displaystyle{\bigcup_{S_1\in \mathcal{L}(W_1),\,v(S_1)=0}S_1(W_1)}
$$
is dense in $W_1$, implies that $G\circ I_1=0$. Next, we fix $y_0\in S_Y$, $w_0^\ast \in S_{W_1^\ast }$, we define $w^\ast =(w^\ast _0,0,\ldots,0)\in S_{W^\ast }$ and the rank-one operator $T=w^\ast \otimes y_0 \in S_{\mathcal{L}(W,Y)}$, and we shall prove that $v_G(T)=0$. To do so, since
$$
v_G(T)=\inf_{\delta>0} \sup \big\{|y^\ast (T w)| \colon y^\ast \in S_{Y^\ast }, \ w\in S_{W}, \ \re y^\ast (Gw)>1-\delta \big\}
$$
for every $k\in \N$ we can take $w_k = (w_{k, 1}, \ldots, w_{k, n}) \in S_{W }$ and $y_k^\ast \in S_{Y^\ast }$ satisfying
\begin{equation*}
\lim_{k}\re y_k^\ast (G w_k)=1\qquad \text{and} \qquad \lim_{k}\big|y_k^\ast (T w_k)\big|=v_G(T).
\end{equation*}
By passing to a subsequence, we may and do assume that $\{\|w_{k,j}\|\}_{k}$ is convergent for every $j=1,\ldots,n$. So, using that the norm in $E$ is absolute, we can define elements
$$ e^+ =\big(\lim_k \|w_{k,1}\|, \lim_k \|w_{k,2}\|,\ldots, \lim_k \|w_{k,n}\|\big), e^- =\big(-\lim_k \|w_{k,1}\|, \lim_k \|w_{k,2}\|,\ldots, \lim_k \|w_{k,n}\|\big)
$$
and $$\widetilde{e}=\big(0,\lim_k \|w_{k,2}\|, \ldots, \lim_k \|w_{k,n}\|\big) = \frac{1}{2}(e^+ + e^-)$$
which clearly satisfy $\|\widetilde{e}\|\leq \|e^+\| = \|e^-\|  \leq 1$. Using now that $G\circ I_1=0$ we can estimate as follows
\begin{align*}
1 &= \lim_{k}\re y_k^\ast (G w_k)= \lim_{k}\re y_k^\ast \big(G (0,w_{k,2},\ldots, w_{k,n})\big)\\
&\leq \lim_{k}\big\|(0,w_{k,2},\ldots, w_{k,n})\big\|\leq \lim_{k}\big\|(0,\|w_{k,2}\|,\ldots, \|w_{k,n}\|)\big\|_E=\|\widetilde{e}\|\leq 1
\end{align*}
which gives $\widetilde{e}\in S_E$ and thus $e^{\pm}\in S_E$. So, we deduce that $\lim_k \|w_{k,1}\|=0$ since $S_E$ is rotund in the direction of the first coordinate. To finish the proof, observe that
$$
v_G(T)=\lim_{k}\big|y_k^\ast (T w_k)\big|= \lim_{k} |y_k^\ast (y_0)||w^\ast (w_k)|\leq \lim_{k} \|w_0^\ast \|\|w_{k,1}\|=0.
$$
Therefore, we get $v_G(T)=0$ and $n_G(W,Y)=0$.
\end{proof}	

\begin{remark}\label{remark:smoothnessrotundityneeded_E-sum}
{\slshape The smoothness and rotundity hypotheses in Proposition \ref{prop:spaces-with-big-Z(X)} cannot be omitted.\ } Indeed, on the one hand, the rank-one operator $G\in \mathcal{L}(\ell_2^2\oplus_\infty \R,\R)$ given by $G=(0,0,1)\otimes \eins$ is a spear operator by Proposition \ref{prop:num-index-rankone-operators} as $\eins$ is a spear vector in $\R$ and $(0,0,1)$ is a spear vector in $(\ell_2^2\oplus_\infty \R)^\ast =\ell_2^2\oplus_1 \R$. Thus, the assumption of smoothness in item (a) of Proposition~\ref{prop:spaces-with-big-Z(X)} is essential. On the other hand, the operator $G^\ast \in \mathcal{L}(\R,\ell_2^2\oplus_1\R)$ is also a spear operator by the same argument, showing that we cannot omit the rotundity in item (b) of Proposition~\ref{prop:spaces-with-big-Z(X)}. 	
\end{remark}	

The next example is even more surprising.

\begin{example}\label{exam:n-zero-nG-one}
{\slshape There exists a Banach space $X$ with $n(X)=0$ such that there is a spear operator $G\in \mathcal{L}(X)$.} Indeed, consider $X=(\ell_2^2\oplus_\infty\R)\oplus_1\R$, which clearly satisfies $n(X)=0$, and $G\in \mathcal{L}\big((\ell_2^2\oplus_\infty\R)\oplus_1\R\big)$ given by $G=(0,0,0,1)\otimes(0,0,0,1)$, which is a spear operator by Proposition \ref{prop:num-index-rankone-operators}.
\end{example}	

Our next result estimates the numerical indices with respect to operators whose domain or range is an $\ell_p$-space.

\begin{proposition}\label{Prop:ell-p-Space}
	Let $1<p<\infty$, $1<q<\infty$ with $\frac1p+\frac1q=1$, let $M_p=\sup_{t\in[0,1]}\frac{|t^{p-1}-t|}{1+t^p}$, and let $\Gamma$ be either an infinite set or a finite set with an even number of elements. Then, in the real case,
	$$
	\mathcal{N}\bigl(\mathcal{L}(X,\ell_p(\Gamma))\bigr)\subseteq [0,M_p] \quad \text{and}\quad  \mathcal{N}\bigl(\mathcal{L}(\ell_p(\Gamma),Y)\bigr)\subseteq [0,M_p]
	$$
for all Banach spaces $X$ and $Y$.
\end{proposition}

\begin{proof} The argument is very similar to the one given at the end of the proof of Theorem \ref{theorem:L(H)K(H)}. By the assumption on the set $\Gamma$ we may write $\ell_p(\Gamma)=\bigl[\oplus_{\lambda\in \Lambda} \ell_p^2\bigr]_{\ell_p}$ for suitable index set $\Lambda$ and, defining $A\in \mathcal{L}(\ell_p^2)$ by $A(x,y)=(y,-x)$, then the operator given by
$$
J\bigl[(x_\lambda)_{\lambda\in \Lambda}\bigr]=(Ax_\lambda)_{\lambda\in \Lambda} \quad \bigl((x_\lambda)_{\lambda\in \Lambda}\in \ell_p(\Gamma) \bigr)
$$
is a surjective isometry. As $v(A)=M_p$ (see the comments after \cite[Theorem 1]{KaMaPa}), we get that $v(J)\leq M_p$. Now, Corollary \ref{Coro:spaces-with-onto-isometry-of-radius-zero} gives the result.
\end{proof}

We now pass to study some results for complex spaces. As a first result, we may calculate the set of values of the numerical indices with respect to operators between two Hilbert spaces.

\begin{proposition}\label{Prop:complexHilbertSpace1}
Let $H_{1}, H_{2}$ be complex Hilbert spaces with dimension greater than one. Then, we have that $\mathcal{N}\bigl(\mathcal{L}(H_1,H_2)\bigr)=\{0,1/2\}$ if $H_1$ and $H_2$ are isometrically isomorphic and $\mathcal{N}\bigl(\mathcal{L}(H_1,H_2)\bigr)=\{0\}$ in other case.
\end{proposition}

\begin{proof}
Observe that $\mathcal{L}(H_1,H_2)$ is a $JB^*$-triple (see \cite[\S 2.2.27, \S 4.1.39]{Cabrera-Rodriguez} for the definition) under the triple product
$$
\{xyz\}=\frac{1}{2}\bigl(xy^\ast z + z y^\ast x\bigr) \qquad \big(x,y,z\in \mathcal{L}(H_1,H_2)\big),
$$
as it is a closed subtriple of the $C^*$-algebra $\mathcal{L}(H_1\oplus_2 H_2)$ (we may use \cite[Facts 4.1.40 and 4.1.41]{Cabrera-Rodriguez}). Now, as $\mathcal{L}(H_1,H_2)$ is not abelian since $\dim(H_1)\geq 2$ and $\dim(H_2)\geq 2$, see \cite[\S 4.1.47]{Cabrera-Rodriguez}, it follows from \cite[Theorem 4.2.24]{Cabrera-Rodriguez} that $n_G(H_1,H_2)$ is equal to $0$ or $1/2$ for every norm-one operator $G\in \mathcal{L}(H_1,H_2)$.

Next, we take into account that, by \cite[Theorem 4.2.24]{Cabrera-Rodriguez}, $J=\mathcal{L}(H_1,H_2)$ contains a geometrically unitary element if and only if $J$ contains a unitary element as Jordan $*$-triple, that is, if there is $U\in J$ such that $\{UUT\}=T$ for every $T\in J$ (see \cite[Definition 4.1.53]{Cabrera-Rodriguez}). It is known to experts that this implies that $H_1$ and $H_2$ are isometrically isomorphic, but we may give an easy argument. Taking into account the formula of the product in $J$, we get that
$$
UU^\ast T + T U^\ast U = 2 T
$$
for every $T\in \mathcal{L}(H_1,H_2)$. Just considering rank-one operators $T\in \mathcal{L}(H_1,H_2)$, it readily follows that $$UU^\ast=\Id_{H_2} \qquad \text{and} \qquad U^\ast U=\Id_{H_1},$$ giving the desired result.
\end{proof}

Following an argument similar to the one given in Theorem \ref{theorem:Real-Hilbert-one-value}, we can establish the next result.

\begin{proposition}\label{Prop:complexHilbertSpace2}
Let $H$ be a complex Hilbert space with $\dim(H) \geq 2$. Then,
$$
\mathcal{N}\bigl(\mathcal{L}(X,H)\bigr)\subseteq [0,1/2] \quad \text{and}\quad  \mathcal{N}\bigl(\mathcal{L}(H,Y)\bigr)\subseteq [0,1/2]
$$
for all complex Banach spaces $X$ and $Y$.
\end{proposition}

\begin{proof}
For each $u \in S_{H}$, let $v \in S_{H}$ with $\langle u, v\rangle = 0$ we have that the operator
\[ T: H \longrightarrow H\,, \quad T(x) = \langle x,v \rangle \, u  \]
satisfies $v(T) \leq 1/2$. An application of Proposition \ref{Prop:boundNumIndexCoro2} gives the result.
\end{proof}

Our next aim is to study the set $\mathcal{N}\bigl(\mathcal{L}(C(K_1),C(K_2))\bigr)$, where $K_{1}$ and $K_{2}$ are compact Hausdorff topological spaces. Recall that, by Lemma \ref{lemma:operator-smooth-extreme}, if $n_G(C(K_1),C(K_2))>0$ for some $G\in \mathcal{L}(C(K_1),C(K_2))$, then $G$ is an extreme operator. There is a well studied special kind of extreme operators between $C(K)$ spaces, the nice operators. A norm-one operator $G\in \mathcal{L}(C(K_1),C(K_2))$ is said to be \emph{nice} if $$G^\ast(\delta_t)\in \T\{\delta_s\colon s\in K_1\}$$
for every $t\in K_2$ (that is, $G^\ast$ carries extreme points of $B_{C(K_2)^\ast}$ to extreme points of $B_{C(K_1)^\ast}$). It is immediate that nice operators are extreme, but the converse result is not always true (see Remark \ref{remark-extreme-not-nice} below). We claim that a nice operator $G$ satisfies that $n_G((C(K_1),C(K_2))=1$. Indeed, this is easy to show by hand using the properties of the $\delta$-functions in the dual of a $C(K)$ space, but also follows directly from \cite[Proposition 4.2]{SpearsBook} and \cite[Example 2.12.a]{SpearsBook}. Therefore, if for a pair of compact Hausdorff topological spaces $(K_1,K_2)$ it is known that every extreme operator in $\mathcal{L}(C(K_1),C(K_2))$ is nice, then the only possible values of the numerical index of operators in  $\mathcal{L}(C(K_1),C(K_2))$ is $0$ or $1$. This idea leads to a couple of results, one for the real case and another one for the complex case.

\begin{theorem}\label{theorem:CK-nice-real}
Let $K_1$, $K_2$ be compact Hausdorff topological spaces such that at least one of them has more than one point. Then, in the real case, one has
$$
\mathcal{N}\bigl(\mathcal{L}(C(K_1),C(K_2))\bigr)=\{0,1\}
$$
provided at least one of the following assumptions holds:
\begin{enumerate}
\item[(1)]  $K_{1}$ is metrizable,
\item[(2)]  $K_{1}$ is Eberlein compact and $K_{2}$ is metrizable,
\item[(3)]  $K_{2}$ is extremally disconnected,
\item[(4)]  $K_{1}$ is scattered.
\end{enumerate}
\end{theorem}

\begin{proof}
First, as at least one of the spaces $C(K_1)$ and $C(K_2)$ has dimension greater than one, Proposition \ref{prop:ifzeronotinsidedimensiononeOperators} gives that $0\in \mathcal{N}\bigl(\mathcal{L}(C(K_1),C(K_2))\bigr)$. By considering the rank-one operator $G=\delta_t\otimes \boldsymbol{1}$, it is immediate that $1\in \mathcal{N}\bigl(\mathcal{L}(C(K_1),C(K_2))\bigr)$ by Proposition \ref{prop:num-index-rankone-operators}.
Finally, to get the reverse inclusion, by the comments before the statement of the theorem, we just have to check that under the given conditions, every extreme operator in $\mathcal{L}(C(K_1),C(K_2))$ is actually nice. For (1), this is shown in \cite[Theorem~1]{Blu-Lin-Phe};  for (2) in  \cite[Theorem~7]{Ami-Lin};  \cite[Theorem~4]{Sharir72} gives (3); finally, (4) follows from \cite[Theorem~5]{Sharir72}.
\end{proof}

For the complex case, we have a similar result.

\begin{theorem}\label{theorem:CK-nice-complex}
Let $K_1$, $K_2$ be compact Hausdorff topological spaces such that at least one of them has more than one point. Then, in the complex case, one has
$$
\mathcal{N}\bigl(\mathcal{L}(C(K_1),C(K_2))\bigr)=\{0,1\}
$$
provided at least one of the following assumptions holds:
\begin{enumerate}
\item[(1)]  $K_{2}$ is extremally disconnected,
\item[(2)]  $K_{1}$ is metrizable and  $K_{2}$ is basically disconnected (i.e.\ the closure of every $F_{\sigma}$-open is open),
\item[(3)]  $K_{1}$ scattered.
\end{enumerate}
\end{theorem}

\begin{proof}
We just have to follow the lines of the proof of Theorem \ref{theorem:CK-nice-real}, but here we have to provide references for the fact that, in the complex case, every extreme operator in $\mathcal{L}(C(K_1),C(K_2))$ is actually nice under the presented conditions. For (1), this is shown in \cite[Theorem~4]{Sharir72}; (2) is proved in \cite[Theorem~1.4]{Gendler}; finally, (3) follows from \cite[Theorem~5]{Sharir72}.
\end{proof}

\begin{remark}\label{remark-extreme-not-nice}
There are examples showing that it is not true in general that all extreme operators between spaces of continuous functions are nice \cite{Sharir76,Sharir77}. The underlying idea in these examples is to consider for an arbitrary compact Hausdorff space $K$ the canonical inclusion $G$ given by
\[ G: C(K) \longrightarrow C(B_{C(K)^\ast}, w^{\ast}) \]
which satisfies that
\[ G^{\ast}(\delta_{\mu}) = \mu \]
for every $\mu \in B_{C(K)}$ and so, it is not nice. Additional hypothesis on the compact space $K$ (e.g.\ $K$ perfect in the complex case, see \cite[Theorem~2.5]{Sharir76}) ensure, however, that $G$ is an extreme point.
We do not know whether the numerical index with respect to operators $G$ defined as above has to be always $0$ or $1$.
\end{remark}

\begin{remark}
Let us also comment that, in the real case, examples as the ones in the previous remark cannot be compact: for arbitrary compact Hausdorff topological spaces $K_1$ and $K_2$, every compact extreme operator $G\in \mathcal{L}(C(K_1),C(K_2))$ is nice \cite[Theorem~4.5]{Mor-Phe}. Moreover, if $K_2$ is separable, every weakly compact extreme operator $G\in \mathcal{L}(C(K_1),C(K_2))$ is nice \cite[Proposition~2.8]{Cor-Lin}.
\end{remark}

As a consequence of Theorems \ref{theorem:CK-nice-real} and \ref{theorem:CK-nice-complex}, we get the following particular case.

\begin{corollary}
Let $K_1$ be a compact Hausdorff topological space and let $(\Omega, \Sigma, \mu)$ be a $\sigma$-finite measure space such that at least one of the spaces $C(K_1)$ or $L_\infty(\mu)$ has dimension at least two. Then,
$$
\mathcal{N}\bigl(\mathcal{L}(C(K_1),L_\infty(\mu))\bigr)=\{0,1\}
$$
in both the real and the complex case.
\end{corollary}

Indeed, this is just consequence of the fact that every $L_\infty(\mu)$ space can be identified with a $C(K_{\mu})$-space where $K_{\mu}$ is extremally disconnected. With this in mind, the following particular case also holds.

\begin{corollary}
Let $(\Omega_{i}, \Sigma_{i}, \mu_{i})$ ($i=1,2$) be $\sigma$-finite measure spaces such that at least one of the spaces $L_\infty(\mu_i)$, $i=1,2$, has dimension at least two. Then,
$$
\mathcal{N}\bigl(\mathcal{L}(L_\infty(\mu_1),L_\infty(\mu_2))\bigr)=\{0,1\}
$$
in both the real and the complex case.
\end{corollary}

We may get an analogous result for $L_1(\mu)$ spaces.

\begin{corollary}
Let $(\Omega_{i}, \Sigma_{i}, \mu_{i})$ ($i=1,2$) be $\sigma$-finite measure spaces. Then,
$$
\mathcal{N}\bigl(\mathcal{L}(L_1(\mu_1),L_1(\mu_2))\bigr)\subseteq\{0,1\}
$$
in both the real and the complex case.
\end{corollary}

\begin{proof}
Fix a norm-one operator $G\in \mathcal{L}(L_1(\mu_1),L_1(\mu_2))$. If $G$ is not extreme,  $n_G(L_1(\mu_1),L_1(\mu_2))=0$ by Lemma \ref{lemma:operator-smooth-extreme}. If, otherwise, $G$ is an extreme operator, $G^\ast\in \mathcal{L}(L_\infty(\mu_2),L_\infty(\mu_1))$ is nice by \cite[Corollary 2.4]{Sharir73}, so $n_{G^\ast}(L_\infty(\mu_2),L_\infty(\mu_1))=1$ from the discussion preceding Theorems \ref{theorem:CK-nice-real} and \ref{theorem:CK-nice-complex}. But then, $n_G(L_1(\mu_1),L_1(\mu_2))=1$ by Lemma \ref{lemma:adjoint-inequality}. This shows that $\mathcal{N}\bigl(\mathcal{L}(L_1(\mu_1),L_1(\mu_2))\bigr)\subseteq \{0,1\}$.
\end{proof}

Let us show that the set $\mathcal{N}\bigl(\mathcal{L}(L_1(\mu_1),L_1(\mu_2))\bigr)$ does not always contain the value $1$.

\begin{example}
{\slshape $\mathcal{N}\bigl(\mathcal{L}(\ell_1,L_1[0,1])\bigr)=\{0\}$.}

Indeed,  by \cite[Proposition 3.3]{SpearsBook} any operator $G\in \mathcal{L}(\ell_1,L_1[0,1])$ satisfying $n_G(\ell_1,L_1[0,1])=1$ would carry the elements of the basis of $\ell_1$ to spear vectors of $L_1[0,1]$ and thus, to extreme points of the unit ball of $L_1[0,1]$ \cite[Proposition 2.11.b]{SpearsBook}, so there are no such operators. On the other hand, $0\in \mathcal{N}\bigl(\mathcal{L}(\ell_1,L_1[0,1])\bigr)$ by Proposition \ref{prop:ifzeronotinsidedimensiononeOperators}.
\end{example}	

\section{Lipschitz numerical range}\label{sect:Lipschitz}

We would like to deal now with the Lipschitz numerical range introduced in \cite{Wang2012, Wang-Huang-Tan} and show that it can be viewed as a particular case of the numerical range with respect to a linear operator. We need some notation. Let $X$, $Y$ be Banach spaces. We denote by $\Lip_{0}{(X, Y)}$ the set of all Lipschitz maps $F: X \longrightarrow Y$ such that $F(0)=0$. This is a Banach space when endowed with the norm
\[ \| F\|_L = \sup \left\{ \frac{\|F(x) - F(y)\|}{\| x - y\|}\colon x,y\in X,\, x \neq y \right\}.
\]
Following \cite{Wang2012, Wang-Huang-Tan}, the \emph{Lipschitz numerical range} of $F\in \Lip_0(X,X)$  is
$$
W_L(F):=\left\{\frac{\xi^\ast\bigl(F(x)-F(y)\bigr)}{\|x-y\|}\colon \xi^\ast\in S_{X^\ast},\, x,y\in X,\, x\neq y,\, \xi^\ast(x-y)=\|x-y\|\right\},
$$
the \emph{Lipschitz numerical radius} of $F$ is just $w_L(F):=\sup\bigl\{|\lambda|\colon \lambda\in W_L(F)\bigr\}$, and the \emph{Lipschitz numerical index} of $X$ is
$$
n_L(X):=\inf\bigl\{w_L(F)\colon F\in \Lip_0(X,X),\, \|F\|_L=1\bigr\} = \max\bigl\{k\geq 0\colon k\|F\|_L \leq w_L(F)\,\forall F\in \Lip_0(X,X)\bigr\}.
$$
We would like to show that the closed convex hull of the Lipschitz numerical range is equal to the numerical range with respect to a linear operator. To do so, we need to recall the concept of Lipschitz free space. First, observe that we can associate to each $x \in X$ an element $\delta_{x} \in \Lip_{0}{(X,\K)}^{\ast}$ which is just the evaluation map $\delta_{x}(F) = F(x)$ for every $F\in \Lip_0(X,\K)$. The \emph{Lipschitz free space} over $X$ is defined as
\[
\FF(X):= \overline{\spn}^{\| \cdot \|}{\{ \delta_{x}\colon x \in X\}} \subseteq \Lip_{0}{(X, \K)}^{\ast}.
\]
It turns out that $\FF(X)$ is an isometric predual of $\Lip_{0}{(X, \K)}$. Moreover, the map $\delta: x \leadsto \delta_{x}$ establishes an isometric (non-linear) embedding $X \hookrightarrow \FF(X)$ since $\|\delta_{x}-\delta_{y}\|_{\FF(X)}=\|x-y\|_X$ for all $x,y\in X$. The name of Lipschitz free space comes from \cite{GodKalt}, but the concept was studied much earlier and it is also known as the Arens-Ells space of $X$. We refer the reader to the paper \cite{Godefroy-surveyFF} and the book \cite{Weaver} for more information and background. The main features of the Lipschitz free space that we are going to use here are contained in the following result which is nowadays considered folklore in the theory of Lipschitz maps and can be found in the cited references \cite{Godefroy-surveyFF,GodKalt,Weaver}.

\begin{lemma}\label{Lemm:elementarypropertiesFF(X)}
Let $X$, $Y$ be Banach spaces.
\begin{enumerate}
\item[(a)] For every $F\in \Lip_0(X,Y)$ there exists a unique linear operator $T_{F}: \FF(X) \longrightarrow Y$ such that $T_{F}\circ \delta = F$ and $\|T_{F}\| = \|F\|_L$. Moreover, $\Lip_0(X,Y)$ is isometrically isomorphic to $\mathcal{L}(\FF(X),Y)$. In particular, $\Lip_0(X,\K)=\FF(X)^\ast$.
\item[(b)] When the above is applied to $\Id\in \Lip_0(X,X)$, we get the operator $\mathcal{G}_X: \FF(X) \longrightarrow X$ given by
\[
\mathcal{G}_X\left(\sum_{x \in X}{a_{x} \delta_{x}}\right) = \sum_{x \in X}{a_{x}  x}
\]
which has norm one and satisfies that $\mathcal{G}_X\circ \delta=\Id_X$.
\item[(c)] The set
\begin{equation*}
 \mathcal{B}_X := \left\{ \frac{\delta_{x} - \delta_{y}}{\| x - y\|}\colon x,y\in X,\, x \neq y\right\}\subseteq \FF(X)
\end{equation*}
is norming for $\FF(X)^\ast=\Lip_0(X,\K)$, i.e.\ $B_{\FF(X)} = \overline{\aconv}{(\mathcal{B}_X)}$.
\end{enumerate}
\end{lemma}

Our result for Lipschitz numerical ranges is the following.

\begin{theorem}\label{Theo:Lip-numerical-range-Id}
Let $X$ be a Banach space. Then,
$$
\overline{\conv} \bigl(W_L(F)\bigr)=V(\mathcal{L}(\FF(X),X),\mathcal{G}_X,T_F)=V(\Lip_0(X,X),\Id,F)
$$
for every $F\in \Lip_0(X,X)$.
\end{theorem}

The result will be consequence of two lemmas. The first one follows directly from Proposition \ref{AProp:num-ranges}, as the set
$$
C=\left\{x^\ast \otimes \frac{\delta_x - \delta_y}{\|x-y\|}\colon x,y\in X,\, x\neq y,\, x^\ast\in S_{X^\ast}\right\}\subseteq \mathcal{L}(\mathcal{F}(X),X)^\ast
$$
satisfies that $B_{\mathcal{L}(\mathcal{F}(X),X)^\ast }= \overline{\conv}^{w^\ast}(C)$ by Lemma \ref{Lemm:elementarypropertiesFF(X)}.c.

\begin{lemma}\label{Lemma:lemma-1-Lip-num-ranges}
Let $X$ be a Banach space. Then
\begin{multline*}
V(\mathcal{L}(\FF(X),X),\mathcal{G}_X,T)\\ =\conv \bigcap_{\delta>0}
\overline{\left\{\frac{x^\ast(T(\delta_x - \delta_y))}{\|x-y\|}\colon x,y\in X,\, x\neq y,\, x^\ast\in S_{X^\ast},\,\re \frac{x^\ast(\mathcal{G}_X(\delta_x - \delta_y))}{\|x-y\|}>1-\delta \right\}}
\end{multline*}
for every $T\in \mathcal{L}(\FF(X),X)$. Equivalently,
\begin{multline*}
V(\Lip_0(X,X),\Id_X,F)\\ =\conv \bigcap_{\delta>0}
\overline{\left\{\frac{x^\ast\bigl(F(x) - F(y)\bigr)}{\|x-y\|}\colon x,y\in X,\, x\neq y,\, x^\ast\in S_{X^\ast},\,\re \frac{x^\ast(x - y)}{\|x-y\|}>1-\delta \right\}}
\end{multline*}
for every $F\in \Lip_0(X,X)$.
\end{lemma}

The second needed preliminary result is the following one which follows from the Bishop-Phelps-Bollob\'{a}s theorem.

\begin{lemma}\label{Lemma:lemma-2-consequenceofBPB}
Let $X$ be a Banach space. Then
$$
\overline{W_L(F)}= \bigcap_{\delta>0} \overline{\left\{\frac{x^{\ast} \bigl(F(x) - F(y)\bigr)}{\| x - y\|}\colon x,y\in X,\, x \neq y,\, x^{\ast} \in S_{X^{\ast}},\, \re \frac{x^{\ast}(x - y)}{\| x - y\|} > 1 - \delta\right\} }
$$
for every $F\in \Lip_0(X,X)$.
\end{lemma}

\begin{proof}
The inclusion ``$\subseteq$'' is obvious, so let us prove the reverse one. Fix $F\in \Lip_0(X,X)$. For every $\delta>0$, write
$$
W_\delta:=\left\{\frac{x^{\ast} \bigl(F(x) - F(y)\bigr)}{\| x - y\|}\colon x,y\in X,\, x \neq y,\, x^{\ast} \in S_{X^{\ast}},\, \re \frac{x^{\ast}(x - y)}{\| x - y\|} > 1 - \delta\right\}.
$$
It is enough to show that for every $\eps>0$, there is $\delta>0$ such that $W_\delta\subseteq W_L(F)+\eps B_\K$. So let us fix $0<\eps<1$ and consider $\delta>0$ such that $2\|F\|_L\sqrt{2\delta}<\eps$. Given $x,y \in X$ with $x \neq y$ and $x^{\ast} \in S_{X^{\ast}}$ satisfying
$$
\re x^{\ast}\left(\frac{x - y}{\|x-y\|}\right) > 1 - \delta,
$$
we can use the Bishop-Phelps-Bollob\'{a}s theorem (see \cite[Corollary~2.4]{ChicaKadetsMartinMorenoRambla-BPBp-modulus} for this version) to find $u \in S_{X}$, $0<\rho<\sqrt{2\delta}$ and $z^{\ast} \in S_{X^{\ast}}$ such that
\[
z^{\ast} \left( \frac{x - y}{\| x - y\|} + \rho u \right) = \left\| \frac{x - y}{\| x - y\|} + \rho u \right\| = 1 \qquad \text{and} \qquad \bigl\| x^{\ast} - z^{\ast}\| < \sqrt{2\delta}.
\]
Write $x' := x + \rho \| x-y\| u$ and $y':= y$, and observe that
$$
\|x'-y'\|=\bigl\|x-y+ \rho\|x-y\|u\bigr\|=\|x-y\|\left\|\frac{x-y}{\|x-y\|}+ \rho u\right\|=\|x-y\|
$$
and so
\[
z^{\ast}\left(\frac{x' - y'}{\| x' - y'\|}\right) = z^\ast \left(\frac{x - y}{\| x - y\|}+\rho u\right)=1.
\]
Therefore, $\frac{z^*\left(F(x')-F(y')\right)}{\|x'-y'\|}\in W_L(F)$. Moreover,
\begin{align*}
\left\|\frac{z^{\ast}\bigl(F(x') - F(y')\bigr)}{\| x' - y'\|} - \frac{x^{\ast}\bigl(F(x) - F(y)\bigr)}{\| x - y\|}\right\| & \leq \left\| \frac{F(x') - F(y')}{\| x' - y'\|} - \frac{F(x)-F(y)}{\| x - y\|} \right\| + \left\|[z^\ast-x^\ast]\frac{F(x) - F(y)}{\| x - y\|}\right\|\\
& \leq  \left\| \frac{F(x') - F(y')-F(x) + F(y)}{\| x - y\|}\right\| + \|F\|_L\,\|z^\ast-x^\ast\|\\
& < \|F\|_L\,\left\| \frac{x' - x}{\| x - y\|}\right\| + \|F\|_L\,\sqrt{2\delta} \\ & = \|F\|_L\bigl(\rho + \sqrt{2\delta}\bigr)< 2\|F\|_L\sqrt{2\delta}<\eps.
\end{align*}
We have shown that $\frac{x^*\left(F(x)-F(y)\right)}{\|x-y\|}\in W_L(F)+\eps B_\K$, so $W_\delta\subseteq W_L(F)+\eps B_\K$ as desired.
\end{proof}

\section{Some stability results}\label{sect:stability}
We collect in this section some results which show the behaviour of the value of the numerical index when we apply some Banach space operations on the operators. We have divided the section in several subsections.

\subsection{Diagonal operators}
The next result allows to calculate the numerical index with respect to a diagonal operator between $c_0$-, $\ell_1$- and $\ell_\infty$-sums of Banach spaces.

\begin{proposition}\label{sumas}
	Let $\{X_\lambda\colon \lambda\in
	\Lambda\}$, $\{Y_\lambda\colon \lambda\in
	\Lambda\}$ be two families of Banach spaces and let $G_\lambda\in \mathcal{L}(X_\lambda,Y_\lambda)$ be a norm-one operator for every $\lambda\in\Lambda$. Let $E$ be one of the Banach spaces $c_0$, $\ell_\infty$, or $\ell_1$, let $X=\left[\bigoplus_{\lambda\in\Lambda} X_\lambda\right]_E$ and $Y=\left[\bigoplus_{\lambda\in\Lambda} Y_\lambda\right]_E$, and define the operator $G\colon X\longrightarrow Y$ by
	$$ G\left[(x_\lambda)_{\lambda\in\Lambda}\right]=(G_\lambda x_\lambda)_{\lambda\in\Lambda} $$
	for every $(x_\lambda)_{\lambda\in\Lambda}\in \left[\bigoplus_{\lambda\in\Lambda} X_\lambda\right]_E$. Then
	$$n_G(X,Y)=\inf_{\lambda} \, n_{G_\lambda}(X_\lambda,Y_\lambda).$$
\end{proposition}

\begin{proof}
We follow the lines of \cite[Proposition~1]{M-P}.
	Fixed $\kappa\in\Lambda$, we have to show that $n_G(X,Y)\leq n_{G_\kappa}(X_\kappa,Y_\kappa)$. Observe that calling $W=\left[\bigoplus_{\lambda\neq\kappa}
	X_\lambda\right]_E$ and $Z=\left[\bigoplus_{\lambda\neq\kappa}
	Y_\lambda\right]_E$, we can write $X=X_\kappa\oplus_\infty W$ and $Y=Y_\kappa\oplus_\infty Z$ when $E$ is $\ell_\infty$ or $c_0$ and $X=X_\kappa\oplus_1 W$ and $Y=Y_\kappa\oplus_1 Z$ when $E$ is $\ell_1$. Given $S\in \mathcal{L}(X_\kappa,Y_\kappa)$, define $T\in \mathcal{L}(X,Y)$ by
	$$T(x_\kappa,w)=(Sx_\kappa,0) \qquad (x_\kappa\in X_\kappa,\ w\in W)$$
	which obviously satisfies $\|T\|=\|S\|$. We claim that $v_G(T)=v_{G_\kappa}(S)$. Indeed, given $\delta>0$, we may suppose that $v_{G,\delta}(T)>0$. We prove that $v_{G,\delta}(T)\leq v_{G_\kappa,\hat{\delta}}(S)$ where $\hat{\delta}=\frac{2\delta}{v_{G,\delta}(T)}$.  For every $0<\eps<\frac{v_{G,\delta}(T)}{2}$, we may find $x=(x_\kappa,w)\in S_X$ and $y^\ast =(y_\kappa^\ast ,z^\ast )\in S_{Y^\ast }$ such that $|y^\ast (Tx)|>v_{G,\delta}(T)-\eps>\frac{v_{G,\delta}(T)}{2}$ and
	\begin{align*}
	1-\delta<\re y^\ast (Gx)\leq \re y_\kappa^\ast (G_\kappa x_\kappa)+\|z^\ast \| \, \|w\|.
	\end{align*}
	Moreover, observe that
	$$\|y_\kappa^\ast \|\,\|x_\kappa\|+\|z^\ast \|\,\|w\| \leq
	\|y^\ast \|\,\|x\|=1,$$
	consequently,
	$$ \|y_\kappa^\ast \|\,\|x_\kappa\|+\|z^\ast \|\,\|w\|-\delta \leq 1-\delta < \re y_\kappa^\ast (G_\kappa x_\kappa)+\|z^\ast \| \, \|w\| $$
	and so $\re y_\kappa^\ast (G_\kappa x_\kappa)>\|y_\kappa^\ast \| \, \|x_\kappa\|-\delta$.
	\\
	Since $$\frac{v_{G,\delta}(T)}{2}< |y^\ast (Tx)|=|y_\kappa^\ast (Sx_{\kappa})|\leq\|y_\kappa^\ast \| \, \|	x_\kappa\|,$$
	we deduce that
	$$ \re \dfrac{y_\kappa^\ast }{\|y_\kappa^\ast \|}\left(G_1\dfrac{x_\kappa}{\|x_\kappa\|}\right)>1-\frac{\delta}{\|y_\kappa^\ast \|\,\|x_\kappa\|}>1-\hat{\delta}.$$
	Then,
	\begin{align*}
	v_{G,\delta}(T)-\eps<\left|y^\ast (Tx)\right|=\left|y_\kappa^\ast (Sx_\kappa)\right|\leq \left|\frac{y_\kappa^\ast }{\|y_\kappa\|}\left(S\frac{x_\kappa}{\|x_\kappa\|} \right)\right| \leq v_{G_{\kappa},\hat{\delta}}(S),
	\end{align*}
	hence $ v_{G}(T)\leq v_{G_\kappa}(S).$
	
	To prove the reverse inequality, we fix $\delta>0$ and $x_\kappa \in S_{X_\kappa}$, $y_\kappa^\ast \in S_{Y_\kappa^\ast }$ with $\re y_\kappa^\ast (G_\kappa x_\kappa)>1-\delta$, and define $x=(x_\kappa,0)\in S_X$ and $y^\ast =(y_\kappa^\ast ,0)\in S_{Y^\ast }$. We clearly have that
	$$
\re y^\ast (Gx)>1-\delta \quad \text{and} \quad |y_{\kappa}^\ast (Sx_\kappa)|=|y^\ast (Tx)|\leq v_{G,\delta}(T).
$$
	Consequently, $v_{G_\kappa,\delta}(S)\leq v_{G,\delta}(T)$ and the claim follows moving $\delta\downarrow 0$.
	
	To sum up, we have obtained that given $S\in \mathcal{L}(X_\kappa,Y_\kappa)$ there is $T\in \mathcal{L}(X,Y)$ with $\|T\|=\|S\|$ and $v_G(T)=v_{G_\kappa}(S)$, so it follows that
	$$n_G(X,Y)\|S\|=n_G(X,Y) \|T\| \leq v_G(T) =v_{G_\kappa}(S)$$ and the arbitrariness of $S\in \mathcal{L}(X_\kappa,Y_\kappa)$ gives that $n_G(X,Y)\leq n_{G_\kappa}(X_\kappa,Y_\kappa)$.
	
	We prove now the reverse inequalities when $E$ is $c_0$ or
	$\ell_\infty$. In both cases, an operator $T\in \mathcal{L}(X,Y)$ can be seen as a family $(T_\lambda)_{\lambda\in\Lambda}$, where
	$T_\lambda\in \mathcal{L}(X,Y_\lambda)$ for every $\lambda$, and $\|T\|=\sup \{\|T_\lambda\|\colon \lambda\in \Lambda\}$. Given $\eps>0$, we may find $\kappa\in \Lambda$ such that
	$\|T_\kappa\|>\|T\|-\eps$, and write $X=X_\kappa\oplus_\infty W$ where $W=\left[\bigoplus_{\lambda\neq\kappa} X_\lambda\right]_E$.
	Since $B_X$ is the convex hull of $S_{X_\kappa} \times S_W$, we may find $x_0\in S_{X_\kappa}$ and $w_0\in S_W$ such that
	$$\|T_\kappa(x_0,w_0)\| > \|T\|-\eps.$$
	Now, fix $x_0^\ast \in S_{X_\kappa^\ast }$ with $x_0^\ast (x_0)=1$ and define the operator  $S\in \mathcal{L}(X_\kappa,Y_\kappa)$ by
	$$ S x = T_\kappa(x,0) +
	x_0^\ast (x)\,T_\kappa(0,w_0)=T_\kappa\big(x,x_0^\ast (x) w_0\big) \qquad
	(x\in X_\kappa)$$
	which satisfies
	$$
	\|S\|\geq \|S x_0\| =\big\|T_\kappa\big(x_0, x_0^\ast (x_0) w_0\big)\big\|= \|T_\kappa(x_0, w_0 )\|> \|T\|-\eps.
	$$
	Given $\delta>0$, we claim that $v_{G_\kappa,\delta}(S)\leq v_{G,\delta}(T)$. Indeed, we may find $u\in S_{X_\kappa}$ and $v^\ast \in S_{Y_\kappa^\ast }$ with  $\re v^\ast (G_{\lambda_0}u)>1-\delta$. Now, we write
	$$
	x=\big(u, x_0^\ast (u)w_0\big)\in S_X, \qquad y^\ast =(v^\ast ,0) \in S_{Y^\ast }
	$$
	which satisfy $\re y^\ast (Gx)=\re v^\ast (G_\kappa u)>1-\delta$, hence
	\begin{align*}
	|v^\ast (Su)|=
	\left|v^\ast \left[T_\kappa\big(u,x_0^\ast (u)w_0\big)\right]\right|=|y^\ast (Tx)|\leq v_{G,\delta}(T).
	\end{align*}
	Then, we deduce that $v_{G_\kappa, \delta}(S)\leq v_{G,\delta}(T)$. From this, we get that
	$$
	v_G(T)\geq v_{G_\kappa}(S)\geq n_{G_\kappa}(X_\kappa,Y_\kappa) \|S\| \geq
	n_{G_\kappa}(X_\kappa,Y_\kappa)\big[\|T\|-\eps\big].
	$$
	Therefore, it  follows that
	$$
	v_G(T)\geq \inf_\lambda
	n_{G_\lambda}(X_\lambda,Y_\lambda)\|T\|
	$$
	and so $n_G(X,Y)\geq \inf_\lambda
	n_{G_\lambda}(X_\lambda,Y_\lambda)$ as required.
	
	Suppose now that $E=\ell_1$. In this case, we can write every operator $T\in \mathcal{L}(X,Y)$ as a family $(T_\lambda)_{\lambda\in \Lambda}$ of operators where $T_\lambda\in \mathcal{L}(X_\lambda,Y)$ for every $\lambda\in\Lambda$, and
	$\|T\|=\sup \{ \|T_\lambda\|\colon \lambda \in \Lambda\}$. Given
	$\eps>0$, find $\kappa\in \Lambda$ such that
	$\|T_\kappa\|>\|T\|-\eps$, and write $X= X_\kappa \oplus_1 W$, $Y=Y_\kappa\oplus_1 Z$, and $T_\kappa=(A,B)$ where $W=\left[\bigoplus_{\lambda\neq\kappa} X_\lambda\right]_{\ell_1}$, $Z=\left[\bigoplus_{\lambda\neq\kappa} Y_\lambda\right]_{\ell_1}$, $A\in \mathcal{L}(X_\kappa,Y_\kappa)$ and $B\in \mathcal{L}(X_\kappa,Z)$. Now, we choose $x_0\in S_{X_\kappa}$ such that
	$$\|T_\kappa x_0\|= \|A x_0\| + \|B x_0\| > \|T\| -\eps,$$
	we find $a_0\in S_{Y_\kappa^\ast }$ and $z^\ast \in S_{Z^\ast }$ satisfying
	$$A x_0 = \|Ax_0\| a_0  \qquad
	\text{and} \qquad z^\ast (Bx_0)=\|Bx_0\|,$$
	and define an operator
	$S\in \mathcal{L}(X_\kappa,Y_\kappa)$ by
	$$
S x = Ax + \left[z^\ast (B x)\right]a_0 \qquad (x\in X_\kappa).
$$
	Then
	$$\|S\|\geq \|S x_0\| = \Bigl\|Ax_0 +
	z^\ast (Bx_0) a_0\Bigr\| = \|Ax_0\|+ \|Bx_0\|> \|T\|-\eps.
$$
	Given $\delta>0$, we prove that $v_{G_\kappa,\delta}(S)\leq v_{G,\delta}(T)$. To do so, fixed $u\in S_{X_\kappa}$ and $v^\ast \in S_{Y_\kappa^\ast }$ with  $\re v^\ast (G_\kappa u)>1-\delta$, we define
	$$
	x=\big(u, x_0^\ast (u)w_0\big)\in S_X \qquad \text{and} \qquad y^\ast =(v^\ast ,0) \in S_{Y^\ast }.
	$$
	Since $\re y^\ast (Gx)=\re v^\ast (G_\kappa u)>1-\delta$, we get that
	$$
	|v^\ast (Su)|=|v^\ast (Au) + v^\ast (a_0)z^\ast (Bu)|=|y^\ast (T_\kappa u)|=|y^\ast (Tx)|\leq v_{G,\delta}(T)
	$$
	which gives $v_{G_\kappa,\delta}(S)\leq v_{G,\delta}(T)$ thanks to the arbitrariness of $u$ and $v^\ast $. Finally, we can write
	$$
	v_G(T)\geq v_{G_\kappa}(S)\geq n_{G_\kappa}(X_\kappa,Y_\kappa) \|S\| \geq
	n_{G_\kappa}(X_\kappa,Y_\kappa)\big[\|T\|-\eps\big]
	$$
	and so we deduce that $ v_G(T)\geq \inf_\lambda	n_{G_\lambda}(X_\lambda,Y_\lambda)\|T\|$, from which the desired inequality $n_G(X,Y)\geq \inf_\lambda
	n_{G_\lambda}(X_\lambda,Y_\lambda)$ follows.
\end{proof}

Let us observe that the first part of the above proof is valid for general absolute sums.

\begin{proposition}
Let $X_1$, $X_2$, $Y_1$, $Y_2$ be Banach spaces and let $E$ be $\R^2$ endowed with an absolute norm. Given norm-one operators $G_i\in \mathcal{L}(X_i,Y_i)$ for $i=1,2$, we define $G\in \mathcal{L}\bigl(X_1\oplus_E X_2,Y_1\oplus_E Y_2\bigr)$ by $$G(x_1,x_2)=(G_1x_1,G_2x_2)\in Y_1\oplus_E Y_2$$
for every $(x_1,x_2)\in X_1\oplus_E X_2$. Then,
$$
n_G(X_1\oplus_E X_2,Y_1\oplus_E Y_2)\leq \min \bigl\{ n_{G_1}(X_1,Y_1), n_{G_2}(X_2,Y_2) \bigr\} .
$$
\end{proposition}

The associativity of $\ell_p$-sums allows us to get the following result from the above one.

\begin{corollary}
	Let $\{X_\lambda\colon \lambda\in
	\Lambda\}$, $\{Y_\lambda\colon \lambda\in
	\Lambda\}$ be two families of Banach spaces, let $G_\lambda\in \mathcal{L}(X_\lambda,Y_\lambda)$ be a norm-one operator for every $\lambda\in\Lambda$, let $1<p<\infty$, and let $X=\left[\bigoplus_{\lambda\in\Lambda} X_\lambda\right]_{\ell_p}$ and $Y=\left[\bigoplus_{\lambda\in\Lambda} Y_\lambda\right]_{\ell_p}$. We define the operator $G\colon X\longrightarrow Y$ by
	$$ G\left[(x_\lambda)_{\lambda\in\Lambda}\right]=(G_\lambda x_\lambda)_{\lambda\in\Lambda} $$
	for every $(x_\lambda)_{\lambda\in\Lambda}\in \left[\oplus_{\lambda\in\Lambda} X_\lambda\right]_{\ell_p}$. Then
	$$n_G(X,Y)\leq \inf_{\lambda} \, n_{G_\lambda}(X_\lambda,Y_\lambda).$$
\end{corollary}

The main application of Proposition \ref{sumas} is the following important example.

\begin{theorem}\label{theorem-allvaluesinX}
In both the real and the complex case, there exist Banach spaces $X$ such that
$$
\mathcal{N}\bigl(\mathcal{L}(X)\bigr)=[0,1].
$$
\end{theorem}

The proof will follow immediately from Proposition \ref{sumas} and the next example.

\begin{example}\label{example:foreverygamma-nG=gamma}
{\slshape For every $\gamma\in [0,1]$ there is a real or complex Banach space $Y_\gamma$ such that there exist norm-one operators $G_{\gamma,1},G_{\gamma,2}\in \mathcal{L}(Y_\gamma)$ satisfying $n_{G_{\gamma,1}}(Y_\gamma,Y_\gamma)=\gamma$ and $n_{G_{\gamma,2}}(Y_\gamma,Y_\gamma)=1$.}
\end{example}

\begin{proof}
We start by showing the existence of a real or complex space $Z_\gamma$ such that there exists a norm-one operator $G\in \mathcal{L}(Z_\gamma)$ satisfying $n_G(Z_\gamma,Z_\gamma)=\gamma$. For $\gamma\in[\frac{1}{2},1]$, it is enough to use the fact that the set $\{n(W) \colon W \text{ two-dimensional space}\}$ covers the interval $[0,1]$ in the real case and $[\frac{1}{\e},1]$ in the complex case \cite[Theorems 3.5 and 3.6]{D-Mc-P-W}. So, for $\gamma\in[\frac{1}{2},1]$ there is a two-dimensional (real or complex) space $Z_\gamma$ satisfying $n(Z_\gamma)=\gamma$ and it suffices to take $G=\Id_{Z_\gamma}$.
	
For $\gamma\in[0,\frac{1}{2}]$, let $X_\gamma=\K^2$ endowed with the norm
$$
\|(x_1,x_2)\|_\gamma=\max\{|x_2|, |x_1|+(1-\gamma)|x_2|\} \qquad \big((x_1,x_2)\in \K^2\big),
$$
let $Z=\ell_\infty^2$, and let $Z_\gamma=X_\gamma\oplus_\infty Z$. Take $x_0^\ast =(0,1)\in S_{X_\gamma^\ast }$, $z_0=(1,1)\in S_Z$, $x_0=(1,0)\in S_{X\gamma}$, $z_0^\ast=(1,0)\in S_{Z^\ast}$ and define $J_1=x_0^\ast \otimes z_0$, $J_2=z_0^\ast\otimes x_0$, and $G=(J_1,J_2)$. Let us show that $n_G(Z_\gamma,Z_\gamma)=\gamma$.

Observe first that $X_\gamma^\ast $ is $\K^2$ endowed with the norm	
$$
\|(x^\ast _1,x^\ast _2)\|=\max\{|x^\ast _1|, \gamma|x^\ast _1|+|x^\ast _2|\} \qquad \big((x_1,x_2)\in \K^2\big).
$$
Since $\|J_1\|=\|J_2\|=1$ and $z_0\in Z$, $z_0^\ast\in Z^\ast$ are spear vectors, using Proposition~\ref{sumas} and  Proposition~\ref{prop:num-index-rankone-operators} we have that
\begin{align*}
n_G(Z_\gamma,Z_\gamma)&=\min\{n_{J_1}(X_\gamma,Z), n_{J_2}(Z,X_\gamma)\}\\
&=\min\{n(X_\gamma^\ast,x_0^\ast)n(Z,z_0), n(Z^\ast,z_0^\ast)n(X_\gamma,x_0)\}=\min\{n(X_\gamma^\ast,x_0^\ast), n(X_\gamma,x_0)\}.
\end{align*}
So it suffices to show that $n(X_\gamma^\ast ,x_0^\ast )=\gamma$ and $n(X_\gamma,x_0)\geq 1-\gamma$. To do so, we fix $x^\ast =(x_1^\ast ,x_2^\ast )\in S_{X_\gamma^\ast }$ and we compute $v(X_\gamma^\ast ,x_0^\ast ,x^\ast )$. The points $x\in S_{X_\gamma}$ satisfying $x_0^\ast (x)=1$ are of the form $x=(t\theta, 1)$ with $t\in[0,\gamma]$ and $\theta \in \T$. Thus we have that
\begin{align*}
v(X_\gamma^\ast ,x_0^\ast ,x^\ast )=\sup\{|t\theta x_1^\ast +x_2^\ast |\colon t\in[0,\gamma], \theta \in \T\}=\gamma|x_1^\ast |+|x_2^\ast |\geq \gamma \|x^\ast \|
\end{align*}
which implies $n(X_\gamma^\ast ,x_0^\ast )\geq \gamma$. Finally, for $x^\ast =(1,0)\in S_{X_\gamma^\ast }$ it is clear that $v(X_\gamma^\ast ,x_0^\ast ,x^\ast )=\gamma$ and so $n(X_\gamma^\ast ,x_0^\ast )= \gamma$ as desired.

To prove $n(X_\gamma,x_0)\geq 1-\gamma$, we fix $x =(x_1 ,x_2)\in S_{X_\gamma}$ and we estimate $v(X_\gamma,x_0 ,x)$. The points $x^\ast\in S_{X_\gamma^\ast}$ satisfying $x^\ast (x_0)=1$ are of the form $x^\ast=(1,t\theta)$ with $t\in[0,1-\gamma]$ and $\theta \in \T$. Thus we have that
\begin{align*}
v(X_\gamma,x_0 ,x)=\sup\{|x_1 +t\theta x_2 |\colon t\in[0,1-\gamma],\, \theta \in \T\}=|x_1|+(1-\gamma)|x_2|\geq (1-\gamma) \|x \|
\end{align*}
which implies $n(X_\gamma ,x_0 )\geq 1-\gamma$.	This finishes the proof of the existence of $Z_\gamma$.

Now, for each $\gamma \in [0,1]$, we take $Y_\gamma =(Z_\gamma\oplus_\infty \K)\oplus_1 \K$. On the one hand, define $G_{\gamma,1}\in \mathcal{L}(Y_\gamma)$ by
$$
G_{\gamma,1}(z,\alpha,\beta)=(Gz,\alpha, \beta) \qquad (z\in Z_\gamma, \alpha,\beta\in \K)
$$
which satisfies $n_{G_{\gamma,1}}(Y_\gamma,Y_\gamma)=n_G(Z_\gamma,Z_\gamma)=\gamma$ by Proposition~\ref{sumas}. On the other hand, observe that $Y_\gamma^\ast=(Z_\gamma^\ast\oplus_1 \K)\oplus_\infty \K$, so the elements $y=(0,0,1)\in S_{Y_\gamma}$ and $y^\ast=(0,1,1)\in S_{Y^\ast_\gamma}$ are spear vectors in $Y_\gamma$ and $Y^\ast_\gamma$ respectively. Therefore, the norm-one operator $G_{\gamma,2}=y^\ast \otimes y \in \mathcal{L}(Y_\gamma)$ satisfies $n_{G_{\gamma,2}}(Y_\gamma,Y_\gamma) =1$ by Proposition~\ref{prop:num-index-rankone-operators}.
\end{proof}	

We are now able to provide with the pending proof.

\begin{proof}[Proof of Theorem \ref{theorem-allvaluesinX}]
For each $\gamma\in [0,1]$, consider the space $Y_\gamma$ given in Example \ref{example:foreverygamma-nG=gamma} and consider the norm-one operators $G_{\gamma,1},G_{\gamma,2}\in \mathcal{L}(Y_\gamma)$ satisfying $n_{G_{\gamma,1}}(Y_\gamma,Y_\gamma)=\gamma$ and $n_{G_{\gamma,2}}(Y_\gamma,Y_\gamma)=1$.
Now, let $X=\left[\bigoplus\nolimits_{\gamma\in [0,1]} Y_\gamma\right]_{c_0}$, and for every $\xi\in [0,1]$, consider the norm-one operator $G^\xi\in \mathcal{L}(X)$ to be the diagonal operator given by $[G^\xi]_\gamma=G_{\gamma,2}$ if $\gamma\neq \xi$ and $[G^\xi]_\xi=G_{\xi,1}$. By Proposition \ref{sumas}, it follows that
$n_{G^\xi}(X,X)=\xi$, finishing the proof.
\end{proof}

\subsection{Composition operators on vector-valued function spaces}
The first result here gives the numerical index with respect to composition operators between spaces of vector-valued continuous functions.

\begin{proposition}
	Let $X$, $Y$ be Banach spaces, let $K$ be a compact Hausdorff topological space and $G\in \mathcal{L}(X,Y)$ be a norm-one operator. Consider the norm-one composition operator $\widetilde{G}\colon C(K,X) \longrightarrow C(K,Y)$ given by $ \widetilde{G}(f)=G\circ f$ for every $f\in C(K,X)$. Then
	$$n_{\widetilde{G}}\big(C(K,X),C(K,Y)\big)=n_G(X,Y).$$
\end{proposition}

\begin{proof}
We follow the lines of \cite[Theorem~5]{M-P}. 	To show that $n_{\widetilde{G}}\big(C(K,X),C(K,Y)\big)\geq n_G(X,Y)$, we fix $T\in \mathcal{L}\big(C(K,X),C(K,Y)\big)$ with $\|T\|=1$ and prove that $v_{\widetilde{G}}(T)\geq n_G(X,Y)$. Given $\eps>0$, we may find $f_0\in C(K,X)$ with $\|f_0\|=1$ and $t_0\in K$ such that
	\begin{equation}\label{cdek:eq1}
	\|[Tf_0](t_0)\|>1-\eps.
	\end{equation}
	Define $z_0=f_0(t_0)$ and find a continuous function $\varphi:K\longrightarrow[0,1]$ such that $\varphi(t_0)=1$ and $\varphi(t)=0$ if
	$\|f_0(t)-z_0\|\geq \eps$.
	Now write $z_0=(1-\lambda)x_1+\lambda x_2$ with $0\leq\lambda\leq1$, $x_1,x_2\in S_X$, and consider the functions
	$$f_j=(1-\varphi)f_0 + \varphi x_j \in C(K,X) \qquad (j=1,2).$$
	Then, $\|\varphi f_0-\varphi z_0\|<\eps$ meaning that
	$$\big\|f_0-((1-\lambda)f_1+\lambda f_2)\big\|<\eps,$$
	and, using (\ref{cdek:eq1}), we must have
	$$\|[Tf_1](t_0)\|>1-2\eps \qquad \text{or} \qquad \|[T
	f_2](t_0)\|>1-2\eps.$$
	By making the right choice of $x_0=x_1$ or $x_0=x_2$, we get $x_0\in S_X$ such that
	\begin{equation}\label{cdek:eq2}
	\Bigl\|\Bigl[T\big((1-\varphi)f_0 + \varphi x_0\big)\Bigr](t_0)\Bigr\|
	>1-2\eps.
	\end{equation}
	Next, we fix $x_0^\ast \in S_{X^\ast }$ with
	$x_0^\ast (x_0)=1$, denote
	$$\Phi(x)=x_0^\ast (x)(1-\varphi)f_0 + \varphi x \in C(K,X) \qquad (x\in X),$$
	and consider the operator $S\in \mathcal{L}(X,Y)$ given by
	\begin{equation*}
	S x =[T(\Phi(x))](t_0)
	\qquad (x\in X)
	\end{equation*}
	which, by (\ref{cdek:eq2}), obviously satisfies $\|S\|\geq \|S x_0\|>1-2\eps$.
	
	Now, given $\delta>0$, and $x\in S_X$, $y^\ast \in S_{Y^\ast }$ such that $\re y^\ast (Gx)>1-\delta$, we define $f\in S_{C(K,X)}$ by
	$ f=\Phi(x)$, and consider the functional $g^\ast \in S_{C(K,Y)^\ast }$ given by
	$$g^\ast (h)=[y^\ast \otimes \delta_{t_0}](h)=y^\ast (h(t_0)) \qquad \big(h\in C(K,Y)\big).$$
	Since $f(t_0)=x$, we have that $\re g^\ast \big(\widetilde{G}f\big)>1-\delta$ and
	\begin{align*}
	|y^\ast (Sx)|=
	\Big|y^\ast \left(\left[T\big(\Phi(x)\big)\right](t_0)\right)\Big|=\left|g^\ast (T f)\right| \leq v_{\widetilde{G},\delta}(T),
	\end{align*}
	hence $v_{G,\delta}(S)\leq v_{\widetilde{G},\delta}(T)$. Therefore,
	$$v_{\widetilde{G}}(T)\geq v_G(S) \geq n_G(X,Y) \|S\| \geq (1-2\eps)n_G(X,Y),$$
	and the arbitrariness of $\eps>0$ gives that $v_{\widetilde{G}}(T)\geq n_G(X,Y)$, as desired.
	
	To prove the reverse inequality, take $S\in \mathcal{L}(X,Y)$ and define $T \in \mathcal{L}\big(C(K,X),C(K,Y)\big)$ by $$[T(f)](t)=S(f(t)) \qquad \big(t\in K,\ f\in
	C(K,X)\big).$$ It is clear that $\|T\|=\|S\|$. To estimate the value of $v_{\widetilde{G}}(T)$ we use Lemma \ref{Lemm-radio-G-A-B} considering $A=S_{C(K,X)}$ and  $B=\{y^\ast \otimes\delta_t \colon y^\ast \in S_{Y^\ast}, \, t\in K \}$, where $(y^\ast \otimes\delta_t)(g)=y^\ast (g(t))$ for every $g\in C(K,Y)$ (as these subsets satisfy $\overline{\conv}(A)=B_{C(K,X)}$ and $\overline{\conv}^{w^\ast }(B)=B_{C(K,Y)^\ast }$). Now, for every $\delta>0$, $f\in S_{C(K,X)}$, $t\in K$, and $y^\ast \in S_{Y^\ast }$ satisfying that $\re y^\ast \big(G\big(f(t)\big)\big)>\nolinebreak1-\delta$, we call $x=f(t)\in S_X$ and observe that
	$\re y^\ast (Gx)>1-\delta$
	and
	\begin{align*}
	\left|y^\ast \big([Tf](t)\big)\right|=\left|y^\ast \left(S\big(f(t)\big)\right)\right|=|y^\ast (Sx)|\leq v_{G,\delta}(S).
	\end{align*}
	Consequently, $v_{\widetilde{G},\delta}(T)\leq v_{G,\delta}(S)$. It clearly follows that
	$$v_G(S)\geq v_{\widetilde{G}}(T) \geq n_{\widetilde{G}}\big(C(K,X),C(K,Y)\big) \|T\| =
	n_{\widetilde{G}}\big(C(K,X),C(K,Y)\big)\|S\|,$$
	so $n_G(X,Y)\geq
	n_{\widetilde{G}}\big(C(K,X),C(K,Y)\big)$, as desired.
\end{proof}

We next deal with K\"{o}the-Bochner vector-valued function spaces, for which we need to introduce some terminology.

Let $(\Omega,\Sigma,\mu)$ be a complete $\sigma$-finite measure
space. We denote by $L_0(\mu)$ the vector space of all (equivalent
classes modulo equality a.e.\ of) $\Sigma$-measurable locally
integrable real-valued functions on $\Omega$. A \emph{K\"{o}the
function space} is a linear subspace $E$ of $L_0(\mu)$ endowed
with a complete norm $\|\cdot\|_E$ satisfying the following
conditions:
\begin{enumerate}
\item[(i)] If $|f|\leq |g|$ a.e.\ on $\Omega$, $g\in E$ and
    $f\in L_0(\mu)$, then $f\in E$ and $\|f\|_E\leq \|g\|_E$.
\item[(ii)] For every $A\in \Sigma$ with $0<\mu(A)<\infty$,
    the characteristic function $\eins_A$ belongs to $E$.
\end{enumerate}
We refer the reader to the classical book by J.~Lindenstrauss and
L.~Tzafriri \cite{LTII} for more information and background on
K\"{o}the function spaces. Let us recall some useful facts about these
spaces which we will use in the sequel. First, $E$ is a Banach
lattice in the pointwise order. The \emph{K\"{o}the dual} $E'$ of $E$
is the function space defined as
$$
E'=\left\{g\in L_0(\mu)\ : \ \|g\|_{E'}:=
\sup_{f\in B_E}\int_\Omega |fg|\,d\mu<\infty \right\},
$$
which is again a K\"{o}the space on $(\Omega,\Sigma,\mu)$. Every
element $g\in E'$ defines naturally a continuous linear functional
on $E$ by the formula
$$
f\longmapsto \int_\Omega fg\, d\mu \qquad (f\in E),
$$
so we have $E'\subseteq E^\ast$ and this inclusion is isometric.

Let $E$ be a K\"{o}the space on a complete $\sigma$-finite measure space $(\Omega,\Sigma,\mu)$ and let $X$ be a real or complex
Banach space. A function $f:\Omega\longrightarrow X$ is said to be \emph{simple} if $f=\sum_{i=1}^n x_i\,\eins_{A_i}$ for some $x_1,\ldots,x_n\in X$ and some $A_1,\ldots,A_n\in \Sigma$. The function $f$ is said to be \emph{strongly measurable} if there is a sequence of simple functions $\{f_n\}$ with $\lim \|f_n(t)-f(t)\|_{X}=0$ for almost all $t\in \Omega$. For a strongly measurable function $f:\Omega \longrightarrow X$ we use notation $|f|$ for the function $|f|(\cdot)=\|f(\cdot)\|_X$. We write $E(X)$ to denote the space of (classes of) those strongly measurable functions $f:\Omega \longrightarrow X$ such that $|f| \in E$ and we endow $E(X)$ with the norm
$$
\|f\|_{E(X)}=\bigl\| |f| \bigr\|_E.
$$
Then $E(X)$ is a real or complex (depending on $X$) Banach space and it is called a \emph{K\"{o}the-Bochner function space}. We refer the reader to the book \cite{LinKothe} for background. For an element $f\in E(X)$  we consider a strongly measurable function $\widetilde{f}:\Omega \longrightarrow S_X$ such that $f=|f|\,\widetilde{f}$ a.e.\ and we observe that $\|f\|_{E(X)}=\|\,|f|\,\|_E$.

Our result for composition operators between K\"{o}the-Bochner function spaces is the following inequality.

\begin{proposition}\label{Prop:Edex}
	Let $X$, $Y$ be Banach spaces, let $(\Omega,\Sigma,\mu)$ be a $\sigma$-finite measure space, let $E$ be a K\"{o}the space on $(\Omega,\Sigma,\mu)$ such that $E'$ is norming for $E$, and let $G\in \mathcal{L}(X,Y)$ be a norm-one operator. Consider the norm-one composition operator $\widetilde{G} \colon E(X)\longrightarrow E(Y)$ given by
	$ \widetilde{G}(f)=G\circ f$ for every $f\in E(X)$. Then
$$
n_{\widetilde{G}}\big(E(X),E(Y)\big)\leq n_G(X,Y).
$$
\end{proposition}

We need a preliminary lemma which is considered folklore in the theory of K\"{o}the-Bochner spaces. As we have not found direct references, we will include a short sketch of its proof. Let us introduce some notation. Let $E$ be a K\"{o}the space on a $\sigma$-finite measure space $(\Omega,\Sigma,\mu)$ and let $Y$ be a Banach space. If $\Phi\colon \Omega \longrightarrow Y^\ast$ belongs to $E'(Y^\ast)$, then the \emph{integral} functional on $E(Y)$ defined by $\Phi$ is given by
      \begin{equation}\label{eq:Kothe-def-integral-functional}
      \langle \Phi,f\rangle = \int_\Omega \langle \Phi(t),f(t)\rangle\,d\mu(t) \qquad \bigl(f\in E(Y)\bigr).
      \end{equation}
Observe that we have already presented the definition of the functions $|\Phi|=\|\Phi(\cdot)\|_{Y^\ast}\in E'$ and $\widetilde{\Phi} \colon\Omega \longrightarrow S_{Y^\ast}$ which satisfy that $\Phi=|\Phi|\widetilde{\Phi}$ a.e. Let us comment that it is possible to define integral functionals as in \eqref{eq:Kothe-def-integral-functional} for functions satisfying weaker requirements but, actually, here we are only interested in those integral functionals coming from functions $\Phi$ in $E'(Y^\ast)$ having countably many values.

\begin{lemma}\label{lemma:propertiesE(X)}
Let $E$ be a K\"{o}the space on a $\sigma$-finite measure space $(\Omega,\Sigma,\mu)$ and let $Y$ be a Banach space.
\begin{enumerate}
  \item[(a)] The set of measurable functions from $\Omega$ to $Y$ having countably many values is dense in $E(Y)$.
  \item[(b)] If $\Phi\in E'(Y^\ast)$ has countably many values, then the integral functional defined as in \eqref{eq:Kothe-def-integral-functional} belongs to $E(Y)^\ast$ and satisfies that $\|\Phi\|_{E(Y)^\ast}= \|\Phi\|_{E'(Y^\ast)} =\|\,|\Phi|\,\|_{E'}$.
  \item[(c)] If $E'$ is norming for $E$, then the set $\mathcal{B}$ of all  integral functionals defined by norm-one functions in $E'(Y^\ast)$ having countably many values, satisfies that $\overline{\conv}^{w^\ast}(\mathcal{B}) =B_{E(Y)^\ast}$.
\end{enumerate}
\end{lemma}

\begin{proof}[Sketch of the proof]
(a). Fix $f\in E(Y)$ and $\eps>0$. We consider a partition of $\Omega$ into countably many pairwise disjoint measurable sets $\Omega=\bigcup_{n\in\N\cup\{0\}} \Omega_n$ with $\mu(\Omega_0)=0$, $0<\mu(\Omega_n)<\infty$ for all $n\in \N$, and such that $f(\Omega_n)$ is separable for all $n\in \N$. Now, for every $n\in \N$ we use the Bochner measurability of $f\eins_{\Omega_n}$ to find a measurable function $g_n\colon \Omega \longrightarrow Y$ with $g_n(\Omega\setminus \Omega_n)=\{0\}$, having countably many values and satisfying that
$$
\|f(t)-g_n(t)\|\leq \frac{\eps}{2^n\|\eins_{\Omega_n}\|} \qquad \bigl(t\in \Omega_n\bigr)
$$
(see \cite[Corollary 3 in p.~42]{DieUhl}, for instance). Observe that $$|f\eins_{\Omega_n} -g_n|\leq \frac{\eps\eins_{\Omega_n}}{2^n\|\eins_{\Omega_n}\|},$$ so $g_n\in E(Y)$ and
$\|f\eins_{\Omega_n} - g_n\|\leq \frac{\eps}{2^n}$.
It is now clear that the sum $g$ of the (formal) series $\sum_{n\geq 1} g_n$ belongs to $E(Y)$, has countably many values, and satisfies $\|f-g\|\leq \eps$.

(b). Observe that $\Phi$ has the form
$$
\Phi(t)=\sum_{n=1}^\infty y_n^\ast\eins_{A_n}(t) \qquad \bigl(t\in \Omega\bigr)
$$
for suitable sequences $\{y_n^\ast\}_{n\in \N}$ of elements of $Y^\ast$ and $\{A_n\}_{n\in \N}$ of pairwise disjoint elements of $\Sigma$ satisfying that the scalar function $$t\longmapsto \sum_{n=1}^\infty \|y_n^\ast\|\eins_{A_n}(t)$$ belongs to $E'$. With this in mind, $\Phi$ acts on $E(Y)$ as
$$
\langle \Phi,f\rangle =\int_\Omega \langle \Phi(t),f(t)\rangle\,d\mu(t)=\sum_{n=1}^\infty \int_{A_n} y_n^*(f(t))\,d\mu(t) \qquad \bigl(f\in E(Y)\bigr).
$$
It is now routine to show that $\Phi\in E(Y)^\ast$ and that $\|\Phi\|_{E(Y)^\ast}=\|\,|\Phi|\,\|_{E'}$.

Assertion (c) follows routinely from the density in $E(Y)$ of the set of countably-valued functions, from the fact that $E'$ is norming for $E$, and from the density in $E'$ of the set of countably-valued functions.
\end{proof}

\begin{proof}[Proof of Proposition \ref{Prop:Edex}]
We follow the lines of \cite[Theorem 4.1]{MarMerPopRan}. Take an operator $S\in \mathcal{L}(X,Y)$ with $\|S\|=1$, and define $T\in \mathcal{L}\big(E(X),E(Y)\big)$ by
	$$
[T(f)](t)=S(f(t)) =|f|(t) \, S(\widetilde{f}(t)) \qquad \big(t\in\Omega, \ f\in E(X)\big).
    $$
	We claim that $T$ is well defined and $\|T\|=1$. Indeed, for $f\in E(X)$, $T(f)$ is strongly measurable and
	$$\|[T(f)](t)\|_Y=|f|(t)\, \|S(\widetilde{f}(t))\|\leq|f|(t) \qquad (t\in\Omega),$$
	so $T(f)\in E(Y)$ with $\|T(f)\|_{E(Y)}\leq \| \,|f|\, \|_E=\|f\|_E(X)$. This gives $\|T\|\leq 1$. Conversely, we fix $A\in \Sigma$ such that $0<\mu(A)<\infty$ and for each $x\in S_X$ consider $f=\|\eins_A\|_E^{-1}x\,\eins_A\in S_{E(X)}$. Then, $\|f\|=1$ and
	$$
	\|[T(f)](t)\|_Y=\dfrac{\eins_A(t) \|S(x)\|_Y}{\|\eins_A\|_E},
	$$
	so
	$$
	\|T\|\geq\|T(f)\|_{E(Y)}=\left\|\dfrac{\eins_A \, \|S(x)\|_Y}{\|\eins_A\|_E} \right\|_E\geq\|S(x)\|_Y.
	$$
	By just taking supremum on $x\in S_X$, we get $\|T\|\geq\|S\|=1$ as desired.
	
	Next, we fix $0<\delta<1$,  $f=|f|\,\widetilde{f}\in S_{E(X)}$ and $\Phi=|\Phi|\,\widetilde{\Phi}\in\mathcal{B}$ satisfying $\re\langle\Phi,\widetilde{G}(f)\rangle>1-\delta$, where $\mathcal{B}\subset E(Y)^\ast$ is the set given in Lemma \ref{lemma:propertiesE(X)}.c. Let $0<\alpha<1$ be such that $1-\alpha =\sqrt{\delta}$ and write
	$$
	 \Omega_1=\left\{t\in\Omega \colon \re\langle\widetilde{\Phi}(t),G(\widetilde{f}(t))\rangle < \alpha \right\}\quad \text{and} \quad \Omega_2=\left\{t\in\Omega \colon \re\langle\widetilde{\Phi}(t),G(\widetilde{f}(t))\rangle > \alpha \right\}.
	$$
	Then,
	\begin{align*}
	1-\delta&< \re\langle\Phi,\widetilde{G}(f)\rangle= \re\int_\Omega |\Phi|(t) \, |f|(t)\,  \langle \widetilde{\Phi}(t),G(\widetilde{f}(t))\rangle \, d\mu(t) \\
	&=\re\int_{\Omega_1} |\Phi|(t) \, |f|(t)\,  \langle \widetilde{\Phi}(t),G(\widetilde{f}(t))\rangle \, d\mu(t)+\re\int_{\Omega_2} |\Phi|(t) \, |f|(t)\,  \langle \widetilde{\Phi}(t),G(\widetilde{f}(t))\rangle \, d\mu(t) \\
	&\leq \alpha \, \int_{\Omega_1} |\Phi|(t) \, |f|(t)\, d\mu(t) + \int_{\Omega_2} |\Phi|(t) \, |f|(t)\, d\mu(t) \\
	& \leq \alpha \, \int_{\Omega_1} |\Phi|(t) \, |f|(t)\, d\mu(t)+1-\int_{\Omega_1} |\Phi|(t) \, |f|(t)\, d\mu(t),
	\end{align*}
	hence $\displaystyle \int_{\Omega_1} |\Phi|(t) \, |f|(t)\, d\mu(t)<\frac{\delta}{1-\alpha}$. Moreover,
	\begin{align*}
	|\langle\Phi,Tf\rangle|&= \left|\int_\Omega |\Phi|(t) \, |f|(t) \, \langle\widetilde{\Phi}(t),S(\widetilde{f}(t))\rangle\, d\mu(t) \right| \\
	& \leq \int_{\Omega_2} |\Phi|(t) \, |f|(t)\, v_{G,1-\alpha}(S)\, d\mu(t)+\int_{\Omega_1}|\Phi|(t) \, |f|(t)\, \left|\langle\widetilde{\Phi}(t),S(\widetilde{f}(t))\rangle\right|\, d\mu(t)\\
	&\leq v_{G,1-\alpha}(S)+\dfrac{\delta}{1-\alpha}=v_{G,\sqrt{\delta}}(S)+\sqrt{\delta}.
	\end{align*}
	Thus, we get that $v_{\widetilde{G},\delta}(T)\leq v_{G,\sqrt{\delta}}(S)+\sqrt{\delta}$ by Lemma \ref{Lemm-radio-G-A-B} (using Lemma \ref{lemma:propertiesE(X)}.c). So, taking infimum with $0<\delta<1$,  we obtain that	$n_{\widetilde{G}}\big(E(X),E(Y)\big)\leq v_{\widetilde{G}}(T)\leq v_G(S)$ and the desired inequality follows.
\end{proof}

Let us comment that, in the case $X=Y$ and $G=\Id_X$, the above result improves \cite[Theorem~4.1]{MarMerPopRan}. We state the new result here.

\begin{corollary}[\textrm{Extension of \cite[Theorem~4.1]{MarMerPopRan}}]
Let $X$ be a Banach space, let $(\Omega,\Sigma,\mu)$ be a $\sigma$-finite measure space, and let $E$ be a K\"{o}the space on $(\Omega,\Sigma,\mu)$ such that $E'$ is norming for $E$. Then
$$
n\bigl(E(X)\bigr)\leq n(X).
$$
\end{corollary}

There are K\"{o}the spaces which do not satisfy the norming requirement of Proposition \ref{Prop:Edex}, see \cite[Remark~1 in page 30]{LTII} for instance. We next present some particular cases in which the previous proposition applies. First, we deal with order continuous spaces. We say that a K\"othe space $E$ is \emph{order continuous} if $0\leq x_\alpha\downarrow0$ and $x_\alpha\in E$ imply that $\lim\|x_\alpha\|=0$ (this is known to be equivalent to the fact that $E$ does not contain an isomorphic copy of $\ell_{\infty}$). If $E$ is order continuous, then $E'=E^\ast$  (see \cite[p.~169]{LinKothe} or \cite[p.~29]{LTII}).

\begin{corollary}\label{Cor:Edex-order-cont}
	Let $X, Y$ be Banach spaces, let $(\Omega,\Sigma,\mu)$ be a probability measure space, let $E$ be an order continuous K\"{o}the space on $(\Omega,\Sigma,\mu)$, and let $G\in \mathcal{L}(X,Y)$ be a norm-one operator.  Consider the norm-one composition operator $\widetilde{G} \colon E(X)\longrightarrow E(Y)$ given by
	$ \widetilde{G}(f)=G\circ f$ for every $f\in E(X)$. Then
	$$ n_{\widetilde{G}}\big(E(X),E(Y)\big)\leq n_G(X,Y). $$
\end{corollary}

For $1\leq p<\infty$, $L_p$-spaces over $\sigma$-finite measures are order continuous K\"othe spaces; for $p=\infty$, this is not longer true, but $L_\infty(\mu)'$ is norming for $L_\infty(\mu)$, see \cite[Remark 1 in page 30]{LTII} for instance. Therefore, we get the following consequence:

\begin{corollary}\label{corollary:composition_inequality_ellp}
Let $X$, $Y$ be Banach spaces, let $(\Omega,\Sigma,\mu)$ be a $\sigma$-finite measure space, let $1\leq p\leq\infty$, and let $G\in \mathcal{L}(X,Y)$ be a norm-one operator. Consider the norm-one composition operator $\widetilde{G} \colon L_p(\mu,X)\longrightarrow L_p(\mu,Y)$ given by
$\widetilde{G}(f)=G\circ f$ for every $f\in L_p(\mu,X)$. Then
$$
n_{\widetilde{G}}\big(L_p(\mu,X),L_p(\mu,Y)\big)\leq n_G(X,Y).
 $$
\end{corollary}

The equality does not hold in general, as for $p\neq 1,\infty$, we have that $n(\ell_p^2)<1$. On the other hand, we will show that the equality holds for the cases $p=1$ and $p=\infty$.

We start dealing with spaces of Bochner integrable functions.

\begin{proposition}\label{L1-general}
	Let $X$, $Y$ be Banach spaces, let $(\Omega,\Sigma,\mu)$ be a $\sigma$-finite measure space, and let $G\in \mathcal{L}(X,Y)$ be a norm-one operator. Consider the norm-one composition operator $\widetilde{G} \colon L_1(\mu,X)\longrightarrow L_1(\mu,Y)$ given by
	$ \widetilde{G}(f)=G\circ f$ for every $f\in L_1(\mu,X)$. Then
	$$ n_{\widetilde{G}}\big(L_1(\mu,X),L_1(\mu,Y)\big)=n_G(X,Y). $$
\end{proposition}

\begin{proof} We follow the lines of \cite[Theorem 8]{M-P}.
	We can assume without loss of generality that $(\Omega,\Sigma,\mu)$ is a probability space, as vector-valued $L_1$-spaces associated to $\sigma$-finite measures are (up to an isometric isomorphism) vector-valued $L_1$-spaces associated to probability measures, see \cite[Proposition 1.6.1]{CembranosMendoza} for instance.
	
	In order to prove that $n_{\widetilde{G}}\big(L_1(\mu,X),L_1(\mu,Y)\big)\geq n_G(X,Y)$, we need to introduce some notation. If $(\Omega,\Sigma,\mu)$ is a probability space, we write $\Sigma^+:=\left\{B\in\Sigma \colon \mu(B)>0 \right\}$. Given $X$ and $Y$ Banach spaces, the set
	\begin{equation*}
	\mathcal{B}:=\left\{\sum_{B\in\pi}y_B^\ast \,\eins_B \colon \pi\subseteq\Sigma^+ \textnormal{ finite partition of } \Omega, \ y_B^\ast \in S_{Y^\ast } \right\}\subseteq S_{L_\infty(\mu,Y^\ast )}
	\end{equation*}
	satisfies that
	\begin{equation}\label{L1:eq1}
	B_{L_1(\mu,Y)^\ast }=\overline{\conv}^{w^\ast }(\mathcal{B}),
	\end{equation}
	since $\T\mathcal{B}=\mathcal{B}$ and it is clearly norming for the simple functions of $L_1(\mu,Y)$.
	On the other hand, we will write
	\begin{equation*}
	\mathcal{A}:=\left\{x\,\frac{\eins_A}{\mu(A)} \colon x\in S_X, \, A\in\Sigma^+ \right\},
	\end{equation*}
	which satisfies that
	\begin{equation} \label{L1:eq2}
	B_{L_1(\mu,X)}=\overline{\conv}(\mathcal{A}).
	\end{equation}
	Indeed, it is enough to notice that every simple function $f\in S_{L_1(\mu,X)}$ belongs to the convex hull of $\mathcal{A}$: such an $f$ can be written as
	$ f=\sum_{A\in\pi} x_A\eins_A$, where $\pi\subseteq\Sigma^+$ is a finite family of pairwise disjoint sets of $\Omega$ and $x_A\in X\backslash \{0\}$ for each $A\in\pi$. Then
	$$ \|f\|=\sum_{A\in\pi} \|x_A\|\mu(A)=1,$$
	and hence
	$$ f=\sum_{A\in\pi}\|x_A\|\mu(A) \dfrac{x_A}{\|x_A\|} \dfrac{\eins_A}{\mu(A)} \in \conv(\mathcal{A}).$$
	Now, fix $T\in \mathcal{L}\big(L_1(\mu,X),L_1(\mu,Y)\big)$ with $\|T\|=1$ and $\eps>0$, we may find by \eqref{L1:eq2} elements $x_0\in S_X$ and $A\in\Sigma^+$ such that
	$$ \left\|T\left(x_0\dfrac{\eins_A}{\mu(A)}\right) \right\|>1-\eps. $$
	Using \eqref{L1:eq1}, there exists $f^\ast =\sum_{B\in\pi} y_B^\ast \,\eins_B$, where $\pi$ is a finite partition of $\Omega$ into sets of $\Sigma^+$ and $y_B^\ast \in S_{Y^\ast }$ for each $B\in\pi$, satisfying that
	\begin{equation}\label{L1:eq3}
	\re f^\ast \left(T\left(x_0\dfrac{\eins_A}{\mu(A)}\right)\right)=\re \sum_{B\in\pi} y_B^\ast \left(\int_B T\left(x_0 \dfrac{\eins_A}{\mu(A)}\right)\, d\mu \,\right)>1-\eps.
	\end{equation}
	Then, we can write
	$$ T\left(x_0\dfrac{\eins_A}{\mu(A)}\right)=\sum_{\substack{B\in\pi\\\mu(A\cap B)\neq 0}}\dfrac{\mu(A\cap B)}{\mu(A)}T\left(x_0\dfrac{\eins_{A\cap B}}{\mu(A\cap B)}\right) $$
	so, by a standard convexity argument, we can assume that there is $B_0\in\pi$ such that, if we take the set $A\cap B_0$ in the role of new $A$, the inequality \eqref{L1:eq3} remains true. After this modification of $A$, we obtain additionally that $A\subseteq B_0$. By the density of norm-attaining functionals, we can and do assume that every $y_B^\ast $ is norm-attaining, so there is $y_{B_0}\in S_Y$ such that $y_{B_0}^\ast (y_{B_0})=1$. Define the operator $S\colon X \longrightarrow Y$ by
	$$ S(x)=\int_{B_0} T\left(x\,\dfrac{\eins_A}{\mu(A)}\right)d\mu +\left[\sum_{B\in\pi\backslash \{B_0\}} y_B^\ast \left(\int_B T\left(x\,\dfrac{\eins_A}{\mu(A)}\right)d\mu\right)\right]y_{B_0} \quad (x\in X).$$
	It is easy to check that $\|S\|\leq1$, and moreover $\|S\|>1-\eps$ since, as a consequence of \eqref{L1:eq3}, we obtain that
	$$ \|S(x_0)\|\geq \left|y_{B_0}^\ast (Sx_0)\right|=\left|f^\ast \left(T\left(x_0\,\dfrac{\eins_A}{\mu(A)}\right)\right)\right|> 1-\eps. $$
	Now, fixed $\delta>0$, we consider $x\in S_X$ and $y^\ast \in S_{Y^\ast }$ with $\re y^\ast (Gx)>1-\delta$. Take $f\in \mathcal{A}$ defined by
	$$f=x\,\dfrac{\eins_A}{\mu(A)}$$
	and $g^\ast \in\mathcal{B}$ by
	$$ g^\ast (h)=y^\ast \left(\int_{B_0}h\,d\mu\,\right)+\sum_{B\in\pi\backslash \{B_0\}} y_B^\ast \left(\int_B h\, d\mu\,\right)y^\ast (y_{B_0}) \qquad \big(h\in L_1(\mu,Y)\big). $$
	We have that
	$$\widetilde{G}f=\widetilde{G}\left(x\,\dfrac{\eins_A}{\mu(A)}\right)=G(x)\,\dfrac{\eins_A}{\mu(A)},$$
	and, recalling that $A\subseteq B_0$ and the partition is a family of pairwise disjoint sets, we deduce that
	\begin{align*}
	\re g^\ast (\widetilde{G}f) &= \re \left( y^\ast \left(\int_{B_0} G(x)\,\dfrac{\eins_A}{\mu(A)}\, d\mu\right)+\left[\sum_{B\in\pi\backslash \{B_0\}} y_B^\ast \left(\int_B G(x)\,\dfrac{\eins_A}{\mu(A)}\,d\mu\right)\right]y^\ast (y_{B_0})\right)\\&=\re y^\ast (Gx)>1-\delta.
	\end{align*}
	Moreover,
	\begin{align*}
	|y^\ast (Sx)| &=\left|y^\ast \left(\int_{B_0} T\left(x\,\dfrac{\eins_A}{\mu(A)}\right) d\mu\right) +\left[\sum_{B\in\pi\backslash \{B_0\}} y_B^\ast \left(\int_B T\left(x\,\dfrac{\eins_A}{\mu(A)}\right) d\mu\right)\right]y^\ast (y_{B_0})\,\right|\\ &=|g^\ast (Tf)|\leq v_{\widetilde{G},\delta}(T).
	\end{align*}
	It follows that $v_{G,\delta}(S)\leq v_{\widetilde{G},\delta}(T)$, hence
	$$
	v_{\widetilde{G}}(T)\geq v_G(S) \geq n_G(X,Y) \|S\| \geq (1-\eps)n_G(X,Y).
	$$
	Taking $\eps \downarrow 0$, we get that $v_{\widetilde{G}}(T)\geq n_G(X,Y)$ and, the arbitrariness of $T$ gives the desired inequality.
	
	The reverse inequality $ n_{\widetilde{G}}\big(L_1(\mu,X),L_1(\mu,Y)\big)\leq n_G(X,Y)$ follows from Corollary \ref{corollary:composition_inequality_ellp}.
\end{proof}

The last result on composition operators on vector-valued function spaces deals with spaces of essentially bounded vector-valued functions.

\begin{proposition}\label{Prop:Linf-general}
	Let $X,Y$ be Banach spaces, let $(\Omega,\Sigma,\mu)$ be a $\sigma$-finite measure space, and let $G\in \mathcal{L}(X,Y)$ be a norm-one operator. Consider the norm-one composition operator $\widetilde{G} \colon L_\infty(\mu,X)\longrightarrow L_\infty(\mu,Y)$ given by
	$ \widetilde{G}(f)=G\circ f$ for every $f\in L_\infty(\mu,X)$. Then
	$$ n_{\widetilde{G}}\big(L_\infty(\mu,X),L_\infty(\mu,Y)\big)=n_G(X,Y). $$
\end{proposition}

The proof of this result borrows the ideas in \cite[Theorem~2.3]{M-V}. We also borrow from \cite{M-V} two preliminary lemmas that we state for convenience of the reader.

\begin{lemma}[\mbox{\cite[Lemma~2.1]{M-V}}]\label{lemma:Linf:lema1}
	Let $f\in L_\infty(\mu,X)$ with $\|f(t)\|>\lambda$ a.e. Then there exists $B\in\Sigma$ with $0<\mu(B)<\infty$ such that
	$$ \left\|\dfrac{1}{\mu(B)}\int_B f(t) \, d\mu(t) \, \right\|>\lambda. $$
\end{lemma}

\begin{lemma}[\mbox{\cite[Lemma~2.2]{M-V}}]\label{lemma:Linf:lema2}
	Let $f\in L_\infty(\mu,X)$, $C\in\Sigma$ with positive measure, and $\eps>0$. Then there exist $x\in X$ and $A\subseteq C$ with $0<\mu(A)<\infty$ such that $\|x\|=\|f\,\eins_C\|$ and $\|(f-x)\,\eins_A\|<\eps$. Accordingly, the set
	$$ \left\{x\,\eins_A+f\,\eins_{\Omega\backslash A} \colon x\in S_X, \, f\in B_{L_\infty(\mu,X)}, \, A\in\Sigma \text{ con } 0<\mu(A)<\infty \right\} $$
	is dense in $S_{L_\infty(\mu,X)}$.
\end{lemma}

\begin{proof}[Proof of Proposition~\ref{Prop:Linf-general}]
	In order to show that $n_{\widetilde{G}}\big(L_\infty(\mu,X),L_\infty(\mu,Y)\big)\geq n_G(X,Y),$ we fix an operator $T\in \mathcal{L}\big(L_\infty(\mu,X),L_\infty(\mu,Y) \big)$ with $\|T\|=1$. Given $\eps>0$, we may find $f_0\in S_{L_\infty(\mu,X)}$ and $C\subseteq \Omega$ with $\mu(C)>0$ such that
	\begin{equation}\label{Linf:eq1}
	\big\|[Tf_0](t)\big\|>1-\eps \qquad (t\in C).
	\end{equation}
	On account of Lemma \ref{lemma:Linf:lema2}, there exist $y_0\in B_X$ and $A\subseteq C$ with $0<\mu(A)<\infty$ such that $\|(f_0-y_0) \, \eins_A\|<\eps$. Now, write
	$y_0=(1-\lambda)x_1+\lambda x_2$ with $0\leq\lambda\leq 1$, $x_1,x_2\in S_X$, and consider the functions
	$$f_j=x_j \, \eins_A+f_0\, \eins_{\Omega\backslash A} \in L_\infty(\mu,X) \qquad (j=1,2).$$
	which clearly satisfy $\|f_0-((1-\lambda)f_1+\lambda f_2)\|<\eps$. Since $A\subseteq C$, by using \eqref{Linf:eq1}, we have
	$$\big\|[T
	f_1](t)\big\|>1-2\eps \qquad \text{or} \qquad \big\|[T
	f_2](t)\big\|>1-2\eps$$
	for every $t\in A$. Now, we choose $i\in\{1,2\}$ such that
	$$A_i=\big\{t\in A \colon \big\|[Tf_i](t)\big\|>1-2\eps\big\}$$
	has positive measure, we write $x_0=x_i$, and we finally use Lemma~\ref{lemma:Linf:lema1} to get $B\subseteq A_i\subseteq A$ with $0<\mu(B)<\infty$ such that
	\begin{equation}\label{Linf:eq2}
	\left\|\dfrac{1}{\mu(B)}\int_B T\left(x_0\,\eins_A+f_0\,\eins_{\Omega\backslash A}\right)\, d\mu \, \right\|>1-2\eps.
	\end{equation}
	Next, we fix $x_0^\ast \in S_{X^\ast }$ with
	$x_0^\ast (x_0)=1$, we write
	$$\Phi(x)=x\,\eins_A+x_0^\ast (x)\,f_0 \,\eins_{\Omega\backslash A} \in L_\infty(\mu,X) \qquad (x\in X).$$
	and we define the operator $S\in \mathcal{L}(X,Y)$ given by
	\begin{equation*}
	S x =\dfrac{1}{\mu(B)}\int_B T\big(\Phi(x)\big) \, d\mu
	\qquad (x\in X)
	\end{equation*}
	which, by \eqref{Linf:eq2}, satisfies $ \|S\|\geq\|Sx_0\|>1-2\eps$.
	
	Given $\delta>0$, we fix $x\in S_X$ and $y^\ast \in S_{Y^\ast }$ with $\re y^\ast (Gx)>1-\delta$. Define $f\in S_{L_\infty(\mu,X)}$ by
	$$ f=\Phi(x)=x\,\eins_A+x_0^\ast (x)\,f_0\,\eins_{\Omega\backslash A}$$
	and $g^\ast \in S_{L_\infty(\mu,Y)^\ast }$ by $$g^\ast (h)=y^\ast \left(\dfrac{1}{\mu(B)}\int_Bh \,d\mu \,\right) \qquad \big(h\in L_\infty(\mu,Y)\big).$$
	Since $B\subseteq A$, we have
	\begin{align*}
	\re g^\ast (\widetilde{G}f)&=\re y^\ast \left(\dfrac{1}{\mu(B)}\int_B G\big(f(t)\big) \,d\mu(t) \,\right) \\
	&= \re y^\ast \left(\dfrac{1}{\mu(B)}\int_B G\left(x\,\eins_A(t)+x_0^\ast (x)\,f_0(t)\,\eins_{\Omega\backslash A}(t)\right) \,d\mu(t) \,\right) \\
	&= \re y^\ast \left(\dfrac{1}{\mu(B)}\int_B G(x) \eins_B(t) \,d\mu(t) \,\right)=\re y^\ast (G x)>1-\delta.
	\end{align*}
	Moreover,
	\begin{align*}
	|y^\ast (Sx)|= \left|y^\ast \left(\dfrac{1}{\mu(B)}\int_B T\big(\Phi(x)\big) \, d\mu\,\right)\right|=|g^\ast (T f)|\leq v_{\widetilde{G},\delta}(T)
	\end{align*}
	and it follows that $v_{G,\delta}(S)\leq v_{\widetilde{G},\delta}(T)$, hence
	$$v_{\widetilde{G}}(T)\geq v_G(S) \geq n_G(X,Y) \|S\| \geq (1-2\eps)n_G(X,Y).$$
	Taking $\eps \downarrow 0$, we get that $v_{\widetilde{G}}(T)\geq n_G(X,Y)$ and, the arbitrariness of $T$ gives the desired inequality.
	
	The reverse inequality is consequence of Corollary \ref{corollary:composition_inequality_ellp}.
\end{proof}

\subsection{Adjoint operators}
As shown in Lemma \ref{lemma:adjoint-inequality}, the numerical index with respect to an operator always dominates the numerical index with respect to its adjoint. Our aim here is to give some particular cases in which both indices coincide. First, we have to recall that such an equality does not always hold, as there are Banach spaces $X$ for which $n(X^\ast)<n(X)$, see \cite[\S2]{KaMaPa} for instance. We also may provide with an easier example which does not use the identity operator.

\begin{example}\label{exam:adjoint}
{\slshape The inclusion $G\colon c_0 \longrightarrow c$ satisfies that $n_G(c_0,c)=1$, whereas its adjoint $G^\ast \colon \ell_1\oplus_1 \K\longrightarrow\ell_1$, given by $(x,\lambda)\longmapsto x$, is not even a vertex of $\mathcal{L}(c^\ast,c_0^\ast)$ and so, it satisfies that $n_{G^\ast }(c^\ast ,c_0^\ast )=0$.}

Indeed, $G$ is a spear operator by, for instance, \cite[Proposition 4.2]{SpearsBook}, so $n_G(c_0,c)=1$. To prove that $G^\ast$ is not a vertex,  	
consider the operator $T\colon \ell_1\oplus_1\K\longrightarrow \ell_1$ given by $T(x,\lambda)=\lambda e_1^\ast $ for $x\in \ell_1$ and $\lambda\in \K$. Then, we have
\begin{align*} \|G^\ast (x,\lambda)+\theta T(x,\lambda)\|&=\left\|\big(x(1)+\theta\lambda\big)e_1^\ast +\sum_{k=2}^{\infty}x(k)e_k^\ast \right\|
\\&=|x(1)|+|\lambda|+\sum_{k=2}^{\infty}|x(k)|=\|x\|+|\lambda|=\|(x,\lambda)\|
\end{align*}
for every $\theta\in \T$, every $x\in \ell_1$, and every $\lambda\in\K$. This shows that $\|G^\ast +\theta T\|\leq 1$ and so $G^\ast $ is not an extreme operator. Therefore, $G^\ast$ is not a vertex by Lemma \ref{Lemm:unotsmoothnotextreme}.
\end{example}

If $X$ and $Y$ are both reflexive spaces, it is clear that the numerical index with respect to every norm-one operator $G\in\mathcal{L}(X,Y)$ coincides with the numerical index with respect to $G^\ast$. Indeed, the inequality
$$
n_{G^{\ast\ast}}(X^{\ast\ast},Y^{\ast\ast})\leq n_{G^{\ast}}(Y^{\ast},X^{\ast}) \leq n_{G}(X,Y)
$$
gives the result. Actually, it is enough that $Y$ is reflexive, or even a much weaker hypothesis: we show that the numerical index with respect to an operator coincides with the one with respect to its adjoint when the range space is $L$-embedded. Recall that a Banach space $Y$ is \emph{L-embedded} if $Y^{\ast \ast }=J_Y(Y)\oplus_1 Y_s$ for suitable closed subspace $Y_s$ of $Y^{\ast \ast }$. We refer to the monograph \cite{HWW} for background. Examples of $L$-embedded spaces are reflexive spaces (trivial), predual of von Neumann algebras so, in particular, $L_1(\mu)$ spaces, the Lorentz spaces $d(w,1)$ and $L^{p,1}$, the Hardy space $H_0^1$, the dual of the disk algebra $A(\mathbb{D})$, among others (see \cite[Examples IV.1.1 and III.1.4]{HWW}).

\begin{proposition}
	Let $X$ be a Banach space, let $Y$ be an $L$-embedded space, and let $G\in \mathcal{L}(X,Y)$ be a norm-one operator. Then $n_G(X,Y)=n_{G^\ast }(Y^\ast ,X^\ast )$.
\end{proposition}

\begin{proof}
	We follow the lines of \cite[Proposition 5.21]{SpearsBook}. Write $Y^{\ast \ast }=J_Y(Y)\oplus_1 Y_s$ and let $P_Y\colon Y^{\ast \ast }\longrightarrow J_Y(Y)$ be the natural projection. For a fixed $T\in \mathcal{L}(Y^\ast ,X^\ast )$ consider operators
	$$ A:=P_Y\circ T^\ast  \circ J_X\colon X\longrightarrow J_Y(Y), \qquad B:=[\Id-P_Y]\circ T^\ast  \circ J_X\colon X\longrightarrow Y_s$$
	and observe that $T^\ast \circ J_X =A\oplus B$. Given $\eps>0$, since $J_X(B_X)$ is dense in $B_{X^{\ast \ast }}$ by Goldstine's Theorem and $T^\ast $ is weak-star to weak-star continuous, we may find $x_0\in S_X$ such that
	$$ \|T^\ast J_X(x_0)\|=\|Ax_0\|+\|Bx_0\|>\|T^\ast \|-\eps. $$
	Now, we may find $y_0\in S_Y$ and $y_s^\ast \in S_{Y_s^\ast }$ such that
	$$ \|Ax_0\|y_0=Ax_0 \qquad \textnormal{and} \qquad y_s^\ast (Bx_0)=\|Bx_0\|.$$
	We define $S\colon X\longrightarrow Y$ by
	$$ S(x)=Ax+y_s^\ast (Bx)y_0 \qquad (x\in X),$$
	and observe that
	$$\|S\|\geq\|Sx_0\|=\big\|A x_0+y_s^\ast (Bx_0)y_0\big\|=\|Ax_0\|+\|Bx_0\|>\|T^\ast \|-\eps.$$
	Given $\delta>0$, we take $x\in S_X$ and $y^\ast \in S_{Y^\ast }$ with $\re y^\ast (Gx)>1-\delta$, and consider $$z=J_X(x)\in S_{X^{\ast \ast }} \qquad \text{and} \qquad z^\ast =(J_{Y^\ast }(y^\ast ),y^\ast (y_0)y_s^\ast )\in S_{Y^{\ast \ast \ast }}$$ as $Y^{\ast \ast \ast }=J_{Y^\ast }(Y^\ast )\oplus_\infty Y_s^\ast $. Since $G^{\ast \ast }\circ J_X=J_Y\circ G$, it is clear that $\re z^\ast (G^{\ast \ast }z)=\re y^\ast (Gx)>1-\delta$.
	
	Moreover,
	$$|z^\ast (T^\ast z)|=\big|J_{Y^\ast }(y^\ast )(Ax+y^\ast (y_0)y_s^\ast (Bx))\big|=|y^\ast (Sx)|,$$
	hence $|y^\ast (Sx)|=|z^\ast (T^\ast z)|\leq v_{G^{\ast \ast },\delta}(T^\ast )$ and, taking supremum, $v_{G,\delta}(S)\leq v_{G^{\ast \ast },\delta}(T^\ast )$. Therefore,
	$$ v_{G^\ast }(T)=v_{G^{\ast \ast }}(T^\ast )\geq v_G(S)\geq n_G(X,Y)\|S\|>n_G(X,Y)\big[\|T\|-\eps\big].$$
	The arbitrariness of $\eps>0$ and of $T\in \mathcal{L}(Y^\ast ,X^\ast )$ give that $n_G(X,Y)\leq n_{G^\ast }(Y^\ast ,X^\ast )$, and the other inequality is always true.
\end{proof}

Particular cases of the above result are the following.

\begin{corollary}
Let $X$ be a Banach space and let $Y$ be a reflexive space. Then, $n_G(X,Y)=n_{G^\ast }(Y^\ast ,X^\ast )$ for every norm-one $G\in \mathcal{L}(X,Y)$.
\end{corollary}

\begin{corollary}
Let $X$ be a Banach space and let $\mu$ be a positive measure. Then, $n_G(X,L_1(\mu))=n_{G^\ast }(L_1(\mu)^\ast ,X^\ast )$ for every norm-one $G\in \mathcal{L}(X,L_1(\mu))$.
\end{corollary}

Finally, we show that, for rank-one operators, the numerical index is preserved by passing to the adjoint.

\begin{proposition}
	Let $X, Y$ be Banach spaces, and let $G\in \mathcal{L}(X,Y)$ be a rank-one operator of norm 1. Then $n_G(X,Y)=n_{G^\ast }(Y^\ast ,X^\ast )$ and so, the same happens to all the successive adjoints of $G$.
\end{proposition}

\begin{proof}
	We can write $G=x_0^\ast \otimes y_0$ for some $x_0^\ast \in S_{X^\ast }$ and $y_0\in S_Y$, so $n_G(X,Y)=n(X^\ast ,x_0^\ast )n(Y,y_0)$ by Proposition~\ref{prop:num-index-rankone-operators}. Besides, as $G^\ast =J_Y(y_0)\otimes x_0^\ast $, we have $n_{G^\ast }(Y^\ast ,X^\ast )=n(Y^{\ast \ast },J_Y(y_0))n(X^\ast ,x_0^\ast )$ again by Proposition~\ref{prop:num-index-rankone-operators}. But $n(Y^{\ast \ast },J_Y(y_0))=n(Y,y_0)$ by Lemma \ref{Lemm:bidual-abstract} and we are done.
\end{proof}

\subsection{Composition of operators}

The next result allows us to control the numerical index with respect to the composition of two operators, in two particular cases.

\begin{lemma}\label{lemma:num-index-composition}
Let $X, Y, Z$ be Banach spaces and let $G_1\in \mathcal{L}(X,Y)$ and $G_2\in \mathcal{L}(Y,Z)$ be norm-one operators.
\begin{enumerate}
\item[(a)] If $G_2$ is an isometric embedding, then $n_{G_2\circ G_1}(X,Z)\leq n_{G_1}(X,Y)$.
\item[(b)] If $\overline{G_1(B_X)}=B_Y$, then $n_{G_2\circ G_1}(X,Z)\leq n_{G_2}(Y,Z)$.	
\end{enumerate}		
\end{lemma}	
	
\begin{proof}
Both (a) and (b) follow from Lemma \ref{Lemm:numericalIndexThroughMaps}. In the first case, it is enough to see that the map $T\longmapsto G_2\circ T$ from $\mathcal{L}(X,Y)$ to $\mathcal{L}(X,Z)$ is an isometric embedding by the hypothesis on $G_2$. For (b), we have that $S\longmapsto S\circ G_1$ from $\mathcal{L}(Y,Z)$ to $\mathcal{L}(X,Z)$ is an isometric embedding by the hypothesis on $G_1$.
\end{proof}	

We now would like to collect some consequences of this result.

The first immediate consequence is that the restriction of the codomain of an operator enlarges the numerical index.

\begin{proposition}
	Let $X, Y$ be Banach spaces, let $G\in \mathcal{L}(X,Y)$ be a norm-one operator, and let $Z$ be a closed subspace of $Y$ such that $G(X)\subseteq Z$. Consider the operator $\overline{G}\colon X\longrightarrow Z$ given by $\overline{G}x=Gx$ for every $x\in X$. Then $ n_G(X,Y)\leq n_{\overline{G}}\big(X,Z).$
\end{proposition}

\begin{proof}
	It follows from Lemma \ref{lemma:num-index-composition}.a as $G=I\circ \overline{G}$ where $I:Z\longrightarrow Y$ denotes the inclusion.  	
\end{proof}	

The inequality in the above result can be strict:

\begin{example}\label{exam:restrict-codomain}
{\slshape The operator $G:\K \longrightarrow \K\oplus_\infty \K$ given by $G(x)=(x,0)$ satisfies $n_G(\K,\K\oplus_\infty\K)=0$, whereas $\overline{G}: \K\longrightarrow \K$ satisfies $n_{\overline{G}}(\K,\K)=1$.}
\end{example}

Another consequence of Lemma \ref{lemma:num-index-composition} is that the numerical index with respect to the injectivization of an operator  is an upper bound for the numerical index with respect to the original operator.

\begin{proposition}\label{prop:niwithrespecttoquotient}
	Let $X$, $Y$ be Banach spaces, let $G\in \mathcal{L}(X,Y)$ be a norm-one operator, and let $q \colon X\longrightarrow X/\ker G$ be the quotient map. Consider the injectivization $\widehat{G}\in \mathcal{L}\big(X/\ker G,Y\big)$ satisfying $\widehat{G}\circ q = G$. Then,
	$$ n_G(X,Y)\leq n_{\widehat{G}}\big(X/\ker G,Y\big). $$
\end{proposition}

\begin{proof}
	It follows from Lemma \ref{lemma:num-index-composition}.b as $\widehat{G}\circ q = G$ and $\overline{q(B_X)}=B_{X/\ker G}$.
\end{proof}

In the particular case when $n_G(X,Y)=1$, we obtain the following result which gives a partial answer to Problem 9.14 of \cite{SpearsBook}.

\begin{corollary}\label{coro:SPEARquotient}
  Let $X$, $Y$ be Banach spaces, let $G\in \mathcal{L}(X,Y)$ be a norm-one operator. Then, under the notation of Proposition \ref{prop:niwithrespecttoquotient}, if $G$ is a spear operator, then so is its injectivization $\widehat{G}$.
\end{corollary}

Again, the inequality in Proposition \ref{prop:niwithrespecttoquotient} may be strict, as the following example shows. It also proves that Corollary \ref{coro:SPEARquotient} is not an equivalence.

\begin{example}
	{\slshape The operator $G\colon \ell_1\oplus_1\K\longrightarrow\ell_1$ given by $G(x,\lambda)=x$ satisfies $n_G(\ell_1\oplus_1\K,\ell_1)=0$ (as it has been proved in example \ref{exam:adjoint}), whereas the injectivization $\widehat{G}$ is the identity operator in $\ell_1$ and so satisfies that $n_{\widehat{G}}(\ell_1,\ell_1)=n(\ell_1)=1$.}
\end{example}

With the help of all of these examples and some others form previous sections, we may prove the following assertion.

\begin{remark}
There is no general function $\Upsilon:[0,1]\times [0,1]\longrightarrow [0,1]$ such that the equality
$$
n_{G_2\circ G_1}(X,Z)=\Upsilon\bigl(n_{G_2}(Y,Z),n_{G_1}(X,Y) \bigr)
$$
holds for every Banach spaces $X, Y, Z$ and every norm-one operators $G_1\in \mathcal{L}(X,Y)$ and $G_2\in \mathcal{L}(Y,Z)$.
\end{remark}

Indeed, suppose that such a function $\Upsilon$ exists. In Remark \ref{remark:smoothnessrotundityneeded_E-sum} it is given an example of a real Banach space $Z$ with $n(Z)=0$ and a norm-one operator $G\in \mathcal{L}(Z,\R)$ such that $n_{G}(Z,\R)=1$. As $G=G\circ \Id_Z$, it follows that $1=\Upsilon(1,0)$. Besides, there is a similar example in the same Remark showing that $1=\Upsilon(0,1)$. On the other hand, if $X$, $Y$ are two-dimensional Banach spaces, we may always find $G\in \mathcal{L}(X,Y)$ with $n_{G}(X,Y)=0$ by Proposition \ref{prop:ifzeronotinsidedimensiononeOperators}. As $G=G\circ \Id_X=\Id_Y\circ G$, it follows that $0=\Upsilon(0,n(X))=\Upsilon(n(Y),0)$. It is enough to consider $X=Y=\ell_\infty^2$ to get a contradiction.

Now, we may wonder whether a further relationship with the composition is valid in general. We answer this question in the negative giving some counterexamples.	

Example \ref{exam:n-zero-nG-one} shows that, in general, it is not true that $n_{G_2\circ G_1}(X,Z)\leq \max\left\{n_{G_1}(X,Y),n_{G_2}(Y,Z)\right\}$, with $G$ playing the role of $G_1$ and the identity operator playing the one of $G_2$.

Example \ref{exam:restrict-codomain} also shows that, in general, it is not true that $n_{G_2\circ G_1}(X,Z)\geq \max\left\{n_{G_1}(X,Y),n_{G_2}(Y,Z)\right\}$.
Actually, it is possible that the inequality $n_{G_2\circ G_1}(X,Z)\geq \min\left\{n_{G_1}(X,Y),n_{G_2}(Y,Z)\right\}$ fails, as the following example shows, since $n(\ell_p)>0$ for $p\neq 2$ by \cite{MarMerPop}.

\begin{example}
	{\slshape Let $1\leq p<q<\infty$. The inclusion $G\colon \ell_p\rightarrow\ell_q$ satisfies that $n_G(\ell_p,\ell_q)=0$.}
\end{example}

\begin{proof}
	Consider the norm-one operator $T\in {\mathcal{L}(\ell_p,\ell_q)}$ given by $T=e_2^\ast \otimes e_1$. Fixed $0<\eps<1/4$, our goal is to prove that
	$v_{G}(T)\leq \max\left\{\eps^{1/p}, \left(1-(1-2\eps)^q\right)^{1/q}\right\}$. To do so we need the following claim.

\noindent\emph{Claim:} Let $0<\delta<1/2$ be such that $(1-\delta)^{\frac{p}{q-p}}>1-\eps$. Given $x\in S_{\ell_p}$ such that $\|x\|_q>(1-2\delta^2)^{1/q}$, there exists a unique $k_0\in \N$ satisfying that $|x(k_0)|^p>1-\eps$.

Indeed, the uniqueness of $k_0$ is clear from the facts $|x(k_0)|^p>1-\eps$, $\eps<1/4$, and $\|x\|_p=1$. Let us show the existence of $k_0$. Since $1-2\delta^2<\|x\|_q^q$, there is $n\in \N$ satisfying that $1-\delta^2<\sum_{k=1}^n|x(k)|^q$, and thus
	$$
	\sum_{k=1}^n|x(k)|^p-\delta^2\leq1-\delta^2<\sum_{k=1}^n|x(k)|^q=\sum_{k=1}^n|x(k)|^p|x(k)|^{q-p}.
	$$
	Let $I=\{k\in \{1,\ldots,n\} \colon |x(k)|^{q-p}>1-\delta\}$. Using \cite[Lemma~8.14]{SpearsBook} with $\lambda_k=|x(k)|^p$, $\beta_k=1$ and $\alpha_k=|x(k)|^{q-p}$, we get that $\sum_{k\notin I} |x(k)|^p<\delta$. So we can write
	$$
	1-\delta^2<\sum^n_{k\notin I} |x(k)|^q+\sum^n_{k\in I} |x(k)|^q\leq \sum^n_{k\notin I} |x(k)|^p+\sum^n_{k\in I} |x(k)|^q<\delta+\sum^n_{k\in I} |x(k)|^q
	$$
	which gives $\sum_{k\in I} |x(k)|^q>1-\delta^2-\delta>0$ and, therefore, $I\neq \emptyset$. For $k_0\in I$, we have $$|x(k_0)|^p>(1-\delta)^{\frac{p}{q-p}}>1-\eps$$ finishing the proof of the claim.
	
	To estimate the numerical radius of $T$, let $0<\widetilde{\delta}<\eps$ be such that $1-\widetilde{\delta}>(1-2\delta^2)^{1/q}$ and take $x\in S_{\ell_p}$ and $y^\ast \in S_{\ell^\ast_q}$ satisfying $\re y^\ast (x)>1-\widetilde{\delta}$, which implies that
	$$
	\|x\|_q>\re y^\ast (x)>1-\widetilde{\delta}>(1-2\delta^2)^{1/q}.
	$$
	The claim tells us that there is $k_0\in\N$ such that $|x(k_0)|^p>1-\eps$ and so $\sum_{k\neq k_0}^{\infty}|x(k)|^p<\eps$. Now, we can estimate $|y^\ast (Tx)|=|y^\ast (1)||x(2)|$ depending on the value of $k_0$. If $k_0\neq 2$ then $|x(2)|<\eps^{1/p}$ and $|y^\ast (Tx)|\leq |x(2)|<\eps^{1/p}$. Suppose, otherwise, that $k_0=2$. Then, as
	\begin{align*}
	1-\widetilde{\delta}<\re y^\ast (x)=|y^\ast (2)||x(2)|+\sum_{k\neq 2}^{\infty} |y^\ast (k)||x(k)|\leq |y^\ast (2)|+\|y^\ast \|_q \sum_{k\neq 2}^{\infty}|x(k)|^p\leq |y^\ast (2)|+\eps,
	\end{align*}
	we get $|y^\ast (2)|>1-\widetilde{\delta}-\eps>1-2\eps$. Therefore, we have
	$$
	|y^\ast (2)|^q>(1-2\eps)^q \qquad \text{and} \qquad |y^\ast (Tx)|\leq|y^\ast (1)|<\left(1-(1-2\eps)^q\right)^{1/q}.
	$$
	Hence, in any case, we get
	$$
	v_{G}(T)\leq v_{G,\widetilde{\delta}}(T)\leq \max\left\{\eps^{1/p}, \left(1-(1-2\eps)^q\right)^{1/q}\right\}
	$$
	and the arbitrariness of $\eps$ gives $v_G(T)=0$ and so, $n_G(\ell_p,\ell_q)=0$.
\end{proof}	

\subsection{Extending the domain and the codomain}
Our final aim in this section is to study ways of extending the domain and the codomain of an operator maintaining the same numerical index. For the domain, we have the following result.

\begin{proposition}
	Let $X, Y, Z$ be Banach spaces, let $G\in \mathcal{L}(X,Y)$ be a norm-one operator, and consider the norm-one operator $\widetilde{G}\colon X\oplus_\infty Z\longrightarrow Y$ given by $\widetilde{G}(x,z)=G(x)$ for every $(x,z)\in X\oplus_\infty Z$. Then
	$$
	n_{\widetilde{G}}\big(X\oplus_\infty Z,Y\big)=n_G(X,Y).
	$$
\end{proposition}

\begin{proof}
	Fix $T\in \mathcal{L}\big(X\oplus_\infty Z, Y\big)$ with $\|T\|>0$ and $0<\eps<\|T\|$. We may find $x_0\in S_X$ and $z_0\in S_Z$ satisfying $\|T(x_0,z_0)\|>\|T\|-\eps$. Now take $x_0^\ast \in S_{X^\ast }$ with $x_0^\ast (x_0)=1$ and define the operator $S\in \mathcal{L}(X,Y)$ by
	$$
	S(x)=T\big(x,x_0^\ast (x)z_0\big) \qquad (x\in X)
	$$
	which satisfies $\|S\|\geq\|Sx_0\|=\|T(x_0,z_0)\|>\|T\|-\eps$.
	
	Now, given $\delta>0$, $x\in S_X$, and $y^\ast \in S_{Y^\ast }$ with $\re y^\ast (Gx)>1-\delta$, we consider $\big(x,x_0^\ast (x)z_0\big)\in S_{X\oplus_\infty Z}$. It is clear that $\re y^\ast \big(\widetilde{G}\big(x,x_0^\ast (x)z_0\big)\big)=\re y^\ast (Gx)>1-\delta$. Moreover,
	$$
	|y^\ast (Sx)|=|y^\ast (T(x,x_0^\ast (x)z_0))|\leq v_{\widetilde{G},\delta}(T),
	$$
	hence $v_{G,\delta}(S)\leq v_{\widetilde{G},\delta}(T)$. Therefore,
	$$
	v_{\widetilde{G}}(T)\geq v_G(S)\geq n_G(X,Y)\|S\|>n_G(X,Y)\big[\|T\|-\eps\big].
	$$
	The arbitrariness of $\eps>0$ and $T\in \mathcal{L}\big(X\oplus_\infty Z, Y\big)$ gives $n_{\widetilde{G}}(X\oplus_\infty Z, Y)\geq n_G(X,Y)$.
	
The reverse inequality follows immediately from Lemma \ref{lemma:num-index-composition}.b as $\widetilde{G}=G\circ P$ where $P:X\oplus_\infty Z\longrightarrow X$ denotes the natural projection.
\end{proof}

For the range space, the result is the following.

\begin{proposition}
	Let $X, Y, Z$ be Banach spaces, let $G\in \mathcal{L}(X,Y)$ be a norm-one operator, and consider the norm-one operator $\widetilde{G}\colon X\longrightarrow Y\oplus_1 Z$ given by $\widetilde{G}x=\big(Gx,0\big)$ for every $x\in X$. Then,
	$$
	n_{\widetilde{G}}\big(X, Y\oplus_1 Z\big)=n_G(X,Y).
	$$
\end{proposition}

\begin{proof}
	Fix $T\in \mathcal{L}\big(X,Y\oplus_1 Z\big)$ with $\|T\|>0$, $\|T\|>\eps>0$, and $x_0\in S_X$ such that $\|Tx_0\|>\|T\|-\eps$. Denote by $P_Y$ and $P_Z$ the projections from $Y\oplus_1 Z$ to $Y$ and $Z$, respectively. Take $y_0\in S_Y$ so that $P_YTx_0=\|P_Y Tx_0\|y_0$ and $z_0^\ast \in S_{Z^\ast }$ satisfying $z_0^\ast (P_ZTx_0)=\|P_ZTx_0\|$. Now define $S\in \mathcal{L}(X,Y)$ by
	$$
	Sx=P_YTx+z_0^\ast (P_ZTx)y_0 \qquad(x\in X)
	$$
	which satisfies
	$$
	\|S\|\geq\|Sx_0\|=\big\|P_YTx_0+\|P_ZTx_0\|y_0\big\|=\|P_YTx_0\|+\|P_ZTx_0\|>\|T\|-\eps.
	$$
	Given $\delta>0$, $x\in S_X$, and $y^\ast \in S_{Y^\ast }$ with $\re y^\ast (Gx)>1-\delta$, we consider $\big(y^\ast ,y^\ast (y_0)z_0^\ast \big)\in S_{(Y\oplus_1 Z)^\ast }$ as $(Y\oplus_1 Z)^\ast =Y^\ast \oplus_\infty Z^\ast $. It is clear that $\re\big(y^\ast ,y^\ast (y_0)z_0^\ast \big)(\widetilde{G}x)=\re y^\ast (Gx)>1-\delta$. Moreover,
	$$
	|y^\ast (Sx)|=|y^\ast (P_YTx+z_0^\ast (P_ZTx)y_0)|=|(y^\ast ,y^\ast (y_0)z_0^\ast \big)(Tx)|\leq v_{\widetilde{G},\delta}(T),
	$$
	and then $v_{G,\delta}(S)\leq v_{\delta,\widetilde{G}}(T)$. Therefore,
	$$
	v_{\widetilde{G}}(T)\geq v_G(S)\geq n_G(X,Y)\|S\|>n_G(X,Y)\big[\|T\|-\eps\big].
	$$
	The arbitrariness of $\eps$ and $T$ finishes gives $n_{\widetilde{G}}(X,Y\oplus_1 Z)\geq n_G(X,Y)$.
	
	The reverse inequality is an immediate consequence of Lemma \ref{lemma:num-index-composition}.a as $\widetilde{G}=I\circ G$ where $I:Y\longrightarrow Y\oplus_1 Z$ denotes the natural inclusion.
\end{proof}

\vspace*{0.5cm}

\noindent \textbf{Acknowledgment.} The authors are grateful to Rafael Pay\'{a} and \'{A}ngel Rodr\'{\i}guez Palacios for kindly answering several inquiries.

\end{document}